\documentclass[reqno,reqno]{amsart}
\usepackage[latin9]{inputenc}
\usepackage{xcolor}
\usepackage{mathtools}
\usepackage{url}
\usepackage{enumitem}
\usepackage{amstext}
\usepackage{amsthm}
\usepackage{amssymb}
\usepackage{graphicx}
\usepackage{esint}
\PassOptionsToPackage{normalem}{ulem}
\usepackage{ulem}
\usepackage[unicode=true,pdfusetitle,
 bookmarks=true,bookmarksnumbered=false,bookmarksopen=false,
 breaklinks=false,pdfborder={0 0 0},pdfborderstyle={},backref=false,colorlinks=true]
 {hyperref}

\makeatletter

\DeclareFontEncoding{LGR}{}{}
\DeclareRobustCommand{\greektext}{%
  \fontencoding{LGR}\selectfont\def\encodingdefault{LGR}}
\DeclareRobustCommand{\textgreek}[1]{\leavevmode{\greektext #1}}
\ProvideTextCommand{\~}{LGR}[1]{\char126#1}

\DeclareUnicodeCharacter{0227}{\~a}
\DeclareUnicodeCharacter{0225}{\`a}

\providecommand{\tabularnewline}{\\}
\providecolor{lyxadded}{rgb}{0,0,1}
\providecolor{lyxdeleted}{rgb}{1,0,0}

\DeclareRobustCommand{\lyxsout}[1]{\ifx\\#1\else\sout{#1}\fi}

\numberwithin{equation}{section}
\numberwithin{figure}{section}
\theoremstyle{plain}
\newtheorem{thm}{\protect\theoremname}
\theoremstyle{definition}
\newtheorem{defn}[thm]{\protect\definitionname}
\theoremstyle{plain}
\newtheorem{lem}[thm]{\protect\lemmaname}
\theoremstyle{definition}
\newtheorem{example}[thm]{\protect\examplename}
\theoremstyle{remark}
\newtheorem{rem}[thm]{\protect\remarkname}
\theoremstyle{plain}
\newtheorem{cor}[thm]{\protect\corollaryname}

\usepackage{pdfsync}
\usepackage{accents}
\usepackage{graphics}
\usepackage{epstopdf}
\usepackage[all]{xy}
\usepackage{amscd}
\usepackage{enumitem}
\usepackage{a4wide}
\usepackage{cite}
\setlist[enumerate]{leftmargin=*,label=(\roman*),align=left}

\newcommand{\xyR}[1]{ \makeatletter
\xydef@\xymatrixrowsep@{#1} \makeatother} 
\newcommand{\xyC}[1]{ \makeatletter
\xydef@\xymatrixcolsep@{#1} \makeatother} 
\entrymodifiers={++[ ][F]} 

\newcommand{\ra}{\longrightarrow}

\newcommand{\field}[1]{\mathbb{#1}}
\newcommand{\R}{\field{R}} 
\newcommand{\N}{\field{N}} 

\newcommand{\D}{\mathcal{D}}

\newcommand{\eps}{\varepsilon} 

\renewcommand{\phi}{\varphi}
\newcommand{\diff}[1]{\ifmmode\mathchoice{\hbox{\rm d}#1}  
 {\hbox{\rm d}#1}  
 {\scalebox{0.75}{$\hbox{\rm d}#1$}}  
 {\scalebox{0.35}{$\hbox{\rm d}#1$}}  
 \fi} 

\newcommand{\abs}[2][\empty]{\ifx#1\empty\left|#2\right|%
\else#1\vert #2 #1\vert\fi}

\newcommand{\cinfty}{\mathcal{C}^\infty}
\newcommand{\Coo}{\mbox{\ensuremath{\mathcal{C}}}^{\infty}} 

\newcommand{\Rtil}{\widetilde \R} 
\newcommand{\gs}{\mathcal{G}^s} 
\newcommand{\ns}{\mathcal{N}^s} 

\newcommand{\sint}[1]{\langle#1\rangle} 
\newcommand{\Eball}{B^{{\scriptscriptstyle \text{\rm E}}}} 

\newcommand{\csp}[1]{{\text{\rm c}}({#1})}
\newcommand{\fcmp}{\Subset_{\text{f}}}
\newcommand{\frontRise}[2]{\ifmmode\mathchoice{{\vphantom{#1}}^{\scalebox{0.6}{$#2$}}}  
 {{\vphantom{#1}}^{\scalebox{0.56}{$#2$}}}  
 {{\vphantom{#1}}^{\scalebox{0.47}{$#2$}}}  
 {{\vphantom{#1}}^{\scalebox{0.35}{$#2$}}}\fi} 
\newcommand{\RC}[1]{\frontRise{\R}{#1}\Rtil}
\newcommand{\rcrho}{\RC{\rho}}
\newcommand{\rti}{\RC{\rho}}
\newcommand{\gsf}{\frontRise{\mathcal{G}}{\rho}\mathcal{GC}^{\infty}}
\newcommand{\gckf}{\frontRise{\mathcal{G}}{\rho}\mathcal{GC}^{k}}
\newcommand{\gcf}[1]{\frontRise{\mathcal{G}}{\rho}\mathcal{GC}^{#1}}

\newcommand{\norm}[1]{\Vert{#1}\Vert}

\newcommand{\ptind}{\displaystyle \mathop {\ldots\ldots\,}} 

\renewcommand{\L}{\textrm L}

\newcommand{\gf}{\rcrho}
\newcommand{\zs}{\gsf}
\newcommand{\gfs}{\rcrho^d}

\newcommand{\fonction}[5]{\begin{array}[t]{lrcl}#1 :&#2 &\longrightarrow &#3\\&#4& \longmapsto &#5 \end{array}}

\makeatother

\providecommand{\corollaryname}{Corollary}
\providecommand{\definitionname}{Definition}
\providecommand{\examplename}{Example}
\providecommand{\lemmaname}{Lemma}
\providecommand{\remarkname}{Remark}
\providecommand{\theoremname}{Theorem}

\begin{document}

\title[Calculus of variations and optimal control for GF]{Calculus of variations and optimal control for generalized functions}
\author{Gast\~ao S.F. Frederico \and Paolo Giordano$^{*}$ \and Aleksandr A.
Bryzgalov \and Matheus J. Lazo}
\address{\textsc{Address: Federal University of Cear\`a, Brazil}}
\email{gastao.frederico@ua.pt}
\address{\textsc{Faculty of Mathematics, University of Vienna, Austria}}
\thanks{A.A. Bryzgalov has been supported by grants P34113 of the Austrian
Science Fund FWF}
\email{aleksandr.bryzgalov@univie.ac.at}
\thanks{M.J.~Lazo has been supported by CNPq and FAPERGS}
\address{\textsc{Federal University of Rio Grande, Brazil}}
\email{matheuslazo@furg.br}
\thanks{$^{*}$Corresponding author. P.~Giordano has been supported by grants
P30407, P34113, P33538 of the Austrian Science Fund FWF}
\address{\textsc{Faculty of Mathematics, University of Vienna, Austria}}
\email{paolo.giordano@univie.ac.at}
\begin{abstract}
We present an extension of some results of higher order calculus of
variations and optimal control to generalized functions. The framework
is the category of generalized smooth functions, which includes Schwartz
distributions, while sharing many nonlinear properties with ordinary
smooth functions. We prove the higher order Euler-Lagrange equations,
the D'Alembert principle in differential form, the du Bois-Reymond
optimality condition and the Noether's theorem. We start the theory
of optimal control proving a weak form of the Pontryagin maximum principle
and the Noether's theorem for optimal control. We close with a study
of a singularly variable length pendulum, oscillations damped by two
media and the Pais\textendash Uhlenbeck oscillator with singular frequencies.
\end{abstract}

\maketitle
\tableofcontents{}

\section{Introduction: singular Lagrangian mechanics}

Leibniz's axiom \emph{Natura non facit saltus }seems to be inconsistent
not only at the atomic level, but also for several macroscopic phenomena.
Even the simple motion of an elastic bouncing ball seems to be more
easily modeled using non-differentiable functions than classical $\mathcal{C}^{2}$
ones, at least if we are not interested to model the non-trivial behavior
at the collision times. Clearly, we can always use smooth functions
to approximately model the forces acting on the ball, but this would
introduce additional parameters (see e.g.~\cite{RaRKTu07,TuRK98}
for smooth approximations of steep potentials in the billiard problem).
Therefore, the motivation to introduce a suitable kind of generalized
functions formalism in classical mechanics is clear: is it possible,
e.g., to generalize Lagrangian mechanics so that the Lagrangian can
be presented as a generalized function? This would undoubtedly be
of an applicable advantage, since many relevant systems are described
by singular Lagrangians: non-smooth constraints, collisions between
two or more bodies, motion in different or in granular media, discontinuous
propagation of rays of light, even turning on the switch of an electrical
circuit, to name but a few, and only in the framework of classical
physics. Indeed, this type of problems is widely studied (see e.g.~\cite{MoSa16,LiHuZh14,KKO:08,Bro99,Kunzle96,LeMoSj91,Tan93,Tuc93,Lim89,Gra30}),
but sometimes the presented solutions are not general or hold only
for special conditions and particular potentials.

From the purely mathematical point of view, an enlarged space of solutions
(with distributional derivatives) for variational problems leads to
the the so-called Lavrentiev phenomenon, see \cite{ChMi94,GrPr11}.
Variational problems with singular solutions (e.g.~non-differentiable
on a dense set, or with infinite derivatives at some points) are also
studied, see \cite{GrPr11,Syc11,CKOPW,Dav88}.

In this sense, the fact that since J.D.~Marsden's works \cite{Mar68,Mar68b,Mar69}
there has not been many further attempt to use Sobolev-Schwartz distributions
for the description of Hamiltonian mechanics, can be considered as
a clue that the classical distributional framework is not well suited
to face this problem in general terms (see also \cite{Par79} for
a similar approach with a greater focus on the generalization in geometry).
In \cite{Mar68,Mar68b,Mar69}, J.D.~Marsden introduces distributions
on manifolds based on flows. Since the traditional system of Hamilton's
equations breaks down, Marsden considers the flow as a limit of smooth
ones. On the other hand, Kunzinger et al.~in \cite{KuObStVi} (using
Colombeau generalized functions) critically analyses the regularization
approach put forward by Marsden and constructed a counter example
for the main flow theorems of \cite{Mar68}. In this field, a number
of problems remains open, see again \cite{Mar68}. For example, in
the non-smooth case the variational theorems fail, and the study of
the virial equation for hard spheres in a box is still an open question.
All this motivates the use of different spaces of nonlinear generalized
functions in mathematical physics and in singular calculus of variations,
see e.g.~\cite{BeLuSq20,Vic12,KKO:08,Col07,Giu03,MO01,GKOS}.

In this paper, we introduce an approach to variational problems involving
singularities that allows the extension of the classical theory with
very natural statements and proofs. We are interested in extremizing
functionals which are either distributional themselves or whose set
of extremals includes generalized functions. Clearly, distribution
theory, being a linear theory, has certain difficulties when nonlinear
problems are in play.

To overcome this type of problems, we are going to use the category
of \emph{generalized smooth functions} (GSF), see \cite{GiKuVe15,GiKu16,GIO,GiKu18,GIO1}.
This theory seems to be a good candidate, since it is an extension
of classical distribution theory which allows to model nonlinear singular
problems, while at the same time sharing many nonlinear properties
with ordinary smooth functions, like the closure with respect to composition
and several non trivial classical theorems of the calculus. One could
describe GSF as a methodological restoration of Cauchy-Dirac's original
conception of generalized function, see \cite{Dir26,Lau89,Kat-Tal12}.
In essence, the idea of Cauchy and Dirac (but also of Poisson, Kirchhoff,
Helmholtz, Kelvin and Heaviside) was to view generalized functions
as suitable types of smooth set-theoretical maps obtained from ordinary
smooth maps depending on suitable infinitesimal or infinite parameters.
For example, the density of a Cauchy-Lorentz distribution with an
infinitesimal scale parameter was used by Cauchy to obtain classical
properties which nowadays are attributed to the Dirac delta, cf.~\cite{Kat-Tal12}.

In the present work, the foundation of the calculus of variations
is set for functionals defined by arbitrary GSF. This in particular
applies to any Schwartz distribution and any Colombeau generalized
function (see e.g.~\cite{Col92,GKOS}).

The main aim of the paper is to start the higher-order calculus of
variations and the theory of optimal control for GSF. The structure
of the paper is as follows. We start with an introduction into the
setting of GSF and give basic notions concerning GSF and their calculus
that are needed for the calculus of variations (Sec.~\ref{sec:Basic-notions}).
After the basic definitions to set the problem in the framework of
GSF, we prove the higher order Euler-Lagrange equations, the D'Alembert
principle in differential form, the du Bois-Reymond optimality condition
and the Noether's theorem (Sec.~\ref{sec:MRHO}). In Sec.~\ref{sec:Optimal-control},
we start the theory of optimal control proving a weak form of the
Pontryagin maximum principle and the Noether's theorem, and in Sec.~\ref{sec:Examples-and-applications}
we close with a study of a singularly variable length pendulum, oscillations
damped by two media and the Pais\textendash Uhlenbeck oscillator with
singular frequencies.

The paper is self-contained, in the sense that it contains all the
statements required for the proofs of calculus of variations we are
going to present. If proofs of preliminaries are omitted, we clearly
give references to where they can be found. Therefore, to understand
this paper, only a basic knowledge of distribution theory is needed.

\section{\label{sec:Basic-notions}Basic notions}

\subsubsection*{The new ring of scalars}

In this work, $I$ denotes the interval $(0,1]\subseteq\R$ and we
will always use the variable $\eps$ for elements of $I$; we also
denote $\eps$-dependent nets $x\in\R^{I}$ simply by $(x_{\eps})$.
By $\N$ we denote the set of natural numbers, including zero.

We start by defining a new simple non-Archimedean ring of scalars
that extends the real field $\R$. The entire theory is constructive
to a high degree, e.g.~neither ultrafilter nor non-standard method
are used. For all the proofs of results in this section, see \cite{GiKu18,GiKuVe15,GIO1,GiKu16}.
\begin{defn}
\label{def:RCGN}Let $\rho=(\rho_{\eps})\in(0,1]^{I}$ be a net such
that $(\rho_{\eps})\to0$ as $\eps\to0^{+}$ (in the following, such
a net will be called a \emph{gauge}), then
\begin{enumerate}
\item $\mathcal{I}(\rho):=\left\{ (\rho_{\eps}^{-a})\mid a\in\R_{>0}\right\} $
is called the \emph{asymptotic gauge} generated by $\rho$.
\item If $\mathcal{P}(\eps)$ is a property of $\eps\in I$, we use the
notation $\forall^{0}\eps:\,\mathcal{P}(\eps)$ to denote $\exists\eps_{0}\in I\,\forall\eps\in(0,\eps_{0}]:\,\mathcal{P}(\eps)$.
We can read $\forall^{0}\eps$ as \emph{for $\eps$ small}.
\item We say that a net $(x_{\eps})\in\R^{I}$ \emph{is $\rho$-moderate},
and we write $(x_{\eps})\in\R_{\rho}$ if 
\[
\exists(J_{\eps})\in\mathcal{I}(\rho):\ x_{\eps}=O(J_{\eps})\text{ as }\eps\to0^{+},
\]
i.e., if 
\[
\exists N\in\N\,\forall^{0}\eps:\ |x_{\eps}|\le\rho_{\eps}^{-N}.
\]
\item Let $(x_{\eps})$, $(y_{\eps})\in\R^{I}$, then we say that $(x_{\eps})\sim_{\rho}(y_{\eps})$
if 
\[
\forall(J_{\eps})\in\mathcal{I}(\rho):\ x_{\eps}=y_{\eps}+O(J_{\eps}^{-1})\text{ as }\eps\to0^{+},
\]
that is if 
\begin{equation}
\forall n\in\N\,\forall^{0}\eps:\ |x_{\eps}-y_{\eps}|\le\rho_{\eps}^{n}.\label{eq:negligible}
\end{equation}
This is a congruence relation on the ring $\R_{\rho}$ of moderate
nets with respect to pointwise operations, and we can hence define
\[
\RC{\rho}:=\R_{\rho}/\sim_{\rho},
\]
which we call \emph{Robinson-Colombeau ring of generalized numbers}.
This name is justified by \cite{Rob73,C1}: Indeed, in \cite{Rob73}
A.~Robinson introduced the notion of moderate and negligible nets
depending on an arbitrary fixed infinitesimal $\rho$ (in the framework
of nonstandard analysis); independently, J.F.~Colombeau, cf.~e.g.~\cite{C1}
and references therein, studied the same concepts without using nonstandard
analysis, but considering only the particular gauge $\rho_{\eps}=\eps$.
\end{enumerate}
\end{defn}

We can also define an order relation on $\RC{\rho}$ by saying that
$[x_{\eps}]\le[y_{\eps}]$ if there exists $(z_{\eps})\in\R^{I}$
such that $(z_{\eps})\sim_{\rho}0$ (we then say that $(z_{\eps})$
is \emph{$\rho$-negligible}) and $x_{\eps}\le y_{\eps}+z_{\eps}$
for $\eps$ small. Equivalently, we have that $x\le y$ if and only
if there exist representatives $[x_{\eps}]=x$ and $[y_{\eps}]=y$
such that $x_{\eps}\le y_{\eps}$ for all $\eps$. Although the order
$\le$ is not total, we still have the possibility to define the infimum
$[x_{\eps}]\wedge[y_{\eps}]:=[\min(x_{\eps},y_{\eps})]$, the supremum
$[x_{\eps}]\vee[y_{\eps}]:=\left[\max(x_{\eps},y_{\eps})\right]$
of a finite number of generalized numbers. Henceforth, we will also
use the customary notation $\RC{\rho}^{*}$ for the set of invertible
generalized numbers, and we write $x<y$ to say that $x\le y$ and
$x-y\in\rcrho^{*}$. Our notations for intervals are: $[a,b]:=\{x\in\RC{\rho}\mid a\le x\le b\}$,
$[a,b]_{\R}:=[a,b]\cap\R$, and analogously for segments $[x,y]:=\left\{ x+r\cdot(y-x)\mid r\in[0,1]\right\} \subseteq\RC{\rho}^{n}$
and $[x,y]_{\R^{n}}=[x,y]\cap\R^{n}$.

As in every non-Archimedean ring, we have the following
\begin{defn}
\label{def:nonArchNumbs}Let $x\in\RC{\rho}$ be a generalized number,
then
\begin{enumerate}
\item $x$ is \emph{infinitesimal} if $|x|\le r$ for all $r\in\R_{>0}$.
If $x=[x_{\eps}]$, this is equivalent to $\lim_{\eps\to0^{+}}x_{\eps}=0$.
We write $x\approx y$ if $x-y$ is infinitesimal.
\item $x$ is \emph{infinite} if $|x|\ge r$ for all $r\in\R_{>0}$. If
$x=[x_{\eps}]$, this is equivalent to $\lim_{\eps\to0^{+}}\left|x_{\eps}\right|=+\infty$.
\item $x$ is \emph{finite} if $|x|\le r$ for some $r\in\R_{>0}$.
\end{enumerate}
\end{defn}

\noindent For example, setting $\diff{\rho}:=[\rho_{\eps}]\in\RC{\rho}$,
we have that $\diff{\rho}^{n}\in\RC{\rho}$, $n\in\N_{>0}$, is an
invertible infinitesimal, whose reciprocal is $\diff{\rho}^{-n}=[\rho_{\eps}^{-n}]$,
which is necessarily a positive infinite number. Of course, in the
ring $\RC{\rho}$ there exist generalized numbers which are not in
any of the three classes of Def.~\ref{def:nonArchNumbs}, like e.g.~$x_{\eps}=\frac{1}{\eps}\sin\left(\frac{1}{\eps}\right)$.

The following result is useful to deal with positive and invertible
generalized numbers. For its proof, see e.g.~\cite{GKOS}.
\begin{lem}
\label{lem:mayer} Let $x\in\RC{\rho}$. Then the following are equivalent:
\begin{enumerate}
\item \label{enu:positiveInvertible}$x$ is invertible and $x\ge0$, i.e.~$x>0$.
\item \label{enu:strictlyPositive}For each representative $(x_{\eps})\in\R_{\rho}$
of $x$ we have $\forall^{0}\eps:\ x_{\eps}>0$.
\item \label{enu:greater-i_epsTom}For each representative $(x_{\eps})\in\R_{\rho}$
of $x$ we have $\exists m\in\N\,\forall^{0}\eps:\ x_{\eps}>\rho_{\eps}^{m}$.
\item \label{enu:There-exists-a}There exists a representative $(x_{\eps})\in\R_{\rho}$
of $x$ such that $\exists m\in\N\,\forall^{0}\eps:\ x_{\eps}>\rho_{\eps}^{m}$.
\end{enumerate}
\end{lem}

\subsubsection*{Topologies on $\RC{\rho}^{n}$}

On the $\RC{\rho}$-module $\RC{\rho}^{n}$ we can consider the natural
extension of the Euclidean norm, i.e.~$|[x_{\eps}]|:=[|x_{\eps}|]\in\RC{\rho}$,
where $[x_{\eps}]\in\RC{\rho}^{n}$. Even if this generalized norm
takes values in $\RC{\rho}$, it shares some essential properties
with classical norms: 
\begin{align*}
 & |x|=x\vee(-x)\\
 & |x|\ge0\\
 & |x|=0\Rightarrow x=0\\
 & |y\cdot x|=|y|\cdot|x|\\
 & |x+y|\le|x|+|y|\\
 & ||x|-|y||\le|x-y|.
\end{align*}
It is therefore natural to consider on $\RC{\rho}^{n}$ a topology
generated by balls defined by this generalized norm and the set of
radii $\RC{\rho}_{>0}$ of positive invertible numbers:
\begin{defn}
\label{def:setOfRadii}Let $c\in\RC{\rho}^{n}$ and $x$, $y\in\RC{\rho}$,
then:
\begin{enumerate}
\item $B_{r}(c):=\left\{ x\in\RC{\rho}^{n}\mid\left|x-c\right|<r\right\} $
for each $r\in\rcrho_{>0}$.
\item $\Eball_{r}(c):=\{x\in\R^{n}\mid|x-c|<r\}$, for each $r\in\R_{>0}$,
denotes an ordinary Euclidean ball in $\R^{n}$ if $c\in\R^{n}$.
\end{enumerate}
\end{defn}

\noindent The relation $<$ has better topological properties as compared
to the usual strict order relation $a\le b$ and $a\ne b$ (that we
will \emph{never} use) because the set of balls $\left\{ B_{r}(c)\mid r\in\rcrho_{>0},\ c\in\RC{\rho}^{n}\right\} $
is a base for a sequentially Cauchy complete topology on $\RC{\rho}^{n}$
called \emph{sharp topology}. We will call \emph{sharply open set}
any open set in the sharp topology. The existence of infinitesimal
neighborhoods (e.g.~$r=\diff{\rho}$) implies that the sharp topology
induces the discrete topology on $\R$. This is a necessary result
when one has to deal with continuous generalized functions which have
infinite derivatives. In fact, if $f'(x_{0})$ is infinite, we have
$f(x)\approx f(x_{0})$ only for $x\approx x_{0}$ , see \cite{GiKu18}.
Also open intervals are defined using the relation $<$, i.e.~$(a,b):=\{x\in\RC{\rho}\mid a<x<b\}$.
Note that by Lem.~\ref{lem:mayer}.\ref{enu:greater-i_epsTom}, for
all $r\in\rti_{>0}$ there exists $m\in\N$ such that $r\ge\diff{\rho}^{m}$.
Therefore, also the set of balls $\left\{ B_{\diff{\rho}^{m}}(c)\mid m\in\N,\ c\in\RC{\rho}^{n}\right\} $
is a base for the sharp topology.

We will also need the following result. 
\begin{lem}
\label{lem:approxOfBoundaryPointsWithInterior}Let $a$, $b\in\rcrho$
such that $a<b$, then the interior $\text{\emph{int}}\left([a,b]\right)$
in the sharp topology is dense in $[a,b]$. 
\end{lem}

The reader can feel uneasy in considering a ring of scalars such as
$\rti$ instead of a field. On the other hand, as mentioned above,
we will see that all our generalized functions are continuous in the
sharp topology. Therefore, the following result is partly reassuring:
\begin{lem}
\label{lem:invDense}Invertible elements of $\rti$ are dense in the
sharp topology, i.e.
\[
\forall h\in\rti\,\forall\delta\in\rti_{>0}\,\exists k\in(h-\delta,h+\delta):\ k\text{ is invertible}.
\]
\end{lem}

\subsection{Open, closed and bounded sets generated by nets}

A natural way to obtain sharply open, closed and bounded sets in $\RC{\rho}^{n}$
is by using a net $(A_{\eps})$ of subsets $A_{\eps}\subseteq\R^{n}$.
We have two ways of extending the membership relation $x_{\eps}\in A_{\eps}$
to generalized points $[x_{\eps}]\in\RC{\rho}^{n}$ (cf.\ \cite{ObVe08,GiKuVe15}):
\begin{defn}
\label{def:internalStronglyInternal}Let $(A_{\eps})$ be a net of
subsets of $\R^{n}$, then
\begin{enumerate}
\item $[A_{\eps}]:=\left\{ [x_{\eps}]\in\RC{\rho}^{n}\mid\forall^{0}\eps:\,x_{\eps}\in A_{\eps}\right\} $
is called the \emph{internal set} generated by the net $(A_{\eps})$.
\item Let $(x_{\eps})$ be a net of points of $\R^{n}$, then we say that
$x_{\eps}\in_{\eps}A_{\eps}$, and we read it as $(x_{\eps})$ \emph{strongly
belongs to $(A_{\eps})$}, if
\begin{enumerate}
\item $\forall^{0}\eps:\ x_{\eps}\in A_{\eps}$.
\item If $(x'_{\eps})\sim_{\rho}(x_{\eps})$, then also $x'_{\eps}\in A_{\eps}$
for $\eps$ small.
\end{enumerate}
\noindent Moreover, we set $\sint{A_{\eps}}:=\left\{ [x_{\eps}]\in\RC{\rho}^{n}\mid x_{\eps}\in_{\eps}A_{\eps}\right\} $,
and we call it the \emph{strongly internal set} generated by the net
$(A_{\eps})$.
\item We say that the internal set $K=[A_{\eps}]$ is \emph{sharply bounded}
if there exists $M\in\RC{\rho}_{>0}$ such that $K\subseteq B_{M}(0)$.
\item Finally, we say that the $(A_{\eps})$ is a \emph{sharply bounded
net} if there exists $N\in\R_{>0}$ such that $\forall^{0}\eps\,\forall x\in A_{\eps}:\ |x|\le\rho_{\eps}^{-N}$.
\end{enumerate}
\end{defn}

\noindent Therefore, $x\in[A_{\eps}]$ if there exists a representative
$[x_{\eps}]=x$ such that $x_{\eps}\in A_{\eps}$ for $\eps$ small,
whereas this membership is independent from the chosen representative
in case of strongly internal sets. An internal set generated by a
constant net $A_{\eps}=A\subseteq\R^{n}$ will simply be denoted by
$[A]$.

The following theorem (cf.~\cite{ObVe08,GiKuVe15,GIO1}) shows that
internal and strongly internal sets have dual topological properties:
\begin{thm}
\noindent \label{thm:strongMembershipAndDistanceComplement}For $\eps\in I$,
let $A_{\eps}\subseteq\R^{n}$ and let $x_{\eps}\in\R^{n}$. Then
we have
\begin{enumerate}
\item \label{enu:internalSetsDistance}$[x_{\eps}]\in[A_{\eps}]$ if and
only if $\forall q\in\R_{>0}\,\forall^{0}\eps:\ d(x_{\eps},A_{\eps})\le\rho_{\eps}^{q}$.
Therefore $[x_{\eps}]\in[A_{\eps}]$ if and only if $[d(x_{\eps},A_{\eps})]=0\in\RC{\rho}$.
\item \label{enu:stronglyIntSetsDistance}$[x_{\eps}]\in\sint{A_{\eps}}$
if and only if $\exists q\in\R_{>0}\,\forall^{0}\eps:\ d(x_{\eps},A_{\eps}^{\text{c}})>\rho_{\eps}^{q}$,
where $A_{\eps}^{\text{c}}:=\R^{n}\setminus A_{\eps}$. Therefore,
if $(d(x_{\eps},A_{\eps}^{\text{c}}))\in\R_{\rho}$, then $[x_{\eps}]\in\sint{A_{\eps}}$
if and only if $[d(x_{\eps},A_{\eps}^{\text{c}})]>0$.
\item \label{enu:internalAreClosed}$[A_{\eps}]$ is sharply closed.
\item \label{enu:stronglyIntAreOpen}$\sint{A_{\eps}}$ is sharply open.
\item \label{enu:internalGeneratedByClosed}$[A_{\eps}]=\left[\text{\emph{cl}}\left(A_{\eps}\right)\right]$,
where $\text{\emph{cl}}\left(S\right)$ is the closure of $S\subseteq\R^{n}$.
\item \label{enu:stronglyIntGeneratedByOpen}$\sint{A_{\eps}}=\sint{\text{\emph{int}\ensuremath{\left(A_{\eps}\right)}}}$,
where $\emph{int}\left(S\right)$ is the interior of $S\subseteq\R^{n}$.
\end{enumerate}
\end{thm}

\noindent For example, it is not hard to show that the closure in
the sharp topology of a ball of center $c=[c_{\eps}]$ and radius
$r=[r_{\eps}]>0$ is 
\begin{equation}
\overline{B_{r}(c)}=\left\{ x\in\rti^{d}\mid\left|x-c\right|\le r\right\} =\left[\overline{\Eball_{r_{\eps}}(c_{\eps})}\right],\label{eq:closureBall}
\end{equation}
whereas
\[
B_{r}(c)=\left\{ x\in\rti^{d}\mid\left|x-c\right|<r\right\} =\sint{\Eball_{r_{\eps}}(c_{\eps})}.
\]

\subsection{Generalized smooth functions and their calculus}

Using the ring $\rti$, it is easy to consider a Gaussian with an
infinitesimal standard deviation. If we denote this probability density
by $f(x,\sigma)$, and if we set $\sigma=[\sigma_{\eps}]\in\RC{\rho}_{>0}$,
where $\sigma\approx0$, we obtain the net of smooth functions $(f(-,\sigma_{\eps}))_{\eps\in I}$.
This is the basic idea we are going to develop in the following
\begin{defn}
\label{def:netDefMap}Let $X\subseteq\RC{\rho}^{n}$ and $Y\subseteq\RC{\rho}^{d}$
be arbitrary subsets of generalized points. Then we say that 
\[
f:X\longrightarrow Y\text{ is a \emph{generalized smooth function}}
\]
if there exists a net $f_{\eps}\in\cinfty(\Omega_{\eps},\R^{d})$
defining $f$ in the sense that
\begin{enumerate}
\item $X\subseteq\langle\Omega_{\eps}\rangle$,
\item $f([x_{\eps}])=[f_{\eps}(x_{\eps})]\in Y$ for all $x=[x_{\eps}]\in X$,
\item $(\partial^{\alpha}f_{\eps}(x_{\eps}))\in\R_{{\scriptscriptstyle \rho}}^{d}$
for all $x=[x_{\eps}]\in X$ and all $\alpha\in\N^{n}$.
\end{enumerate}
The space of generalized smooth functions (GSF) from $X$ to $Y$
is denoted by $\gsf(X,Y)$.
\end{defn}

Let us note explicitly that this definition states minimal logical
conditions to obtain a set-theoretical map from $X$ into $Y$ and
defined by a net of smooth functions of which we can take arbitrary
derivatives still remaining in the space of $\rho$-moderate nets.
In particular, the following Thm.~\ref{thm:propGSF} states that
the equality $f([x_{\eps}])=[f_{\eps}(x_{\eps})]$ is meaningful,
i.e.~that we have independence from the representatives for all derivatives
$[x_{\eps}]\in X\mapsto[\partial^{\alpha}f_{\eps}(x_{\eps})]\in\RC{\rho}^{d}$,
$\alpha\in\N^{n}$.
\begin{thm}
\label{thm:propGSF}Let $X\subseteq\RC{\rho}^{n}$ and $Y\subseteq\RC{\rho}^{d}$
be arbitrary subsets of generalized points. Let $f_{\eps}\in\cinfty(\Omega_{\eps},\R^{d})$
be a net of smooth functions that defines a generalized smooth map
of the type $X\longrightarrow Y$, then
\begin{enumerate}
\item \label{enu:wellDef}$\forall\alpha\in\N^{n}\,\forall(x_{\eps}),(x'_{\eps})\in\R_{\rho}^{n}:\ [x_{\eps}]=[x'_{\eps}]\in X\ \Rightarrow\ [\partial^{\alpha}f_{\eps}(x_{\eps})]=[\partial^{\alpha}f_{\eps}(x'_{\eps})]$.
\item \label{enu:GSF-cont}Each $f\in\gsf(X,Y)$ is continuous with respect
to the sharp topologies induced on $X$, $Y$.
\item \label{enu:globallyDefNet}$f:X\longrightarrow Y$ is a GSF if and
only if there exists a net $v_{\eps}\in\cinfty(\R^{n},\R^{d})$ defining
a generalized smooth map of type $X\longrightarrow Y$ such that $f=[v_{\eps}(-)]|_{X}$.
\item \label{enu:category}GSF are closed with respect to composition, i.e.~subsets
$S\subseteq\RC{\rho}^{s}$ with the trace of the sharp topology, and
GSF as arrows form a subcategory of the category of topological spaces.
We will call this category $\gsf$, the \emph{category of GSF}. Therefore,
with pointwise sum and product, any space $\gsf(X,\rcrho)$ is an
algebra.
\end{enumerate}
\end{thm}

Similarly, we can define generalized functions of class $\gckf$,
with $k\le+\infty$:
\begin{defn}
Let $X\subseteq\RC{\rho}^{n}$ and $Y\subseteq\RC{\rho}^{d}$ be arbitrary
subsets of generalized points and $k\in\N\cup\{+\infty\}$. Then we
say that 
\[
f:X\longrightarrow Y\text{ is a \emph{generalized }}\mathcal{C}^{k}\text{ \emph{function}}
\]
if there exists a net $f_{\eps}\in\mathcal{C}^{k}(\Omega_{\eps},\R^{d})$
defining $f$ in the sense that
\begin{enumerate}
\item $X\subseteq\langle\Omega_{\eps}\rangle$,
\item $f([x_{\eps}])=[f_{\eps}(x_{\eps})]\in Y$ for all $x=[x_{\eps}]\in X$,
\item \label{enu:moderCk}$(\partial^{\alpha}f_{\eps}(x_{\eps}))\in\R_{{\scriptscriptstyle \rho}}^{d}$
for all $x=[x_{\eps}]\in X$ and all $\alpha\in\N^{n}$ such that
$|\alpha|\le k$.
\item \label{enu:k1}$\forall\alpha\in\N^{n}\,\forall[x_{\eps}],[x'_{\eps}]\in X:\ |\alpha|=k,\ [x_{\eps}]=[x'_{\eps}]\ \Rightarrow\ [\partial^{\alpha}f_{\eps}(x_{\eps})]=[\partial^{\alpha}f_{\eps}(x'_{\eps})]$.
\item \label{enu:k2}For all $\alpha\in\N^{n}$, with $|\alpha|=k$, the
map $[x_{\eps}]\in X\mapsto[\partial^{\alpha}f_{\eps}(x_{\eps})]\in\rti^{d}$
is continuous in the sharp topology.
\end{enumerate}
The space of generalized $\mathcal{C}^{k}$ functions from $X$ to
$Y$ is denoted by $\gckf(X,Y)$.
\end{defn}

\noindent Note that properties \ref{enu:k1}, \ref{enu:k2} are required
only for $|\alpha|=k$ because for lower length they can be proved
using property \ref{enu:moderCk} and the classical mean value theorem
for $f_{\eps}$ (see e.g.~\cite{GiKuVe15,GIO1}). From \ref{enu:wellDef}
and \ref{enu:GSF-cont} of Thm.~\ref{thm:propGSF} it follows that
this definition of $\gckf$ is equivalent to Def.~\ref{def:netDefMap}
if $k=+\infty$. Moreover, properties similar to \ref{enu:globallyDefNet},
\ref{enu:category} of Thm.~\ref{thm:propGSF} can also be proved
for $\gckf$.

The differential calculus for $\gcf{k+1}$ maps can be introduced
by showing existence and uniqueness of another $\gckf$ map serving
as incremental ratio (sometimes this is called \emph{derivative �
la Carath�odory}, see e.g.~\cite{Kuh91}).
\begin{thm}
\label{thm:FR-forGSF} Let $U\subseteq\RC{\rho}^{n}$ be a sharply
open set, let $v=[v_{\eps}]\in\RC{\rho}^{n}$, $k\in\N\cup\{+\infty\}$,
and let $f\in\gcf{k+1}(U,\RC{\rho})$ be a $\gcf{k+1}$map generated
by the net of functions $f_{\eps}\in\mathcal{C}^{k+1}(\Omega_{\eps},\R)$.
Then
\begin{enumerate}
\item \label{enu:existenceRatio}There exists a sharp neighborhood $T$
of $U\times\{0\}$ and a map $r\in\gckf(T,\RC{\rho})$, called the
\emph{generalized incremental ratio} of $f$ \emph{along} $v$, such
that 
\[
\forall(x,h)\in T:\ f(x+hv)=f(x)+h\cdot r(x,h).
\]
\item \label{enu:uniquenessRatio}Any two generalized incremental ratios
coincide on a sharp neighborhood of $U\times\{0\}$, so that we can
use the notation $\frac{\partial f}{\partial v}[x;h]:=r(x,h)$ if
$(x,h)$ are sufficiently small.
\item \label{enu:defDer}We have $\frac{\partial f}{\partial v}[x;0]=\left[\frac{\partial f_{\eps}}{\partial v_{\eps}}(x_{\eps})\right]$
for every $x\in U$ and we can thus define $Df(x)\cdot v:=\frac{\partial f}{\partial v}(x):=\frac{\partial f}{\partial v}[x;0]$,
so that $\frac{\partial f}{\partial v}\in\gckf(U,\RC{\rho})$.
\end{enumerate}
\end{thm}

Note that this result permits us to consider the partial derivative
of $f$ with respect to an arbitrary generalized vector $v\in\RC{\rho}^{n}$
which can be, e.g., infinitesimal or infinite. Using recursively this
result, we can also define subsequent differentials $D^{j}f(x)$ as
$j-$multilinear maps, and we set $D^{j}f(x)\cdot h^{j}:=D^{j}f(x)(h,\ptind^{j},h)$
if $j\le k$. The set of all the $j-$multilinear maps $\left(\rti^{n}\right)^{j}\ra\rti^{d}$
over the ring $\rti$ will be denoted by $L^{j}(\rti^{n},\rti^{d})$.
For $A=[A_{\eps}(-)]\in L^{j}(\rti^{n},\rti^{d})$, we set $|{A}|:=[|{A_{\eps}}|]$,
the generalized number defined by the operator norms of the multilinear
maps $A_{\eps}\in L^{j}(\R^{n},\R^{d})$.

The following result follows from the analogous properties for the
nets of smooth functions defining $f$ and $g$.
\begin{thm}
\label{thm:rulesDer} Let $U\subseteq\rcrho^{n}$ be an open subset
in the sharp topology, let $v\in\rcrho^{n}$ and $f$, $g\in\gcf{k+1}(U,\rti^{d})$.
Then
\begin{enumerate}
\item $\frac{\partial(f+g)}{\partial v}=\frac{\partial f}{\partial v}+\frac{\partial g}{\partial v}$
\item $\frac{\partial(r\cdot f)}{\partial v}=r\cdot\frac{\partial f}{\partial v}\quad\forall r\in\rcrho$
\item $\frac{\partial(f\cdot g)}{\partial v}=\frac{\partial f}{\partial v}\cdot g+f\cdot\frac{\partial g}{\partial v}$
\item For each $x\in U$, the map $\diff{f}(x).v:=\frac{\partial f}{\partial v}(x)\in\rcrho$
is $\rcrho$-linear in $v\in\rcrho^{n}$
\item Let $U\subseteq\rcrho^{n}$ and $V\subseteq\rcrho^{d}$ be open subsets
in the sharp topology and $g\in\gcf{k+1}(V,U)$, $f\in\gcf{k+1}(U,\rcrho)$
be generalized maps. Then for all $x\in V$ and all $v\in\rcrho^{d}$,
we have $\frac{\partial\left(f\circ g\right)}{\partial v}(x)=\diff{f}\left(g(x)\right).\frac{\partial g}{\partial v}(x)$.
\end{enumerate}
\end{thm}

Note that the absolute value function $|-|:\rti\ra\rti$ is not a
GSF because its derivative is not sharply continuous at the origin;
clearly, it is a $\gcf{0}$ function. This is a good motivation to
introduce integration of $\gckf$ functions. One dimensional integral
calculus is based on the following
\begin{thm}
\label{thm:existenceUniquenessPrimitives}Let $k\in\N\cup\{+\infty\}$
and $f\in\gckf([a,b],\rcrho)$ be defined in the interval $[a,b]\subseteq\RC{\rho}$,
where $a<b$. Let $c\in[a,b]$. Then, there exists one and only one
generalized $\mathcal{C}^{k+1}$ map $F\in\gcf{k+1}([a,b],\rcrho)$
such that $F(c)=0$ and $F'(x)=f(x)$ for all $x\in[a,b]$. Moreover,
if $f$ is defined by the net $f_{\eps}\in\mathcal{C}^{k}(\R,\R)$
and $c=[c_{\eps}]$, then $F(x)=\left[\int_{c_{\eps}}^{x_{\eps}}f_{\eps}(s)\diff{s}\right]$
for all $x=[x_{\eps}]\in[a,b]$.
\end{thm}

\noindent We can thus define
\begin{defn}
\label{def:integral}Under the assumptions of Thm.~\ref{thm:existenceUniquenessPrimitives},
we denote by $\int_{c}^{(-)}f:=\int_{c}^{(-)}f(s)\,\diff{s}\in\gcf{k+1}([a,b],\rcrho)$
the unique generalized $\mathcal{C}^{k+1}$ map such that:
\begin{enumerate}
\item $\int_{c}^{c}f=0$
\item $\left(\int_{u}^{(-)}f\right)'(x)=\frac{\diff{}}{\diff{x}}\int_{u}^{x}f(s)\diff{s}=f(x)$
for all $x\in[a,b]$.
\end{enumerate}
\end{defn}

\noindent All the classical rules of integral calculus hold in this
setting:
\begin{thm}
\label{thm:intRules}Let $f\in\gckf(U,\rcrho)$ and $g\in\gckf(V,\rcrho)$
be generalized $\mathcal{C}^{k}$ maps defined on sharply open domains
in $\rcrho$. Let $a$, $b\in\rcrho$ with $a<b$ and $c$, $d\in[a,b]\subseteq U\cap V$,
then
\begin{enumerate}
\item \label{enu:additivityFunction}$\int_{c}^{d}\left(f+g\right)=\int_{c}^{d}f+\int_{c}^{d}g$
\item \label{enu:homog}$\int_{c}^{d}\lambda f=\lambda\int_{c}^{d}f\quad\forall\lambda\in\rcrho$
\item \label{enu:additivityDomain}$\int_{c}^{d}f=\int_{c}^{e}f+\int_{e}^{d}f$
for all $e\in[a,b]$
\item \label{enu:chageOfExtremes}$\int_{c}^{d}f=-\int_{d}^{c}f$
\item \label{enu:foundamental}$\int_{c}^{d}f'=f(d)-f(c)$
\item \label{enu:intByParts}$\int_{c}^{d}f'\cdot g=\left[f\cdot g\right]_{c}^{d}-\int_{c}^{d}f\cdot g'$
\item \label{enu:intMonotone}If $f(x)\le g(x)$ for all $x\in[a,b]$, then
$\int_{a}^{b}f\le\int_{a}^{b}g$.
\item Let $a$, $b$, $c$, $d\in\rcrho$, with $a<b$ and $c<d$, and $f\in\gcf{k+1}([a,b]\times[c,d],\rcrho^{d})$,
then
\[
\frac{\diff{}}{\diff{s}}\int_{a}^{b}f(\tau,s)\,\diff{\tau}=\int_{a}^{b}\frac{\partial}{\partial s}f(\tau,s)\,\diff{\tau}\quad\forall s\in[c,d].
\]
\end{enumerate}
\end{thm}

\begin{thm}
\label{thm:changeOfVariablesInt}Let $f\in\gckf(U,\rcrho)$ and $\phi\in\gcf{k+1}(V,U)$
be GSF defined on sharply open domains in $\rcrho$. Let $a$, $b\in\rcrho$,
with $a<b$, such that $[a,b]\subseteq V$, $\phi(a)<\phi(b)$, $[\phi(a),\phi(b)]\subseteq U$.
Finally, assume that $\phi([a,b])\subseteq[\phi(a),\phi(b)]$. Then
\[
\int_{\phi(a)}^{\phi(b)}f(t)\diff{t}=\int_{a}^{b}f\left[\phi(s)\right]\cdot\phi'(s)\diff{s}.
\]
\end{thm}

We also have a generalization of Taylor formula:
\begin{thm}
\label{thm:Taylor}Let $f\in\gckf(U,\rcrho)$ be a generalized $\mathcal{C}^{k+1}$
function defined in the sharply open set $U\subseteq\rcrho^{d}$.
Let $a$, $b\in\rcrho^{d}$ such that the line segment $[a,b]\subseteq U$,
and set $h:=b-a$. Then, for all $n\in\N_{\le k}$ we have
\begin{enumerate}
\item \label{enu:LagrangeRest}$\exists\xi\in[a,b]:\ f(a+h)=\sum_{j=0}^{n}\frac{\diff{^{j}f}(a)}{j!}\cdot h^{j}+\frac{\diff{^{n+1}f}(\xi)}{(n+1)!}\cdot h^{n+1}.$
\item \label{enu:integralRest}$f(a+h)=\sum_{j=0}^{n}\frac{\diff{^{j}f}(a)}{j!}\cdot h^{j}+\frac{1}{n!}\cdot\int_{0}^{1}(1-t)^{n}\,\diff{^{n+1}f}(a+th)\cdot h^{n+1}\,\diff{t}.$
\end{enumerate}
\noindent Moreover, there exists some $R\in\rcrho_{>0}$ such that
\begin{equation}
\forall k\in B_{R}(0)\,\exists\xi\in[a,a+k]:\ f(a+k)=\sum_{j=0}^{n}\frac{\diff{^{j}f}(a)}{j!}\cdot k^{j}+\frac{\diff{^{n+1}f}(\xi)}{(n+1)!}\cdot k^{n+1}\label{eq:LagrangeInfRest}
\end{equation}
\begin{equation}
\frac{\diff{^{n+1}f}(\xi)}{(n+1)!}\cdot k^{n+1}=\frac{1}{n!}\cdot\int_{0}^{1}(1-t)^{n}\,\diff{^{n+1}f}(a+tk)\cdot k^{n+1}\,\diff{t}\approx0.\label{eq:integralInfRest}
\end{equation}
\end{thm}

Formulas \ref{enu:LagrangeRest} and \ref{enu:integralRest} correspond
to a plain generalization of Taylor's theorem for ordinary functions
with Lagrange and integral remainder, respectively. Dealing with generalized
functions, it is important to note that this direct statement also
includes the possibility that the differential $\diff{^{n+1}f}(\xi)$
may be infinite at some point. For this reason, in \eqref{eq:LagrangeInfRest}
and \eqref{eq:integralInfRest}, considering a sufficiently small
increment $k$, we get more classical infinitesimal remainders $\diff{^{n+1}f}(\xi)\cdot k^{n+1}\approx0$.
We can also define right and left derivatives as e.g.~$f'(a):=f'_{+}(a):=\lim_{\substack{t\to a\\
a<t
}
}f'(t)$, which always exist if $f\in\gcf{k+1}([a,b],\RC{\rho}^{d})$.

Analogously to the classical case, we say that $x_{0}\in X$ is a
local minimum of $f\in\gckf(X,\rcrho)$ if there exists a sharply
open neighbourhood (in the trace topology) $Y\subseteq X$ of $x_{0}$
such that $f(x_{0})\leq f(y)$ for all $y\in Y$. A local maximum
is defined accordingly.
\begin{lem}
\label{lem:local_min_diff}Let $X\subseteq\rcrho$ and let $f\in\gsf(X,\rcrho)$.
\begin{enumerate}
\item If $x_{0}\in X$ is a sharply interior local minimum of $f$ then
$f'(x_{0})=0$.
\item Let $a$, $b\in\rcrho$ with $a<b$, $[a,b]\subseteq X$ and $x_{0}$
be a sharply interior point of $[a,b]$. Assume that $x_{0}$ is a
local minimum of $f$. Then $f''(x_{0})\geq0$. Vice versa, if $f'(x_{0})=0$
and $f''(x_{0})>0$, then $x_{0}$ is a local minimum of $f$.
\end{enumerate}
\end{lem}

\subsection{\label{subsec:Embedding}Embedding of Sobolev-Schwartz distributions
and Colombeau generalized functions}

We finally recall two results that give a certain flexibility in constructing
embeddings of Schwartz distributions. Note that both the infinitesimal
$\rho$ and the embedding of Schwartz distributions have to be chosen
depending on the problem we aim to solve. A trivial example in this
direction is the ODE $y'=y/\diff{\eps}$, which cannot be solved for
$\rho=(\eps)$, but it has a solution for $\rho=(e^{-1/\eps})$. As
another simple example, if we need the property $H(0)=1/2$, where
$H$ is the Heaviside function, then we have to choose the embedding
of distributions accordingly. See also \cite{GiLu16,LuGi17} for further
details.\\
 If $\phi\in\mathcal{D}(\R^{n})$, $r\in\R_{>0}$ and $x\in\R^{n}$,
we use the notations $r\odot\phi$ for the function $x\in\R^{n}\mapsto\frac{1}{r^{n}}\cdot\phi\left(\frac{x}{r}\right)\in\R$
and $x\oplus\phi$ for the function $y\in\R^{n}\mapsto\phi(y-x)\in\R$.
These notations permit us to highlight that $\odot$ is a free action
of the multiplicative group $(\R_{>0},\cdot,1)$ on $\mathcal{D}(\R^{n})$
and $\oplus$ is a free action of the additive group $(\R_{>0},+,0)$
on $\mathcal{D}(\R^{n})$. We also have the distributive property
$r\odot(x\oplus\phi)=rx\oplus r\odot\phi$.
\begin{lem}
\label{lem:strictDeltaNet}Let $b\in\R_{\rho}$ be a net such that
$\lim_{\eps\to0^{+}}b_{\eps}=+\infty$. Let $d\in(0,1)$. There exists
a net $\left(\psi_{\eps}\right)_{\eps\in I}$ of $\mathcal{D}(\R^{n})$
with the properties:
\begin{enumerate}
\item \label{enu:suppStrictDeltaNet}$supp(\psi_{\eps})\subseteq B_{1}(0)$
and $\psi_{\eps}$ is even for all $\eps\in I$.
\item Let $\omega_{n}$ denote the surface area of $S^{n-1}$ and set $c_{n}:=\frac{2n}{\omega_{n}}$
for $n>1$ and $c_{1}:=1$, then $\psi_{\eps}(0)=c_{n}$ for all $\eps\in I$.
\item \label{enu:intOneStrictDeltaNet}$\int\psi_{\eps}=1$ for all $\eps\in I$.
\item \label{enu:moderateStrictDeltaNet}$\forall\alpha\in\N^{n}\,\exists p\in\N:\ \sup_{x\in\R^{n}}\left|\partial^{\alpha}\psi_{\eps}(x)\right|=O(b_{\eps}^{p})$
as $\eps\to0^{+}$.
\item \label{enu:momentsStrictDeltaNet}$\forall j\in\N\,\forall^{0}\eps:\ 1\le|\alpha|\le j\Rightarrow\int x^{\alpha}\cdot\psi_{\eps}(x)\diff{x}=0$.
\item \label{enu:smallNegPartStrictDeltaNet}$\forall\eta\in\R_{>0}\,\forall^{0}\eps:\ \int\left|\psi_{\eps}\right|\le1+\eta$.
\item \label{enu:int1Dim}If $n=1$, then the net $(\psi_{\eps})_{\eps\in I}$
can be chosen so that $\int_{-\infty}^{0}\psi_{\eps}=d$.
\end{enumerate}
\noindent In particular $\psi_{\eps}^{b}:=b_{\eps}^{-1}\odot\psi_{\eps}$
satisfies \ref{enu:intOneStrictDeltaNet} - \ref{enu:smallNegPartStrictDeltaNet}.
\end{lem}

\noindent Concerning embeddings of Schwartz distributions, we have
the following result, where $\csp{\Omega}:=\{[x_{\eps}]\in[\Omega]\mid\exists K\Subset\Omega\,\forall^{0}\eps:\ x_{\eps}\in K\}$
is called the set of \emph{compactly supported points in }$\Omega\subseteq\R^{n}$.
\begin{thm}
\label{thm:embeddingD'}Under the assumptions of Lemma \ref{lem:strictDeltaNet},
let $\Omega\subseteq\R^{n}$ be an open set and let $(\psi_{\eps}^{b})$
be the net defined in \ref{lem:strictDeltaNet}. Then the mapping
\[
\iota_{\Omega}^{b}:T\in\mathcal{E}'(\Omega)\mapsto\left[\left(T\ast\psi_{\eps}^{b}\right)(-)\right]\in\gsf(\csp{\Omega},\rti)
\]
uniquely extends to a sheaf morphism of real vector spaces 
\[
\iota^{b}:\mathcal{D}'\ra\gsf(\csp{(-)},\rti),
\]
and satisfies the following properties:
\begin{enumerate}
\item \label{enu:embSmooth}If $b\ge\diff{\rho}^{-a}$ for some $a\in\R_{>0}$,
then $\iota^{b}|_{\Coo(-)}:\Coo(-)\ra\gsf(\csp{(-)},\RC{\rho})$ is
a sheaf morphism of algebras and $\iota_{\Omega}^{b}(f)(x)=f(x)$
for all smooth functions $f\in\Coo(\Omega)$ and all $x\in\Omega$;
\item If $T\in\mathcal{E}'(\Omega)$ then $\text{\text{\emph{supp}}}(T)=\text{\emph{\text{supp}}}(\iota_{\Omega}^{b}(T))$;
\item \label{enu:eps-D'}$\lim_{\eps\to0^{+}}\int_{\Omega}\iota_{\Omega}^{b}(T)_{\eps}\cdot\phi=\langle T,\phi\rangle$
for all $\phi\in\mathcal{D}(\Omega)$ and all $T\in\mathcal{D}'(\Omega)$;
\item $\iota^{b}$ commutes with partial derivatives, i.e.~$\partial^{\alpha}\left(\iota_{\Omega}^{b}(T)\right)=\iota_{\Omega}^{b}\left(\partial^{\alpha}T\right)$
for each $T\in\mathcal{D}'(\Omega)$ and $\alpha\in\N$.
\end{enumerate}
\end{thm}

Concerning the embedding of Colombeau generalized functions, we recall
that the special Colombeau algebra on $\Omega$ is defined as the
quotient $\gs(\Omega):=\mathcal{E}_{M}(\Omega)/\ns(\Omega)$ of \emph{moderate
nets} over \emph{negligible nets}, where the former is 
\[
\mathcal{E}_{M}(\Omega):=\{(u_{\eps})\in\cinfty(\Omega)^{I}\mid\forall K\Subset\Omega\,\forall\alpha\in\N^{n}\,\exists N\in\N:\sup_{x\in K}|\partial^{\alpha}u_{\eps}(x)|=O(\eps^{-N})\}
\]
and the latter is 
\[
\ns(\Omega):=\{(u_{\eps})\in\cinfty(\Omega)^{I}\mid\forall K\Subset\Omega\,\forall\alpha\in\N^{n}\,\forall m\in\N:\sup_{x\in K}|\partial^{\alpha}u_{\eps}(x)|=O(\eps^{m})\}.
\]
Using $\rho=(\eps)$, we have the following compatibility result:
\begin{thm}
\label{thm:inclusionCGF}A Colombeau generalized function $u=(u_{\eps})+\ns(\Omega)^{d}\in\gs(\Omega)^{d}$
defines a GSF $u:[x_{\eps}]\in\csp{\Omega}\longrightarrow[u_{\eps}(x_{\eps})]\in\Rtil^{d}$.
This assignment provides a bijection of $\gs(\Omega)^{d}$ onto $\gsf(\csp{\Omega},\rti^{d})$
for every open set $\Omega\subseteq\R^{n}$.
\end{thm}

\begin{example}
\label{enu:deltaCompDelta}~
\begin{enumerate}
\item \label{enu:delta}Let $\delta$, $H\in\gsf(\rcrho,\rcrho)$ be the
corresponding $\iota^{b}$-embeddings of the Dirac delta and of the
Heaviside function. Then $\delta(x)=b\cdot\psi(b\cdot x)$, where
$\psi(x):=[\psi_{\eps}(x_{\eps})]$. We have that $\delta(0)=b\ge\diff{\rho}^{-a}$
and $\delta(x)=0$ if $|bx|\ge1$ because of Lem.~\ref{lem:strictDeltaNet}.\ref{enu:suppStrictDeltaNet}.
The condition $|bx|\ge1$ surely holds e.g.~if $|x|>\frac{-1}{k\log\diff{\rho}}\approx0$
for some $k\in\R_{>0}$ because $b\ge\diff{\rho}^{-a}$; in particular,
it is satisfied if $|x|>r$ for some $r\in\R_{>0}$. By the intermediate
value theorem (see \cite{GIO1}), $\delta$ takes any value in the
interval $[0,b]\subseteq\rcrho$. Similar properties can be stated
e.g.~for $\delta^{2}(x)=b^{2}\cdot\psi(b\cdot x)^{2}$.
\item Analogously, we have $H(x)=1$ if $|bx|\ge1$ and $x>0$; $H(x)=0$
if $|bx|\ge1$ and $x<0$; finally $H(0)=\frac{1}{2}$ because of
Lem.~\ref{lem:strictDeltaNet}.\ref{enu:suppStrictDeltaNet}. By
the intermediate value theorem, $H$ takes any value in the interval
$[0,1]\subseteq\rcrho$.
\item The composition $\delta\circ\delta\in\gsf(\rcrho,\rcrho)$ is given
by $(\delta\circ\delta)(x)=b\psi\left(b^{2}\psi(bx)\right)$ and is
an even function. If $|bx|\ge1$, then $(\delta\circ\delta)(x)=b$.
Since $(\delta\circ\delta)(0)=0$, again using the intermediate value
theorem, we have that $\delta\circ\delta$ takes any value in the
interval $[0,b]\subseteq\rcrho$. Suitably choosing the net $(\psi_{\eps})$
it is possible to have that if $-\frac{1}{2b}\le x\le\frac{1}{2b}$
(hence $x$ is infinitesimal), then $(\delta\circ\delta)(x)=0$. If
$x=\frac{k}{b}$ for some $k\in\N_{>0}$, then $x$ is still infinitesimal
but $(\delta\circ\delta)(x)=b$. Analogously, one can deal with compositions
such as $H\circ\delta$ and $\delta\circ H$.
\end{enumerate}
\noindent See Fig.~\ref{fig:MollifierHeaviside} for a graphical
representations of $\delta$ and $H$. The infinitesimal oscillations
shown in this figure occur only in an infinitesimal neighborhood of
the origin because of \ref{enu:delta} in example \ref{enu:deltaCompDelta}
and because $\frac{-1}{k\log\diff{\rho}}\approx0$. They can be proved
to actually occur as a consequence of Lem.~\ref{lem:strictDeltaNet}.\ref{enu:momentsStrictDeltaNet}
which is a necessary property to prove Thm.~\ref{thm:embeddingD'}.\ref{enu:embSmooth},
see \cite{GIO1,GiLu16}. It is well-known that the latter property
is one of the core ideas to bypass the Schwartz's impossibility theorem,
see e.g.~\cite{GKOS}.
\end{example}

\noindent \begin{center}
\begin{figure}
\label{fig: Col_mol}
\begin{centering}
\includegraphics[scale=0.15]{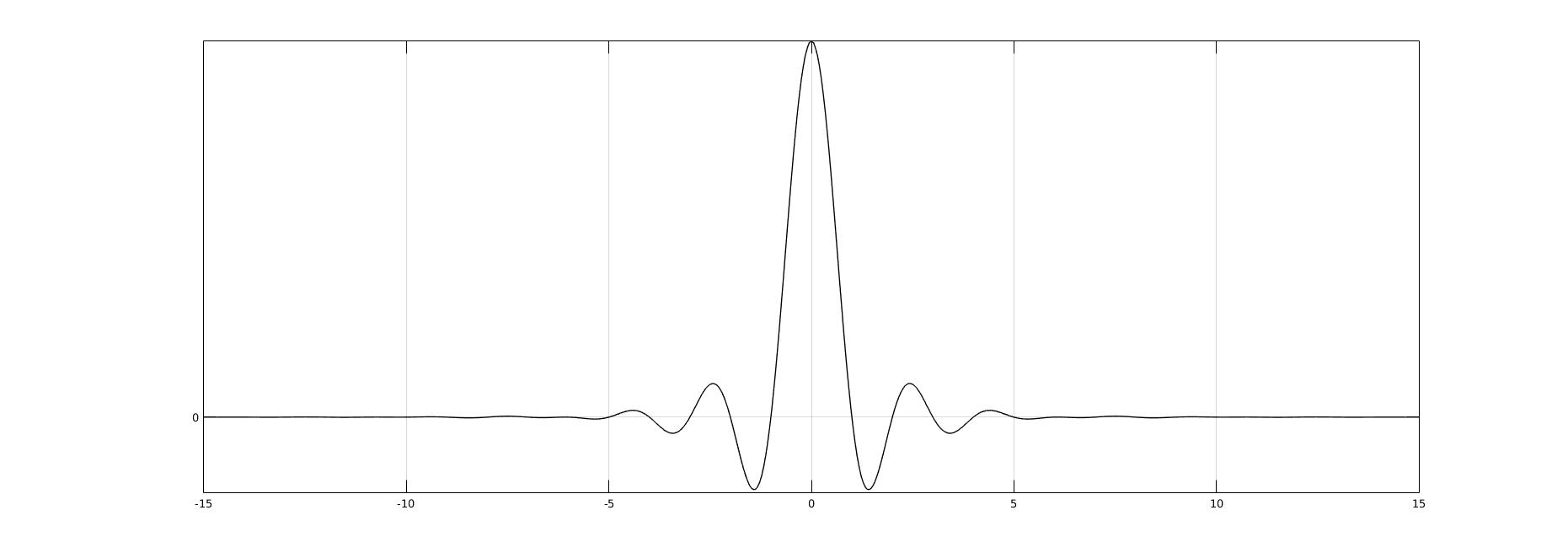}
\par\end{centering}
\begin{centering}
\includegraphics[scale=0.15]{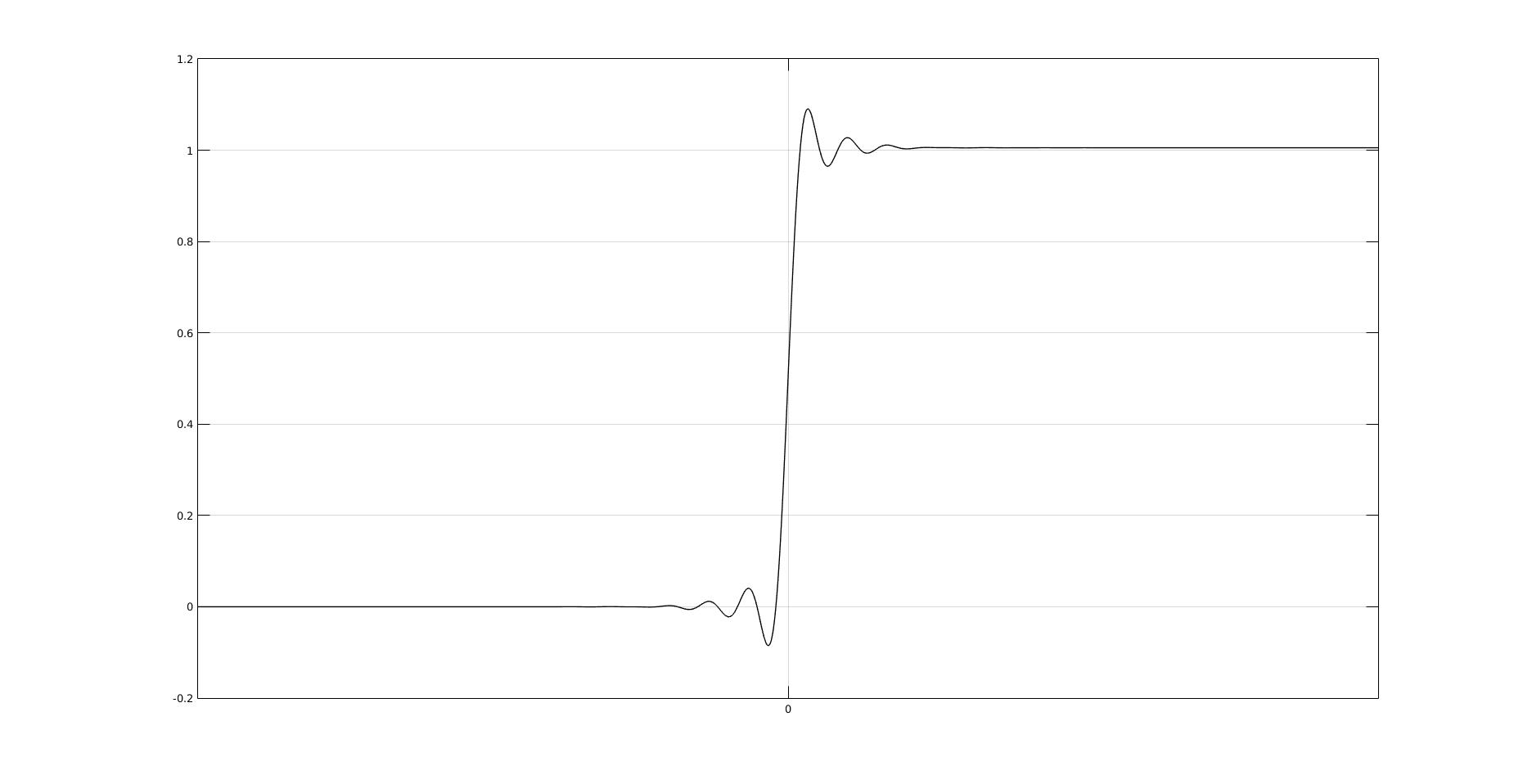}
\par\end{centering}
\caption{\label{fig:MollifierHeaviside}Representations of Dirac delta and
Heaviside function}
\end{figure}
\par\end{center}

\subsection{Functionally compact sets and multidimensional integration}

\subsubsection{\label{subsec:EVTandFcmp}Extreme value theorem and functionally
compact sets}

For GSF, suitable generalizations of many classical theorems of differential
and integral calculus hold: intermediate value theorem, mean value
theorems, suitable sheaf properties, local and global inverse function
theorems, Banach fixed point theorem and a corresponding Picard-Lindel�f
theorem both for ODE and PDE, see \cite{GiKuVe15,GiKu16,GIO1,LuGi17,GiLu16}.

Even though the intervals $[a,b]\subseteq\Rtil$, $a$, $b\in\R$,
are not compact in the sharp topology (see \cite{GiKuVe15}), analogously
to the case of smooth functions, a $\gckf$ map satisfies an extreme
value theorem on such sets. In fact, we have:
\begin{thm}
\label{thm:extremeValues}Let $f\in\gckf(X,\Rtil)$ be a generalized
$\mathcal{C}^{k}$ function defined on the subset $X$ of $\Rtil^{n}$.
Let $\emptyset\ne K=[K_{\eps}]\subseteq X$ be an internal set generated
by a sharply bounded net $(K_{\eps})$ of compact sets $K_{\eps}\Subset\R^{n}$
, then 
\begin{equation}
\exists m,M\in K\,\forall x\in K:\ f(m)\le f(x)\le f(M).\label{eq:epsExtreme}
\end{equation}
\end{thm}

We shall use the assumptions on $K$ and $(K_{\eps})$ given in this
theorem to introduce a notion of ``compact subset'' which behaves
better than the usual classical notion of compactness in the sharp
topology.
\begin{defn}
\label{def:functCmpt} A subset $K$ of $\Rtil^{n}$ is called \emph{functionally
compact}, denoted by $K\fcmp\Rtil^{n}$, if there exists a net $(K_{\eps})$
such that
\begin{enumerate}
\item \label{enu:defFunctCmpt-internal}$K=[K_{\eps}]\subseteq\Rtil^{n}$.
\item \label{enu:defFunctCmpt-sharpBound}$\exists R\in\rcrho_{>0}:\ K\subseteq B_{R}(0)$,
i.e.~$K$ is sharply bounded.
\item \label{enu:defFunctCmpt-cmpt}$\forall\eps\in I:\ K_{\eps}\Subset\R^{n}$.
\end{enumerate}
If, in addition, $K\subseteq U\subseteq\Rtil^{n}$ then we write $K\fcmp U$.
Finally, we write $[K_{\eps}]\fcmp U$ if \ref{enu:defFunctCmpt-sharpBound},
\ref{enu:defFunctCmpt-cmpt} and $[K_{\eps}]\subseteq U$ hold. Any
net $(K_{\eps})$ such that $[K_{\eps}]=K$ is called a \emph{representative}
of $K$.
\end{defn}

\noindent We motivate the name \emph{functionally compact subset}
by noting that on this type of subsets, GSF have properties very similar
to those that ordinary smooth functions have on standard compact sets.
\begin{rem}
\noindent \label{rem:defFunctCmpt}\ 
\begin{enumerate}
\item \label{enu:rem-defFunctCmpt-closed}By Thm.~\ref{thm:strongMembershipAndDistanceComplement}.\ref{enu:internalAreClosed},
any internal set $K=[K_{\eps}]$ is closed in the sharp topology.
In particular, the open interval $(0,1)\subseteq\Rtil$ is not functionally
compact since it is not closed.
\item \label{enu:rem-defFunctCmpt-ordinaryCmpt}If $H\Subset\R^{n}$ is
a non-empty ordinary compact set, then the internal set $[H]$ is
functionally compact. In particular, $[0,1]=\left[[0,1]_{\R}\right]$
is functionally compact.
\item \label{enu:rem-defFunctCmpt-empty}The empty set $\emptyset=\widetilde{\emptyset}\fcmp\Rtil$.
\item \label{enu:rem-defFunctCmpt-equivDef}$\Rtil^{n}$ is not functionally
compact since it is not sharply bounded.
\item \label{enu:rem-defFunctCmpt-cmptlySuppPoints}The set of compactly
supported points $\csp{\R}$ is not functionally compact because the
GSF $f(x)=x$ does not satisfy the conclusion \eqref{eq:epsExtreme}
of Thm.~\ref{thm:extremeValues}.
\end{enumerate}
\end{rem}

\noindent In the present paper, we need the following properties of
functionally compact sets.
\begin{thm}
\label{thm:image}Let $K\subseteq X\subseteq\Rtil^{n}$, $f\in\gckf(X,\Rtil^{d})$.
Then $K\fcmp\Rtil^{n}$ implies $f(K)\fcmp\Rtil^{d}$.
\end{thm}

\noindent As a corollary of this theorem and Rem.\ \eqref{rem:defFunctCmpt}.\ref{enu:rem-defFunctCmpt-ordinaryCmpt}
we get
\begin{cor}
\label{cor:intervalsFunctCmpt}If $a$, $b\in\Rtil$ and $a\le b$,
then $[a,b]\fcmp\Rtil$.
\end{cor}

\noindent Let us note that $a$, $b\in\Rtil$ can also be infinite
numbers, e.g.~$a=\diff{\rho}^{-N}$, $b=\diff{\rho}^{-M}$ or $a=-\diff{\rho}^{-N}$,
$b=\diff{\rho}^{-M}$ with $M>N$, so that e.g.~$[-\diff{\rho}^{-N},\diff{\rho}^{M}]\supseteq\R$.
Finally, in the following result we consider the product of functionally
compact sets:
\begin{thm}
\noindent \label{thm:product}Let $K\fcmp\Rtil^{n}$ and $H\fcmp\Rtil^{d}$,
then $K\times H\fcmp\Rtil^{n+d}$. In particular, if $a_{i}\le b_{i}$
for $i=1,\ldots,n$, then $\prod_{i=1}^{n}[a_{i},b_{i}]\fcmp\Rtil^{n}$.
\end{thm}

A theory of compactly supported GSF has been developed in \cite{GiKu16},
and it closely resembles the classical theory of LF-spaces of compactly
supported smooth functions. It establishes that for suitable functionally
compact subsets, the corresponding space of compactly supported GSF
contains extensions of all Colombeau generalized functions, and hence
also of all Schwartz distributions.

As in the classical case (see e.g.~\cite{GeFo00}), thanks to the
extreme value theorem \ref{thm:extremeValues} and the property of
functionally compact sets $K$, we can naturally define a topology
on the space $\gckf(K,\rcrho^{d})$:
\begin{defn}
\label{def:genNormsSpaceGSF}Let $K\fcmp\rcrho^{n}$ be a functionally
compact set such that $K=\overline{\accentset{\circ}{K}}$ (so that
partial derivatives at sharply boundary points can be defined as limits
of partial derivatives at sharply interior points; such $K$ are called
\emph{solid} sets). Let $l\in\N_{\le k}$ and $v\in\gckf(K,\rcrho^{d})$.
Then 
\[
\Vert v\Vert_{l}:=\max_{\substack{|\alpha|\le l\\
1\le i\le d
}
}\max\left(\left|\partial^{\alpha}v^{i}(M_{ni})\right|,\left|\partial^{\alpha}v^{i}(m_{ni})\right|\right)\in\rcrho,
\]
where $M_{ni}$, $m_{ni}\in K$ satisfy
\[
\forall x\in K:\ \partial^{\alpha}v^{i}(m_{ni})\le\partial^{\alpha}v^{i}(x)\le\partial^{\alpha}v^{i}(M_{ni}).
\]
\end{defn}

\noindent The following result (see \cite{GIO} and \cite{ArGaJu09}
for a similar approach) permits us to calculate the (generalized)
norm $\Vert v\Vert_{l}$ using any net $(v_{\eps})$ that defines
$v$.
\begin{lem}
\label{lem:normSpaceGSF}Under the assumptions of Def.~\ref{def:genNormsSpaceGSF},
let $[K_{\eps}]=K\fcmp\rcrho^{n}$ be any representative of $K$.
Then we have:
\begin{enumerate}
\item \label{enu:normAndDefNet}If the net $(v_{\eps})$ defines $v$, then
$\Vert v\Vert_{l}=\left[\max_{\substack{|\alpha|\le l\\
1\le i\le d
}
}\max_{x\in K_{\eps}}\left|\partial^{\alpha}v_{\eps}^{i}(x)\right|\right]\in\rcrho$;
\item \label{enu:normPos}$\Vert v\Vert_{l}\ge0$;
\item $\Vert v\Vert_{l}=0$ if and only if $v=0$;
\item $\forall c\in\rcrho:\ \Vert c\cdot v\Vert_{l}=|c|\cdot\Vert v\Vert_{l}$;
\item \label{enu:normTriang}For all $u\in\gckf(K,\rcrho^{d})$, we have
$\Vert u+v\Vert_{l}\le\Vert u\Vert_{l}+\Vert v\Vert_{l}$ and $\Vert u\cdot v\Vert_{l}\le c_{l}\cdot\Vert u\Vert_{l}\cdot\Vert v\Vert_{l}$
for some $c_{l}\in\rcrho_{>0}$.
\end{enumerate}
\end{lem}

\noindent Using these $\rcrho$-valued norms, we can naturally define
a topology on the space $\gckf(K,\rcrho^{d})$.
\begin{defn}
\label{def:sharpTopSpaceGSF}Let $K\fcmp\rcrho^{n}$ be a solid set.
Let $l\in\N_{\le k}$, $u\in\gckf(K,\rcrho^{d})$, $r\in\rcrho_{>0}$,
then
\begin{enumerate}
\item $B_{r}^{l}(u):=\left\{ v\in\gckf(K,\rcrho^{d})\mid\Vert v-u\Vert_{l}<r\right\} $
\item If $U\subseteq\gckf(K,\rcrho^{d})$, then we say that $U$ is a \emph{sharply
open set} if 
\[
\forall u\in U\,\exists l\in\N_{\le k}\,\exists r\in\rcrho_{>0}:\ B_{r}^{l}(u)\subseteq U.
\]
\end{enumerate}
\end{defn}

\noindent One can easily prove that sharply open sets form a sequentially
Cauchy complete topology on $\gckf(K,\rcrho^{d})$, see e.g.~\cite{GiKu18,baglinipaolo}.
The structure $\left(\gckf(K,\rcrho^{d}),\left(\norm{-}_{l}\right)_{l\le k}\right)$
has the usual properties of a graded Fr�chet space if we replace everywhere
the field $\R$ with the ring $\rcrho$, and for this reason it is
called an $\rcrho$-graded Fr�chet space.

\subsubsection{Multidimensional integration}

Finally, to deal with higher-order calculus of variations, we have
to introduce multidimensional integration of GSF on suitable subsets
of $\rcrho^{n}$ (see \cite{GIO1}).
\begin{defn}
\label{def:intOverCompact}Let $\mu$ be a measure on $\R^{n}$ and
let $K$ be a functionally compact subset of $\RC{\rho}^{n}$. Then,
we call $K$ $\mu$-\emph{measurable} if the limit 
\begin{equation}
\mu(K):=\lim_{m\to\infty}[\mu(\overline{\Eball}_{\rho_{\eps}^{m}}(K_{\eps}))]\label{eq:muMeasurable}
\end{equation}
exists for some representative $(K_{\eps})$ of $K$. Here $m\in\N$,
the limit is taken in the sharp topology on $\RC{\rho}$ since $[\mu(\overline{\Eball}_{\rho_{\eps}^{m}}(K_{\eps}))]\in\rcrho$,
and $\overline{\Eball}_{r}(A):=\{x\in\R^{n}:d(x,A)\le r\}$.
\end{defn}

\noindent In the following result, we show that this definition generates
a correct notion of multidimensional integration for GSF.
\begin{thm}
\label{thm:muMeasurableAndIntegral}Let $K\subseteq\RC{\rho}^{n}$
be $\mu$-measurable.
\begin{enumerate}
\item \label{enu:indepRepr}The definition of $\mu(K)$ is independent of
the representative $(K_{\eps})$.
\item \label{enu:existsRepre}There exists a representative $(K_{\eps})$
of $K$ such that $\mu(K)=[\mu(K_{\eps})]$.
\item \label{enu:epsWiseDefInt}Let $(K_{\eps})$ be any representative
of $K$ and let $f\in\gsf(K,\RC{\rho})$ be a GSF defined by the net
$(f_{\eps})$. Then 
\[
\int_{K}f\,\diff{\mu}:=\lim_{m\to\infty}\biggl[\int_{\overline{\Eball}_{\rho_{\eps}^{m}}(K_{\eps})}f_{\eps}\,\diff{\mu}\biggr]\in\rcrho
\]
exists and its value is independent of the representative $(K_{\eps})$.
\item \label{enu:existsReprDefInt}There exists a representative $(K_{\eps})$
of $K$ such that 
\begin{equation}
\int_{K}f\,\diff{\mu}=\biggl[\int_{K_{\eps}}f_{\eps}\,\diff{\mu}\biggr]\in\rcrho\label{eq:measurable}
\end{equation}
for each $f\in\gsf(K,\RC{\rho})$.
\item If $K=\prod_{i=1}^{n}[a_{i},b_{i}]$, then $K$ is $\lambda$-measurable
($\lambda$ being the Lebesgue measure on $\R^{n}$) and for all $f\in\gsf(K,\RC{\rho})$
we have
\[
\int_{K}f\,\diff{\lambda}=\left[\int_{a_{1,\eps}}^{b_{1,\eps}}\,dx_{1}\dots\int_{a_{n,\eps}}^{b_{n,\eps}}f_{\eps}(x_{1},\dots,x_{n})\,\diff{x_{n}}\right]\in\rcrho
\]
for any representatives $(a_{i,\eps})$, $(b_{i,\eps})$ of $a_{i}$
and $b_{i}$ respectively. Therefore, if $n=1$, this notion of integral
coincides with that of Thm.~\ref{thm:existenceUniquenessPrimitives}
and Def.~\ref{def:integral}.
\item Let $K\subseteq\RC{\rho}^{n}$ be $\lambda$-measurable, where $\lambda$
is the Lebesgue measure, and let $\phi\in\gsf(K,\RC{\rho}^{d})$ be
such that $\phi^{-1}\in\gsf(\phi(K),\RC{\rho}^{n})$. Then $\phi(K)$
is $\lambda$-measurable and 
\[
\int_{\phi(K)}f\,\diff{\lambda}=\int_{K}(f\circ\phi)\left|\det(\diff{\phi})\right|\,\diff{\lambda}
\]
for each $f\in\gsf(\phi(K),\RC{\rho})$.
\end{enumerate}
\end{thm}

\section{\label{sec:MRHO}Higher-order calculus of variations for generalized
functions}

In this section, we prove for GSF the higher-order Euler-Lagrange
equation, the du Bois-Reymond optimality condition and the Noether's
theorem.

\noindent We start with some definition of basic notions and notations.
Using the sharp topology of Def.~\ref{def:sharpTopSpaceGSF}, we
can define when a curve is a minimizer of a given functional. Note
explicitly that there are no restrictions on the generalized numbers
$a$, $b\in\rcrho$, $a<b$, e.g.~they can also both be infinite
numbers.
\begin{defn}
\label{def:minimizer}Let $t_{1}$, $t_{2}\in\rcrho$, with $t_{1}<t_{2}$,
and let $m\in\N_{>0}$, then
\begin{enumerate}
\item For all $q_{t_{1}}:=\left(q_{t_{1}}^{0},\ldots,q_{t_{1}}^{m-1}\right)$,
$q_{t_{2}}:=\left(q_{t_{2}}^{0},\ldots,q_{t_{2}}^{m-1}\right)\in\rcrho^{d\cdot m}$,
we set
\[
\gsf_{\text{bd}}(q_{t_{1}},q_{t_{2}};m):=\left\{ q\in\gsf([t_{1},t_{2}],\rcrho^{d})\mid q^{(i)}(t_{j})=q_{t_{j}}^{i}\ \forall i=0,\ldots,m-1,j=1,2\right\} .
\]
We simply set $\gsf_{0}(m):=\gsf_{\text{bd}}(0,0;m)$, $\gsf_{0}:=\gsf_{\text{bd}}(0,0;1)$
and
\[
\gsf_{0}(t_{1},t_{2}):=\bigcap_{m\in\N}\gsf_{0}(m)=\left\{ q\in\gsf([t_{1},t_{2}],\rcrho^{d})\mid q^{(i)}(t_{j})=0\ \forall i\in\N,j=1,2\right\} .
\]
The subscript ``bd'' stands here for ``boundary values''. Note
explicitly that both $\gsf_{0}(m)$ and $\gsf_{0}(t_{1},t_{2})$ are
$\rti$-submodules of $\left(\rti^{d}\right)^{[t_{1},t_{2}]}$.
\item Let $t_{1}$, $t_{2}\in\rcrho$ with $t_{1}<t_{2}$. Let $q\in\gsf([t_{1},t_{2}],\rcrho^{d})$
and $L\in\gsf([t_{1},t_{2}]\times\rcrho^{d\cdot(m+1)},\rcrho)$ and
define
\begin{align}
[q]^{m}(t):= & \Bigl(t,q(t),q^{(1)}(t),\ldots,q^{(m)}(t)\Bigr)\quad\forall t\in[t_{1},t_{2}]\label{eq:def_I}\\
L[q]^{m}(t):= & L\left(t,q(t),q^{(1)}(t),\ldots,q^{(m)}(t)\right)\quad\forall t\in[t_{1},t_{2}]\\
J[q]:= & \int_{t_{1}}^{t_{2}}L[q]^{m}(t)\,\diff{t}\in\rcrho.\label{eq:Pm}
\end{align}
Note explicitly that Thm.~\ref{enu:category}.\ref{thm:propGSF}
(closure of GSF with respect to composition) and Def.~\ref{def:integral}
of 1-dimensional integral of GSF (i.e.~Thm.~\ref{thm:existenceUniquenessPrimitives})
allows us to say that $J(q)$ is a well-defined number in $\rcrho$.
\item We say that $q$ is a \emph{local minimizer of }$J$ \emph{in} $\gsf_{\text{bd}}(q_{t_{1}},q_{t_{2}};m)$
if $q\in\gsf_{\text{bd}}(q_{t_{1}},q_{t_{2}};m)$ and 
\begin{equation}
\exists r\in\rcrho_{>0}\,\exists l\in\N\,\forall p\in B_{r}^{l}(q)\cap\gsf_{\text{bd}}(q_{t_{1}},q_{t_{2}};m):\ J(p)\ge I(q)\label{eq:defMinimizer}
\end{equation}
\item Let $q\in\gsf\left([t_{1},t_{2}],\rcrho\right)$. We define the \emph{first
}and \emph{second variation }of $J$ \emph{in direction} $h\in\gsf_{0}(m)$
\emph{at} $q$ as
\[
\delta J(q;h):=\left.\frac{\diff{}}{\diff{s}}J(q+sh)\right|_{s=0}\quad\delta^{2}J(q;h):=\left.\frac{\diff{}^{2}}{\diff{^{2}s}}J(q+sh)\right|_{s=0}.
\]
The GSF $q$ is called \emph{weak extremal of $J$ if $\delta J(q;h)=0$}
for all $h\in\gsf_{0}(m)$\emph{.}
\item \label{enu:Dalembert}More generally, if $q\in\gsf\left([t_{1},t_{2}],\rcrho^{d}\right)$
and $Q\in\gsf([t_{1},t_{2}]\times\rcrho^{d\cdot(m+1)},\rcrho)$, we
say that $J$ \emph{satisfies at $q$ the D'Alembert's principle with
generalized forces $Q$} if for all $h\in\gsf_{0}(t_{1},t_{2})$:
\[
\delta J(q;h)=\int_{t_{1}}^{t_{2}}h(t)\cdot Q[q]^{m}(t)\,\diff{t}.
\]
\end{enumerate}
\end{defn}

The following results establish classical necessary and sufficient
conditions to decide if a function $u$ is a local minimizer for the
functional \eqref{eq:Pm}.
\begin{thm}
\label{thm:necessCondsForMinimizer}Let $t_{1}$, $t_{2}\in\rti$
with $t_{1}<t_{2}$, let $L\in\gsf([t_{1},t_{2}]\times\rcrho^{d\cdot(m+1)},\rcrho)$,
let $q_{t_{1}}$, $q_{t_{2}}\in\rti^{d}$ and let $q$ be a local
minimizer of $J$ in $\gsf_{\text{\emph{bd}}}(q_{t_{1}},q_{t_{2}};m)$.
Then

\begin{enumerate}
\item $\delta J(q;h)=0$ for all $h\in\gsf_{0}(m)$;
\item \label{enu:2ndVarPosNec}$\delta^{2}J(q;h)\geq0$ for all $h\in\gsf_{0}(m)$.
\end{enumerate}
\end{thm}

\begin{proof}
Let $r\in\rti_{>0}$ be such that \eqref{eq:defMinimizer} holds.
Since $h\in\gsf_{0}(m)$, the map $s\in\rti\mapsto q+sh\in\gsf_{\text{bd}}(q_{t_{1}},q_{t_{2}};m)$
is well defined and continuous with respect to the trace of the sharp
topology in its codomain. Therefore, we can find $\bar{r}\in\rti_{>0}$
such that $q+sh\in B_{r}^{l}(q)\cap\gsf_{\text{bd}}(q_{t_{1}},q_{t_{2}};m)$
for all $s\in B_{\bar{r}}(0)$. We hence have $J(q+sh)\ge J(q)$.
This shows that the GSF $s\in B_{\bar{r}}(0)\mapsto J(q+sh)\in\rti$
has a local minimum at $s=0$. Now, we apply Lem.~\ref{lem:local_min_diff}
and thus the claims are proven.
\end{proof}
\begin{thm}
\label{thm:suffCondsForMinimizer}Let $t_{1}$, $t_{2}\in\rti$ with
$t_{1}<t_{2}$ and $q_{t_{1}}$, $q_{t_{2}}\in\rti^{d}$. Let $q\in\gsf_{\text{\emph{bd}}}(q_{t_{1}},q_{t_{2}};m)$
be such that

\begin{enumerate}
\item \label{enu:1stVarZero}$\delta J(q;h)=0$ for all $h\in\gsf_{0}(m)$.
\item \label{enu:2ndVarZero}$\delta^{2}J(v;h)\geq0$ for all $h\in\gsf_{0}(m)$
and for all $v\in B_{r}^{l}(q)\cap\gsf_{\text{\emph{bd}}}(q_{t_{1}},q_{t_{2}};m)$,
where $r\in\rti_{>0}$ and $l\in\N$.
\end{enumerate}
Then $q$ is a local minimizer of the functional $J$ in $\gsf_{\text{\emph{bd}}}(q_{t_{1}},q_{t_{2}};m)$.
Moreover, if $\delta^{2}J(v;h)>0$ for all $h\in\gsf_{0}(m)$ such
that $\Vert h\Vert_{l}>0$ and for all $v\in B_{2r}^{l}(u)\cap\gsf_{\text{\emph{bd}}}(q_{t_{1}},q_{t_{2}};m)$,
then $J(v)>J(q)$ for all $v\in B_{r}^{l}(q)\cap\gsf_{\text{\emph{bd}}}(q_{t_{1}},q_{t_{2}};m)$
such that $\Vert v-q\Vert_{l}>0$.

\end{thm}

\begin{proof}
For any $v\in B_{r}^{l}(q)\cap\gsf_{\text{bd}}(q_{t_{1}},q_{t_{2}};m)$,
we set $\psi(s):=J(q+s(v-q))\in\rti$ for all $s\in B_{1}(0)$, so
that $q+s(v-q)\in B_{r}^{l}(q)$. Since $(v-q)(t_{1})=0=(v-q)(t_{2})$,
we have $v-q\in\gsf_{0}(m)$, and properties \ref{enu:1stVarZero},
\ref{enu:2ndVarZero} yield $\psi'(0)=\delta J(q;v-q)=0$ and $\psi''(s)=\delta^{2}J(q+s(v-q);v-q)\ge0$
for all $s\in B_{1}(0)$. We claim that $s=0$ is a minimum of $\psi$.
In fact, for all $s\in B_{1}(0)$ by Taylor's Thm.~\ref{thm:Taylor}
\[
\psi(s)=\psi(0)+s\psi'(0)+\frac{s^{2}}{2}\psi''(\xi)
\]
for some $\xi\in[0,s]$. But $\psi'(0)=0$ and hence $\psi(s)-\psi(0)=\frac{s^{2}}{2}\psi''(\xi)\ge0$.
Finally, since by Thm.~\ref{thm:strongMembershipAndDistanceComplement}.\ref{enu:internalAreClosed}
the interval $[0,+\infty)$ is closed in the sharp topology
\[
\lim_{s\to1^{-}}\psi(s)=J(v)\ge\psi(0)=J(u),
\]
which is our conclusion. Note explicitly that if $\delta^{2}J(v;h)=0$
for all $h\in\gsf_{0}(m)$ and for all $v\in B_{r}^{l}(u)\cap\gsf_{\text{bd}}(q_{t_{1}},q_{t_{2}};m)$,
then $\psi''(\xi)=0$ and hence $J(v)=J(q)$.

Now, assume that $\delta^{2}J(v;h)>0$ for all $h\in\gsf_{0}(m)$
such that $\Vert h\Vert_{l}>0$ and for all $v\in B_{2r}^{l}(u)\cap\gsf_{\text{bd}}(q_{t_{1}},q_{t_{2}};m)$,
and take $v\in B_{r}^{l}(u)\cap\gsf_{\text{bd}}(q_{t_{1}},q_{t_{2}};m)$
such that $\Vert v-q\Vert_{l}>0$. As above set $\psi(s):=J(q+s(v-q))\in\rti$
for all $s\in B_{3/2}(0)$, so that $q+s(v-q)\in B_{2r}^{l}(q)$.
We have $\psi'(0)=0$ and $\psi''(s)=\delta^{2}J(q+s(v-q);v-q)>0$
for all $s\in B_{3/2}(0)$ because $\Vert v-u\Vert_{l}>0$. Using
Taylor's theorem, we get $\psi(1)=\psi(0)+\frac{1}{2}\psi''(\xi)$
for some $\xi\in[0,1]$. Therefore $\psi(1)-\psi(0)=J(v)-J(q)=\frac{1}{2}\psi''(\xi)>0$.
\end{proof}
In the following, we always assume the notations of Def.~\ref{def:minimizer}.
Moreover, we also set
\[
\rcrho_{d}[t]:=\left\{ p\in\rcrho[t]\mid\deg(p)\le d\right\} 
\]
for the set of all the polynomials with coefficients in the ring $\rcrho$
having degree less of equal to $d\in\N$.

In order to prove the higher-order version of the fundamental lemma,
we need the following preliminary results:
\begin{lem}
\label{lem:higherFundLem}Let $a$, $b\in\rcrho$ be such that $a<b$,
and let $f\in\zs\left([a,b],\gf_{\ge0}\right)$. If $\int_{a}^{b}f(t)\,\diff{t}=0$,
then $f=0$.
\end{lem}

\begin{proof}
By contradiction, assume that $f(x)\ne0$ for some $x\in[a,b]$. Let
$(f_{\eps})$ be a net of smooth functions that define the GSF $f$
and let $[x_{\eps}]=x$ be a representative of $x$. The property
\begin{equation}
\forall q\in\N\,\forall^{0}\eps:\ f_{\eps}(x_{\eps})\le\rho_{\eps}^{q}\label{eq:contr1}
\end{equation}
implies $f(x)\le\diff{\rho}^{q}$ and hence also, letting $q\to+\infty$,
$f(x)\le0$. Since $f(x)\ge0$, this would imply $f(x)=0$. Therefore,
taking the negation of \eqref{eq:contr1} we get
\[
\exists q\in\N\,\exists L\subseteq I:\ 0\in\overline{L},\ \forall\eps\in L:\ f_{\eps}(x_{\eps})>\rho_{\eps}^{q},
\]
where $\overline{L}$ is the closure of $L$ in $I$. Set $\bar{f}_{\eps}:=f_{\eps}$
for $\eps\in L$ and $\bar{f}_{\eps}:=\frac{1}{2}\left(f_{\eps}+\rho_{\eps}^{q}\right)$
otherwise. Directly from Def.~\ref{def:netDefMap}, it follows that
$\bar{f}:=\left[\bar{f}_{\eps}(-)\right]\in\gsf([a,b],\rcrho_{\ge0})$
and $\bar{f}(x)>\diff{\rho}^{q+1}$. The sharp continuity of $\bar{f}$
at $x$ (Thm.~\ref{thm:propGSF}.\ref{enu:GSF-cont}) implies
\begin{equation}
\exists c,d:\ c<d,\ [c,d]\subseteq[a,b],\ \bar{f}|_{[c,d]}>\diff{\rho}^{q+1}.\label{eq:cont}
\end{equation}
Moreover, from $f\ge0$ and Thm\@.~\ref{thm:intRules} we obtain
\[
0=\int_{a}^{b}f\ge\int_{a}^{c}f+\int_{c}^{d}f+\int_{d}^{b}f\ge\int_{c}^{d}f.
\]
Thereby, we can find a negligible net $[z_{\eps}]=0$ such that $z_{\eps}\ge\int_{c_{\eps}}^{d_{\eps}}f_{\eps}$
for all $\eps$ sufficiently small, where $[c_{\eps}]=c$, $[d_{\eps}]=d$
with $c_{\eps}<d_{\eps}$ because $c<d$ (see Lem.~\ref{lem:mayer}).
In particular, for $\eps\in L$ sufficiently small $f_{\eps}=\bar{f}_{\eps}$
and hence $\int_{c_{\eps}}^{d_{\eps}}f_{\eps}=\int_{c_{\eps}}^{d_{\eps}}\bar{f}_{\eps}\le z_{\eps}$.
But then \eqref{eq:cont} and Thm\@.~\ref{thm:intRules}.\ref{enu:intMonotone}
yield
\[
\forall^{0}\eps\in L:\ z_{\eps}\ge\int_{c_{\eps}}^{d_{\eps}}\bar{f}_{\eps}\ge\int_{c_{\eps}}^{d_{\eps}}\rho_{\eps}^{q+1}=\rho_{\eps}^{q+1}\left(d_{\eps}-c_{\eps}\right).
\]
Therefore, $\rho_{\eps}^{q+1}\le\frac{z_{\eps}}{d_{\eps}-c_{\eps}}$.
However, $\left[\frac{z_{\eps}}{d_{\eps}-c_{\eps}}\right]=\frac{[z_{\eps}]}{d-c}=0$
and hence $\rho_{\eps}^{q+1}\le\rho_{\eps}^{q+2}$ by \eqref{eq:negligible},
a contradiction.
\end{proof}
The second preliminary result introduces the use of approximate identities
for convolution with GSF (see also \cite[Lem.~4.3]{GIO}):
\begin{lem}
\label{lem:approxId}Let $a$, $b\in\rti$ be such that $a<b$ and
let $f\in\gsf([a,b],\rti)$. Let $x\in[a,b]$ and $R\in\rti_{>0}$
be such that $B_{R}(x)\subseteq[a,b]$. Only in this statement, we
write $\forall^{0}t\in\rti_{>0}$ to denote $\exists r\in\rti_{>0}\,\forall t\in B_{r}(0)\cap\rti_{>0}$,
i.e.~``for $t$ sufficiently small (in the sharp topology)''. Assume
that $G_{t}\in\gsf(\rti,\rti)$ satisfy 
\begin{enumerate}
\item \label{enu:int1}$\forall^{0}t\in\rti_{>0}:\ \int_{-R}^{R}G_{t}=1$. 
\item \label{enu:G_tZero}For $t$ small, $(G_{t})_{t\in\rti_{>0}}$ is
zero outside every ball $B_{\delta}(0)$, $0<\delta<R$, i.e. 
\begin{equation}
\forall\delta\in\rti_{>0}\,\forall^{0}t\in\rti_{>0}\,\forall y\in[-R,-\delta]\cup[\delta,R]:\ G_{t}(y)=0.\label{eq:G_tIsZero}
\end{equation}
\item \label{enu:delta_bd} $\exists M\in\rti_{>0}\,\forall^{0}t\in\rti_{>0}\colon\ \int_{-R}^{R}\left|G_{t}(y)\right|\,\diff{y}\leq M$. 
\end{enumerate}
Then 
\[
\lim_{t\to0^{+}}\int_{-R}^{R}f(x-y)G_{t}(y)\diff{y}=f(x).
\]
Moreover $\int_{-R}^{R}f(x-y)G_{t}(y)\diff{y}=\int_{x-R}^{x+R}f(y)G_{t}(x-y)\diff{y}$. 
\end{lem}

\begin{proof}
We only have to generalize the classical proof concerning limits of
convolutions with strict delta nets. We first note that 
\[
\int_{-R}^{R}f(x-y)G_{t}(y)\diff{y}=\int_{x-R}^{x+R}f(y)G_{t}(x-y)\diff{y}
\]
so that these integrals exist because $(x-R,x+R)=B_{R}(x)\subseteq[a,b]$.
Using \ref{enu:int1}, for $t$ small, let's say for $0<t<S\in\rti_{>0}$,
we get 
\begin{align*}
\left|\int_{-R}^{R}f(x-y)G_{t}(y)\diff{y}-f(x)\right| & =\left|\int_{-R}^{R}\left[f(x-y)-f(x)\right]G_{t}(y)\diff{y}\right|\\
 & \le\int_{-R}^{R}\left|f(x-y)-f(x)\right|\cdot\left|G_{t}(y)\right|\diff{y}.
\end{align*}
For each $r\in\rti_{>0}$, sharp continuity of $f$ at $x$ yields
$\left|f(x-y)-f(x)\right|<r$ for all $y$ such that $|y|<\delta\in\rti_{>0}$,
and we can take $\delta<R$. Assuming that \ref{enu:G_tZero} holds
for all $0<t<s$, for $0<t<s\wedge S$, we have 
\begin{equation}
\left|\int_{-R}^{R}f(x-y)G_{t}(y)\diff{y}-f(x)\right|\le r\int_{-\delta}^{+\delta}\left|G_{t}(y)\right|\diff{y}\le rM,\label{eq:rTimesIntegG}
\end{equation}
where we used \ref{enu:delta_bd}. The right hand side of \eqref{eq:rTimesIntegG}
can be taken arbitrarily small in $\rti_{>0}$ because it holds for
all $r\in\rti_{>0}$.
\end{proof}

\subsection{Higher-order Euler\textendash Lagrange equation, D'Alembert principle
and du Bois\textendash Rey\-mond optimality condition}

In this section, we first compute the first variation $\delta J(q;h)$
and thereby we deduce the Euler-Lagrange equations both in integral
and differential forms, and the D'Alembert principle in differential
form.
\begin{lem}
\label{Lem:firstVarInt}If $q\in\gsf\left([t_{1},t_{2}],\rcrho\right)$
and $h\in\gsf_{0}(m)$, then
\begin{align*}
\delta J(q;h) & =\int_{t_{1}}^{t_{2}}h^{(m)}(t)\cdot\Biggl[\sum\limits _{i=0}^{m}(-1)^{i}\Biggl(\underbrace{\int_{t_{1}}^{t}\int_{t_{1}}^{s_{1}}\dots\int_{t_{1}}^{s_{m-i-1}}}_{m-i~\textnormal{times}}\Bigl(\partial_{i+2}L[q]^{m}(s_{m-i})\Bigr)\,\diff{s_{m-i}}\dots\diff{s_{1}}\Biggr)\Biggr]\,\diff{t}\\
 & =\int_{t_{1}}^{t_{2}}h(t)\cdot\left[\sum\limits _{i=0}^{m}(-1)^{i}\frac{\diff{^{i}}}{\diff{t}^{i}}\partial_{i+2}L[q]^{m}(t)\right].
\end{align*}
\end{lem}

\begin{proof}
According to Def.~\ref{def:minimizer}, and to Thm\@.~\ref{thm:intRules}
for any $h\in\gsf_{0}(m)$ we have 
\begin{equation}
\delta J(q;h)=\int_{t_{1}}^{t_{2}}\left(\sum_{i=0}^{m}\partial_{i+2}L[q]^{m}(t)\cdot h^{(i)}(t)\right)\,\diff{t}.\label{pel}
\end{equation}
By repeated integration by parts (Thm\@.~\ref{thm:intRules}.\ref{enu:intByParts})
one has 
\begin{multline}
\sum\limits _{i=0}^{m}\int_{t_{1}}^{t_{2}}\partial_{i+2}L[q]^{m}(t)\cdot h^{(i)}(t)\,\diff{t}\\
=\sum\limits _{i=0}^{m}\Biggl\{\Biggl[\sum\limits _{j=1}^{m-i}(-1)^{j+1}h^{(i+j-1)}(t)\cdot\Biggl(\underbrace{\int_{t_{1}}^{t}\int_{t_{1}}^{s_{1}}\dots\int_{t_{1}}^{s_{j-1}}}_{j~\textnormal{times}}\Bigl(\partial_{i+2}L[q]^{m}(s_{j})\Bigr)\,\diff{s_{j}}\dots\diff{s_{2}}\,\diff{s_{1}}\Biggr)\Biggr]_{t_{1}}^{t_{2}}\\
+(-1)^{i}\int_{t_{1}}^{t_{2}}h^{(m)}(t)\cdot\Biggl(\underbrace{\int_{t_{1}}^{t}\int_{t_{1}}^{s_{1}}\dots\int_{t_{1}}^{s_{m-i-1}}}_{m-i~\textnormal{times}}\Bigl(\partial_{i+2}L[q]^{m}(s_{m-i})\Bigr)\,\diff{s_{m-i}}\dots\diff{s_{2}}\,\diff{s_{1}}\Biggr)\,\diff{t}\Biggr\}\label{eq:identity1}
\end{multline}
Since $h^{(i)}(t_{1})=0=h^{(i)}(t_{2})$, $i=0,\ldots,m-1$, the first
summands in \eqref{eq:identity1} vanish. Therefore, equations \eqref{pel}
and \eqref{eq:identity1} yield the first conclusion. The second equality
simply follows by $m$ times integration by parts.
\end{proof}
In order to prove the Euler-Lagrange equation in integral form (see
below Cor.~\ref{Cor:ELdeordm}), we first need to show the fundamental
lemma and the higher order du Bois\textendash Reymond lemma of the
calculus of variations for generalized functions.
\begin{lem}[Fundamental Lemma of the Calculus of Variations]
\label{lem:fund_lem_calc_var}Let $a$, $b\in\rti$ such that $a<b$,
and let $f\in\gsf([a,b],\rti)$. If 
\begin{equation}
\int_{a}^{b}f(t)h(t)\,\diff{t}=0\:\text{ for all }\:h\in\gsf_{0}(a,b),\label{eq:HpFundLem}
\end{equation}
then $f=0$. 
\end{lem}

\begin{proof}
Let $x\in[a,b]$. Because of Thm.~\ref{thm:propGSF}.\ref{enu:GSF-cont}
and Lem.~\ref{lem:approxOfBoundaryPointsWithInterior}, without loss
of generality we can assume that $x$ is a sharply interior point,
so that $B_{R}(x)\subseteq[a,b]$ for some $R\in\rti_{>0}$. Let $\phi\in\D_{[-1,1]}(\R)$
be such that $\int\phi=1$. Set $G_{t,\eps}(x):=\frac{1}{t_{\eps}}\phi\left(\frac{x}{t_{\eps}}\right)$,
where $x\in\R$ and $t\in\rti_{>0}$, and $G_{t}(x):=[G_{t,\eps}(x_{\eps})]$
for all $x\in\rti$. Then, for $t$ sufficiently small, we have $G_{t}(x-.)\in\gsf_{0}(a,b)$
and \eqref{eq:HpFundLem} yields $\int_{a}^{b}f(y)G_{t}(x-y)\diff{y}=0$.
For $t$ small, we both have that $G_{t}(x-.)=0$ on $[a,x-R]\cup[x+R,b]$
and the assumptions of Lem.~\ref{lem:approxId} hold. Therefore 
\begin{align*}
0 & =\int_{a}^{b}f(y)G_{t}(x-y)\diff{y}=\int_{x-R}^{x+R}f(y)G_{t}(x-y)\diff{y}=\\
 & =\int_{-R}^{R}f(x-y)G_{t}(y)\diff{y},
\end{align*}
and Lem.~\ref{lem:approxOfBoundaryPointsWithInterior} hence yields
$f(x)=0$. 
\end{proof}
\begin{lem}[Fundamental higher order du Bois\textendash Reymond lemma]
\label{lem:hodr}Let $a$, $b\in\rcrho$ be such that $a<b$, and
let $f\in\zs\left([a,b],\gf\right)$. If 
\begin{equation}
\int_{a}^{b}f(t)h^{(m)}(t)\,\diff{t}=0\:\text{ for all }\:h\in\gsf_{0}(a,b)\label{hodl}
\end{equation}
Then $f(t)\in\rcrho_{m-1}[t]$.
\end{lem}

\begin{proof}
For all $h\in\gsf_{0}(a,b)$, using $m$ times integration by parts
in \eqref{hodl}, by $h^{(i)}(a)=h^{(i)}(b)=0$ for all $i=0,\ldots,m-1$
we get $\int_{a}^{b}f^{(m)}(t)h(t)\,\diff{t}=0$. Thereby, the fundamental
Lem.~\ref{lem:fund_lem_calc_var} yields $f^{(m)}=0$. Integrating
$m$ times (see Thm.~\ref{thm:existenceUniquenessPrimitives}), we
get the claim $f(t)\in\rcrho_{m-1}[t]$.
\end{proof}
\noindent As a simple consequence, we have the following
\begin{cor}
\label{cor:fundLem}Let $a$, $b\in\rcrho$ be such that $a<b$, and
let $f$, $g\in\gsf([a,b],\rcrho)$. If
\[
\int_{a}^{b}\left(f(t)h(t)+g(t)h'(t)\right)\,\diff{t}=0\:\text{ for all }\:h\in\gsf_{0}(a,b),
\]
then $g'=f$.
\end{cor}

\begin{proof}
Set $F(x):=\int_{a}^{x}f(t)\,\diff{t}$ for $x\in[a,b]$, so that
$F\in\gsf([a,b],\rcrho)$ and $F'=f$ by Def.~\ref{def:integral}.
Integrating by parts (Thm.~\ref{thm:intRules}.\ref{enu:intByParts}),
we get
\[
\int_{a}^{b}\left(f(t)h(t)+g(t)h'(t)\right)\,\diff{t}=\int_{a}^{b}\left(g(t)-F(t)\right)h'(t)\,\diff{t}=0
\]
for all $h\in\gsf_{0}(a,b)$. Therefore, Lem.~\ref{lem:hodr} implies
that $g(t)-F(t)$ is a constant (in $\rcrho$), and hence $g'=F'=f$
as claimed.
\end{proof}
Applying the higher-order du Bois\textendash Reymond Lem.~\ref{lem:hodr},
we arrive at the first form of the Euler-Lagrange equations:
\begin{cor}
\label{Cor:ELdeordm}If $q\in\gsf\left([t_{1},t_{2}],\rcrho\right)$
is a weak extremal of functional \eqref{eq:Pm}, then $q$ satisfies
the following higher-order Euler\textendash Lagrange integral equation:
\begin{equation}
\sum_{i=0}^{m}(-1)^{i}\Biggl(\underbrace{\int_{t_{1}}^{t}\int_{t_{1}}^{s_{1}}\dots\int_{t_{1}}^{s_{m-i-1}}}_{m-i~\textnormal{times}}\Bigl(\partial_{i+2}L[q]^{m}(t)\Bigr)\,\diff{s_{m-i}}\dots\diff{s_{1}}\Biggr)\in\rcrho_{m-1}[t].\label{eq:ELdeordmInt1}
\end{equation}
\end{cor}

\noindent From the second equality of Lem.~\ref{Lem:firstVarInt}
and the fundamental Lem.~\ref{lem:fund_lem_calc_var}, we obtain
\begin{cor}[Higher-order Euler\textendash Lagrange equations in differential form]
\label{cor:16}If the GSF $q\in\gsf\left([t_{1},t_{2}],\rcrho\right)$
is a weak extremal of functional \eqref{eq:Pm}, then 
\begin{equation}
\sum_{i=0}^{m}(-1)^{i}\frac{\diff{^{i}}}{\diff{t}^{i}}\partial_{i+2}L[q]^{m}(t)=0\qquad\forall t\in[t_{1},t_{2}].\label{eq:ELdeordm1}
\end{equation}
\end{cor}

\noindent Finally, Lem.~\ref{Lem:firstVarInt}, Def.~\ref{def:minimizer}.\ref{enu:Dalembert}
and the fundamental Lem.~\ref{lem:fund_lem_calc_var} directly yield
the D'Alembert principle in differential form:
\begin{cor}
\label{cor:DalembertDiff}Let $q\in\gsf\left([t_{1},t_{2}],\rcrho\right)$,
$Q\in\gsf([t_{1},t_{2}]\times\rcrho^{d\cdot(m+1)},\rcrho)$ and assume
that the functional $J$ satisfies at \emph{$q$ }the D'Alembert's
principle with generalized forces \emph{$Q$, then
\[
\sum_{i=0}^{m}(-1)^{i}\frac{\diff{^{i}}}{\diff{t}^{i}}\partial_{i+2}L[q]^{m}(t)=Q[q]^{m}(t)\qquad\forall t\in[t_{1},t_{2}].
\]
}
\end{cor}

Associated to a given function $q\in\gsf\left([t_{1},t_{2}],\rcrho\right)$,
it is convenient to introduce the following quantities (see \cite{Torres:proper}):
\begin{equation}
\varphi^{j}(q)(t):=\sum_{i=0}^{m-j}(-1)^{i}\frac{\diff{^{i}}}{\diff{t}^{i}}\partial_{i+j+2}L[q]^{m}(t)\quad\forall j=0,\ldots,m\,\forall t\in[t_{1},t_{2}]\label{eq:eqprin11}
\end{equation}
These operators are useful for our purposes because of the following
properties: 
\begin{equation}
\frac{\diff{}}{\diff{t}}\varphi^{j}(q)=\partial_{j+1}L[q]^{m}-\varphi^{j-1}(q)\quad\forall j=1,\ldots,m.\label{dr}
\end{equation}

\noindent We are now in conditions to prove a higher-order du Bois\textendash Reymond
optimality condition.
\begin{thm}
\label{thm:cDRifm}If $q\in\zs\left([t_{1},t_{2}],\gfs\right)$ is
a weak extremal of functional \eqref{eq:Pm}, then 
\begin{equation}
\frac{\diff{}}{\diff{t}}\left(L[q]^{m}-\sum_{j=1}^{m}\varphi^{j}(q)\cdot q^{(j)}\right)=\partial_{1}L[q]^{m}.\label{eq:DBRordm:2}
\end{equation}
\end{thm}

\begin{proof}
Using, for simplicity, $\phi^{j}:=\phi^{j}(q)$, we have 
\begin{multline}
\frac{\diff{}}{\diff{t}}\left(L[q]^{m}-\sum_{j=1}^{m}\varphi^{j}\cdot q^{(j)}\right)=\\
=\left(\partial_{1}L[q]^{m}+\sum_{j=0}^{m}\partial_{j+2}L[q]^{m}\cdot q^{(j+1)}-\sum_{j=1}^{m}\left(\dot{\varphi}^{j}\cdot q^{(j)}+\varphi^{j}\cdot q^{(j+1)}\right)\right).\label{pr}
\end{multline}
Using \eqref{dr}, equation \eqref{pr} becomes 
\begin{multline}
\frac{d}{dt}\left(L[q]^{m}-\sum_{j=1}^{m}\varphi^{j}\cdot q^{(j)}\right)=\partial_{1}L[q]^{m}+\sum_{j=0}^{m}\partial_{j+2}L[q]^{m}\cdot q^{(j+1)}\\
-\sum_{j=1}^{m}\left(\left(\partial_{j+1}L[q]^{m}-\varphi^{j-1}\right)\cdot q^{(j)}+\varphi^{j}\cdot q^{(j+1)}\right).\label{pr2}
\end{multline}

\noindent We now simplify the second term on the right-hand side of
\eqref{pr2}, which forms a telescopic sum: 
\begin{multline}
\sum_{j=1}^{m}\left(\left(\partial_{j+1}L[q]^{m}-\varphi^{j-1}\right)\cdot q^{(j)}+\varphi^{j}\cdot q^{(j+1)}\right)=\\
=\sum_{j=0}^{m-1}\Bigl(\left(\partial_{j+2}L[q]^{m}-\varphi^{j}\right)\cdot q^{(j+1)}+\varphi^{j+1}\cdot q^{(j+2)}\Bigr)=\\
=\sum_{j=0}^{m-1}\left(\partial_{j+2}L[q]^{m}\cdot q^{(j+1)}\right)-\varphi^{0}\cdot\dot{q}+\varphi^{m}\cdot q^{(m+1)}.\label{pr3}
\end{multline}

\noindent Substituting \eqref{pr3} into \eqref{pr2} and using the
higher-order Euler\textendash Lagrange equations \eqref{eq:ELdeordm1},
and since, by definition, 
\[
\varphi^{m}=\partial_{m+2}L[q]^{m}
\]
and 
\[
\varphi^{0}=\sum_{i=0}^{m}(-1)^{i}\frac{\diff{}^{i}}{\diff{t}^{i}}\partial_{i+2}L[q]^{m}=0\,,
\]
we obtain the intended result, that is, 
\begin{multline*}
\frac{d}{dt}\left(L[q]^{m}-\sum_{j=1}^{m}\varphi^{j}\cdot q^{(j)}\right)=\\
=\partial_{1}L[q]^{m}+\partial_{m+2}L[q]^{m}\cdot q^{(m+1)}\\
+\varphi^{0}\cdot\dot{q}-\varphi^{m}\cdot q^{(m+1)}=\partial_{1}L[q]^{m}.
\end{multline*}
\end{proof}

\subsection{Higher-order Noether's theorem}

In our approach to Noether's theorem, we can use the non-Archimedean
language of $\rcrho$ to formalize some infinitesimal properties frequently
informally used in this topic.
\begin{defn}
\label{def:sy}Let $D\subseteq\rti^{k}$ be a sharply open set. We
say that $\Psi=\left\{ \psi(s,\cdot)\right\} _{s\in P}\in\gsf(D,D)$
is a \emph{one parameter group of diffeomorphisms of }$D$ if it satisfies:
\begin{enumerate}
\item \label{enu:diffGsf}$P$ is a sharply open set, $0\in P$ and $\psi\in\gsf(P\times D,D)$;
\item \label{enu:diffDiff}For each $s\in P$, the map $\psi(s,\cdot)\in\gsf(D,D)$
is invertible, and $\psi(s,\cdot)^{-1}\in\gsf(D,D)$;
\item \label{enu:diffId}$\psi(0,\cdot)=\text{Id}_{D}$;
\item \label{enu:diffAdd}$\forall s,s'\in P:\ s+s'\in P\ \Rightarrow\ \psi(s,\cdot)\circ\psi(s',\cdot)=\psi(s+s',\cdot)$.
\end{enumerate}
\end{defn}

\noindent We will see on Sec.~\ref{subsec:Pais=002013Uhlenbeck-oscillator}
the actual usefulness of the case of strict inclusion $D\subset\rti^{k}$.
Note that, because of properties \ref{enu:diffGsf} and \ref{enu:diffId},
from Taylor's formula Thm.~\ref{thm:Taylor}, we get
\[
\exists S\in\rcrho_{>0}\,\forall s\in[0,S]\cap P:\ \psi(s,q(t))\approx q(t)+s\frac{\partial\psi}{\partial s}(0,q(t)).
\]
Therefore, for a sufficiently small infinitesimal $s$, we can always
say that $\psi(s,t)$ is infinitely close to a transformation of the
form $q(t)\mapsto q(t)+s\eta(q(t))$, where $\eta(q):=\frac{\partial\psi}{\partial s}(0,q)$.
To write an equality sign instead of the infinitely close sign $\approx$,
we can use Thm\@.~\ref{thm:FR-forGSF} and the notation $\frac{\partial\psi}{\partial s}[0,q(t);s]$
for the generalized partial incremental ratio with respect to the
variable $s$ (we recall that it is another GSF):
\begin{align*}
\exists S\in\rcrho_{>0}\,\forall s\in[0,S]\cap P:\ \psi(s,q(t))=q(t)+s\frac{\partial\psi}{\partial s}[0,q(t);s]\\
\frac{\partial\psi}{\partial s}[0,q(t);s]\approx\frac{\partial\psi}{\partial s}(0,q(t)).
\end{align*}

\begin{defn}
\label{def:inva1}Let both $T=\left\{ \tau(s,\cdot)\right\} _{s\in P}\in\gsf(T',T')$
and $S=\left\{ \sigma(s,\cdot)\right\} _{s\in P}\in\gsf(S',S')$ be
one parameter groups of diffeomorphims on the open sets $T'\subseteq[t_{1},t_{2}]$
and $S'\subseteq\rti^{d}$ respectively. The functional \eqref{eq:Pm}
is said to be \emph{invariant under the action of }$T$ and $S$,
if for any weak extremal $q\in\zs\left([t_{1},t_{2}],S'\right)$ it
satisfies
\begin{equation}
L[q]^{m}(t)=L\left(\tau(s,t),\sigma(s,q(t)),\frac{\diff{\sigma}(s,q(t))}{\diff{\tau}(s,t)},\dots,\frac{\diff{^{m}\sigma}(s,q(t))}{\diff{\tau}^{m}(s,t)}\right)\frac{\partial\tau}{\partial t}(s,t)\label{eq:invdf11}
\end{equation}
for all $s\in P$ and all $t\in T'$, where the expressions $\frac{\diff{^{i}\sigma}(s,q(t))}{\diff{\tau}^{i}(s,t)}$,
$i=1,\ldots,m$ are defined in the following Rem.~\ref{rem:derSigTau}.\ref{enu:derSigTau}.
\end{defn}

\begin{rem}
\label{rem:derSigTau}\ 
\begin{enumerate}
\item From Def.~\ref{def:sy}.\ref{enu:diffDiff} we have $\tau\left(s,\tau(s,\cdot)^{-1}(t)\right)=t$,
thereby the chain rule Thm.~\ref{thm:rulesDer} yields that $\frac{\partial\tau}{\partial t}(s,t)$
is invertible for all $s\in P$, $t\in T'$.
\item \label{enu:derSigTau}Based on the previous remark, in \eqref{eq:invdf11}
the expressions $\frac{\diff{^{i}\sigma}(s,q(t))}{\diff{\tau}^{i}(s,t)}$,
$i=1,\ldots,m$ are defined as 
\begin{equation}
\frac{\diff{\sigma}(s,q(t))}{\diff{\tau}(s,t)}:=\frac{\frac{\partial\sigma(s,q(t))}{\partial t}(s,t)}{\frac{\partial\tau}{\partial t}(s,t)},\label{eq:invdf12}
\end{equation}
and 
\begin{equation}
\frac{\diff{^{i}\sigma}(s,q(t))}{\diff{\tau}^{i}(s,t)}:=\frac{\frac{\partial}{\partial t}\left(\frac{\diff{^{i-1}\sigma}(s,q(t))}{\diff{\tau}^{i-1}(s,t)}\right)(s,t)}{\frac{\partial\tau}{\partial t}(s,t)},\quad\forall i=2,\ldots,m.\label{eq:invdf22}
\end{equation}
As usual, because of Def.~\ref{def:sy}.\ref{enu:diffId}, we have
$\frac{\diff{^{i}\sigma}(0,q(t))}{\diff{\tau}^{i}(0,t)}=q^{(i)}(t)$.
\item Using again the generalized incremental ratios, i.e~Thm.~\ref{thm:FR-forGSF},
for all $s\in P$, all $t\in T'$ and for $h\in\rcrho$ sufficiently
small, we have
\begin{align*}
\sigma(s,q(t+h)) & =\sigma(s,q(t))+h\cdot\frac{\partial\sigma(s,q(t))}{\partial t}[s,t;h]\\
\tau(s,t+h) & =\tau(s,t)+h\cdot\frac{\partial\tau}{\partial t}[s,t;h].
\end{align*}
Therefore, the ratio (of differentials):
\[
\frac{\sigma(s,q(t+h))-\sigma(s,q(t))}{\tau(s,t+h)-\tau(s,t)}=\frac{\frac{\partial\sigma(s,q(t))}{\partial t}[s,t;h]}{\frac{\partial\tau}{\partial t}[s,t;h]}\approx\frac{\frac{\partial\sigma(s,q(t))}{\partial t}(s,t)}{\frac{\partial\tau}{\partial t}(s,t)}
\]
for all sufficiently small invertible $h$.
\item On the basis of Def.~\ref{def:integral}, condition \eqref{eq:invdf11}
is equivalent to
\[
\int_{t_{a}}^{t_{b}}L[q]^{m}(t)\,\diff{t}=\int_{t_{a}}^{t_{b}}L\left(\tau(s,t),\sigma(s,q(t)),\frac{\diff{\sigma}(s,q(t))}{\diff{\tau}(s,t)},\dots,\frac{\diff{^{m}\sigma}(s,q(t))}{\diff{\tau}^{m}(s,t)}\right)\frac{\partial\tau}{\partial t}(s,t)\,\diff{t}
\]
for all $s\in P$ and for any subinterval $[{t_{a}},{t_{b}}]\subseteq T'$
with $t_{a}<t_{b}$, even if $L$ can also be a distribution. Clearly,
if $L$ is not a continuous function, this equivalence is due to the
regularization process inherent to the embedding $\iota^{b}:\mathcal{D}'\ra\gsf(\csp{-},\rti)$
of Sec.~\ref{subsec:Embedding}.
\end{enumerate}
\end{rem}

The next lemma gives a necessary condition for the functional \eqref{eq:Pm}
to be invariant under the action of the parameter groups of diffeomorphisms
$T=\left\{ \tau(s,\cdot)\right\} _{s\in P}\in\gsf(T',T')$ and $S=\left\{ \sigma(s,\cdot)\right\} _{s\in P}\in\gsf(S',S')$.
\begin{lem}
\label{lem:cnsi}If the functional \eqref{eq:Pm} is invariant in
the sense of Def.~\ref{def:inva1}, then 
\begin{equation}
\frac{\partial L[q]^{m}}{\partial t}\frac{\partial\tau}{\partial t}(0,t)+\frac{\partial\tau}{\partial t}(0,t)\sum_{i=0}^{m}\frac{\partial L[q]^{m}}{\partial q^{(i)}}\cdot\eta^{i}(q,t)+L[q]^{m}\frac{\diff{}}{\diff{t}}\frac{\partial\tau}{\partial s}(0,t)=0\quad\forall t\in T'\label{civho21}
\end{equation}
for any weak extremal $q\in\zs\left([t_{1},t_{2}],S'\right)$. In
\eqref{civho21}, we set
\begin{equation}
\begin{cases}
\eta^{0}(q,\cdot):=\frac{\partial\sigma}{\partial s}(0,q(\cdot)),\\
\eta^{i}(q,\cdot):=\frac{\diff{\eta^{i-1}}(q,\cdot)}{\diff{t}}-q^{(i)}\frac{\diff{}}{\diff{t}}\frac{\partial\tau}{\partial s}(0,\cdot),\quad\forall i=1,\ldots,m.
\end{cases}\label{civho}
\end{equation}
\end{lem}

\begin{proof}
Differentiating \eqref{eq:invdf11} with respect to $s$ at $s=0\in P$
and $t\in T'$ and using Rem~\ref{rem:derSigTau}.\ref{enu:derSigTau},
we obtain
\begin{equation}
\frac{\partial L[q]^{m}}{\partial t}\frac{\partial\tau}{\partial t}(0,t)+\frac{\partial\tau}{\partial t}(0,t)\sum_{i=0}^{m}\frac{\partial L[q]^{m}}{\partial q^{(i)}}\cdot\frac{\partial}{\partial s}\left.\frac{\diff{^{i}\sigma}(s,q(t))}{\diff{\tau}^{i}(s,t)}\right|_{s=0}+L[q]^{m}(t)\frac{\partial^{2}\tau}{\partial s\partial t}(0,t)=0.\label{civho1}
\end{equation}
Note explicitly that using two times Thm\@.~\ref{thm:FR-forGSF},
for any GSF $f$ we always have
\begin{align*}
\frac{\partial^{2}f}{\partial x\partial y}(x_{0},y_{0}) & =\left[\frac{\partial f_{\eps}^{2}}{\partial x_{\eps}\partial y_{\eps}}(x_{0\eps},y_{0\eps})\right]=\left[\frac{\partial f_{\eps}^{2}}{\partial y_{\eps}\partial x_{\eps}}(x_{0\eps},y_{0\eps})\right]=\\
 & =\frac{\partial^{2}f}{\partial y\partial x}(x_{0},y_{0})=\left.\frac{\diff{}}{\diff{x}}\right|_{x_{0}}\left(\frac{\partial f}{\partial y}(x,y_{0})\right)=\frac{\diff{}}{\diff{x}}\frac{\partial f}{\partial y}(x_{0},y_{0}).
\end{align*}
Using Rem.~\ref{rem:derSigTau}.\ref{enu:derSigTau} and \eqref{civho},
we have 
\begin{align}
\frac{\partial}{\partial s}\left.\left(\frac{\diff{\sigma}(s,q(t))}{\diff{\tau}(s,t)}\right)\right|_{s=0} & =\frac{\partial^{2}\sigma(s,q(t))}{\partial s\partial t}(0,t)-\dot{q}(t)\frac{\partial^{2}\tau}{\partial s\partial t}(0,t)\label{civho2}\\
 & =\frac{\diff{}}{\diff{t}}\left(\left.\frac{\partial\sigma(s,q(t))}{\partial s}\right|_{s=0}\right)(t)-\dot{q}(t)\frac{\partial^{2}\tau}{\partial s\partial t}(0,t)\\
 & =\frac{\diff{\eta^{0}(q,t)}}{\diff{t}}-\dot{q}(t)\frac{\partial^{2}\tau}{\partial s\partial t}(0,t)=\eta^{1}(q,t),
\end{align}
and for all $i=2,\ldots,m$
\begin{align}
\frac{\partial}{\partial s}\left.\left(\frac{\diff{^{i}\sigma}(s,q(t))}{\diff{\tau}^{i}(s,t)}\right)\right|_{s=0} & =\frac{\partial^{2}}{\partial s\partial t}\left.\left(\frac{\diff{^{i-1}\sigma}(s,q(t))}{\diff{\tau}^{i-1}(s,t)}\right)\right|_{s=0}-q^{(i)}(t)\frac{\partial^{2}\tau}{\partial s\partial t}(0,t)\nonumber \\
 & =\frac{\diff{}}{\diff{t}}\left(\frac{\partial}{\partial s}\left.\left(\frac{\diff{^{i-1}\sigma}(s,q(t))}{\diff{\tau}^{i-1}(s,t)}\right)\right|_{s=0}\right)(t)-q^{(i)}(t)\frac{\partial^{2}\tau}{\partial s\partial t}(0,t)\label{civho3}\\
 & =\frac{\diff{\eta^{i-1}}(q,t)}{\diff{t}}-q^{(i)}(t)\frac{\partial^{2}\tau}{\partial s\partial t}(0,t)=\eta^{i}(q,t).
\end{align}
Substituting \eqref{civho3} into~\eqref{civho1}, we obtain the
conclusion.
\end{proof}
\begin{defn}
\label{def:leicond1}Let $T'\subseteq[t_{1},t_{2}]$ be a sharply
open set. We say that a GSF $C[q]^{m}(\cdot)\in\gsf\left(T',\rcrho^{d}\right)$
is a \emph{constant of motion on $T'$, $S'$} if 
\begin{equation}
\frac{\diff{}}{\diff{t}}C[q]^{m}(t)=0\quad\forall t\in T'\label{eq:conslaw:td11}
\end{equation}
along all the weak extremals $q\in\zs\left([t_{1},t_{2}],S'\right)$.
\end{defn}

\noindent Note that Thm.~\ref{thm:existenceUniquenessPrimitives}
implies that condition \eqref{eq:conslaw:td11} is equivalent to ask
that $C[q]^{m}(\cdot)$ is constant on any interval $J\subseteq T'$.
In fact, if $t'\in J$, $F(t):=C[q]^{m}(t)-C[q]^{m}(t')$ is the unique
GSF such that $F(t')=0$ and $F'(t)=0$ on any closed interval contained
in $T'$ (and hence on an arbitrary interval by sharp continuity).
\begin{thm}
\label{thm:tn}If functional \eqref{eq:Pm} is invariant in the sense
of Def.~\ref{def:inva1}, then the quantity $C[q]^{m}(t)$ defined
for all $q\in\zs\left([t_{1},t_{2}],S'\right)$ and $t\in T'$ by
\begin{equation}
C[q]^{m}(t):=\sum_{i=1}^{m}\varphi^{i}(q,t)\cdot\eta^{i-1}(q,t)+\left(L[q]^{m}(t)-\sum_{i=1}^{m}\varphi^{i}(q,t)\cdot q^{(i)}(t)\right)\frac{\partial\tau}{\partial s}(0,t)\label{eq:LC}
\end{equation}
is a constant of motion on $T'$ and $S'$.
\end{thm}

\begin{proof}
The proof follows directly from the previous results and no specific
property of GSF is needed. For simplicity, we use the notations $\phi^{i}:=\phi^{i}(q,\cdot)$
and $\eta^{i}:=\eta^{i}(q,\cdot)$.
\begin{equation}
C[q]^{m}(\cdot)=\varphi^{1}\cdot\eta^{0}+\sum_{i=2}^{m}\varphi^{i}\cdot\eta^{i-1}+\left(L[q]^{m}-\sum_{i=1}^{m}\varphi^{i}\cdot q^{(i)}\right)\frac{\partial\tau}{\partial s}(0,\cdot).\label{eq:LC1}
\end{equation}
Differentiating \eqref{eq:LC1} with respect to $t$, we obtain 
\begin{multline}
\frac{\diff{}}{\diff{t}}C[q]^{m}(t)=\eta^{0}\cdot\frac{\diff{}}{\diff{t}}\varphi^{1}+\varphi^{1}\cdot\frac{\diff{}}{\diff{t}}\eta^{0}+\sum_{i=2}^{m}\left(\eta^{i-1}\cdot\frac{\diff{}}{\diff{t}}\varphi_{1}^{i}+\phi^{i}\cdot\frac{\diff{}}{\diff{t}}\eta^{i-1}\right)\\
\qquad+\frac{\partial\tau}{\partial s}(0,\cdot)\frac{\diff{}}{\diff{t}}\left(L[q]^{m}-\sum_{i=1}^{m}\varphi^{i}\cdot q^{(i)}\right)\\
+\left(L[q]^{m}-\sum_{i=1}^{m}\varphi^{i}\cdot q^{(i)}\right)\frac{\diff{}}{\diff{t}}\frac{\partial\tau}{\partial s}(0,\cdot).\label{eq:TeNetm2}
\end{multline}
Using the higher order Euler\textendash Lagrange equation \eqref{eq:ELdeordm1},
the higher order du Bois\textendash Reymond condition \eqref{eq:DBRordm:2},
relations \eqref{eq:eqprin11} and \eqref{civho} in \eqref{eq:TeNetm2},
one gets 
\begin{multline}
\frac{\diff{}}{\diff{t}}C[q]^{m}(t)=\partial_{2}L[q]^{m}\cdot\frac{\partial\sigma}{\partial s}(0,q(\cdot))+\varphi^{1}\cdot\left(\eta^{1}+\dot{q}\frac{\diff{}}{\diff{t}}\frac{\partial\tau}{\partial s}(0,\cdot)\right)\\
+\sum_{i=2}^{m}\left\{ \left(\partial_{i+1}L[q]^{m}-\varphi^{i-1}\right)\cdot\eta^{i-1}+\varphi^{i}\cdot\left(\eta^{i}+q^{(i)}\frac{\diff{}}{\diff{t}}\frac{\partial\tau}{\partial s}(0,\cdot)\right)\right\} \\
+\partial_{1}L[q]^{m}\frac{\partial\tau}{\partial s}(0,\cdot)+\left(L[q]^{m}-\sum_{i=1}^{m}\varphi^{i}\cdot q^{(j)}\right)\frac{\diff{}}{\diff{t}}\frac{\partial\tau}{\partial s}(0,\cdot)\\
=\partial_{1}L[q]^{m}\frac{\partial\tau}{\partial s}(0,\cdot)+L[q]^{m}\frac{\diff{}}{\diff{t}}\frac{\partial\tau}{\partial s}(0,\cdot)+\partial_{2}L[q]^{m}\cdot\frac{\partial\sigma}{\partial s}(0,q(\cdot))\\
+\varphi^{1}\cdot\left(\eta^{1}+\dot{q}\frac{\diff{}}{\diff{t}}\frac{\partial\tau}{\partial s}(0,\cdot)\right)-\varphi^{1}\cdot\eta^{1}\\
-\varphi^{1}\cdot\dot{q}\frac{\diff{}}{\diff{t}}\frac{\partial\tau}{\partial s}(0,\cdot)+\varphi^{m}\cdot\eta^{m}+\sum_{i=2}^{m}\partial_{i+1}L[q]^{m}\cdot\eta^{i-1}=0.\label{eq:dems}
\end{multline}
\end{proof}

\section{\label{sec:Optimal-control}Optimal control problems}

Thm.~\ref{thm:tn} gives a Lagrangian formulation of Noether's principle
extended to the generalized smooth functions setting. In this section,
we adopt the Hamiltonian formalism in order to generalize Noether's
principle to the wider context of optimal control, see e.g.~\cite{CD:Djukic:1972}.

Let us consider the optimal control problem for GSF in Lagrange form,
i.e.~the minimization of the functional: 
\begin{equation}
\int_{t_{1}}^{t_{2}}L\left(t,q(t),u(t)\right)\,\diff{t}\longrightarrow\min\label{eq:JO}
\end{equation}

\noindent subject to the Cauchy problem

\begin{equation}
\tag{CP}\begin{cases}
\dot{q}(t)=\varphi\left(t,q(t),u(t)\right),\\
q(t_{1})=q_{1}\,.
\end{cases}\label{cp}
\end{equation}
We explicitly note here that, because our generalized functions are
set-theoretical maps, and because of the closure with respect to composition,
the natural notion of solution for a differential equation with GSF
is given by pointwise equality. On the other hand, Thm.~\ref{thm:existenceUniquenessPrimitives}
and Cor.~\ref{cor:fundLem} imply that, for GSF, the notion of pointwise
solution is equivalent to that of weak solution.

\noindent In problem \eqref{eq:JO}\textendash \eqref{cp}, $q_{1}\in\gfs\,,$
the Lagrangian $L:[t_{1},t_{2}]\times H\times K\ra\gf$, the state
equation $\varphi:[t_{1},t_{2}]\times H\times K\ra\gfs$, the state
$q:[t_{1},t_{2}]\ra H$ and the control $u:[t_{1},t_{2}]\ra K$ are
assumed to be GSF with respect to all the arguments, where $H\fcmp\rcrho^{d}$,
$K\fcmp\rcrho^{l}$ are solid functionally compact sets. Note that
a particular case of this assumption is e.g.~$H=[-\diff{\rho}^{-q},\diff{\rho}^{-q}]^{d}\supseteq\R^{d}$
because $\diff{\rho}^{-q}$ is an infinite number. Since Sobolev-Schwartz
distributions can be embedded as GSF (see Sec.~\ref{subsec:Embedding}),
this is clearly more general than the usual traditional case, where
state and control functions are assumed to be piecewise smooth.

\noindent In particular, we will always consider state and control
functions in the following spaces
\begin{align}
q\in\mathcal{Q} & :=\left\{ q\in\gsf([t_{1},t_{2}],H)\mid\Vert q-q_{1}\Vert_{0}\le r\right\} \label{eq:Q}\\
u\in\mathcal{A} & :=\gsf([t_{1},t_{2}],K),\label{eq:A}
\end{align}
where $r\in\rcrho_{>0}$ is a fixed generalized radius such that $\overline{B_{r}(q_{1})}\subseteq H$.
In this section, we use simplified notations, e.g., of the form $L(\cdot,q,\dot{q})$
or $\phi(\cdot,q,u)$ to denote compositions $t\in[t_{1},t_{2}]\mapsto L(t,q(t),\dot{q}(t))\in\rti$
and $t\in[t_{1},t_{2}]\mapsto\phi(t,q(t),u(t))\in\rti^{d}$.

In developing this topic, it is therefore essential to already have
suitable results about solutions of ODE with GSF. The Banach fixed
point theorem can be easily generalized to spaces of generalized continuous
functions with the sup-norm $\Vert-\Vert_{0}$ (see Def.~\ref{def:genNormsSpaceGSF}).
As a consequence, we have the following Picard-Lindel�f theorem for
ODE in the $\gckf$ setting, see also \cite{ErlGross,baglinipaolo},
and \cite{Sto09} for an updated reference about solution of (fractional)
ODE in the framework of Colombeau's algebra.
\begin{thm}
\label{thm:PicLindFiniteContr}Let $t_{0}\in\rti$, $y_{0}\in\RC{\rho}^{d}$,
$\alpha$, $r\in\rti_{>0}$. Let $F\in\gckf([t_{0}-\alpha,t_{0}+\alpha]\times\overline{B_{r}(y_{0})},\rti^{d})$.
Set $M:={\displaystyle \max_{\substack{t_{0}-\alpha\le t\le t_{0}+\alpha\\
|y-y_{0}|\le r
}
}}|F(t,y)|$, $L:={\displaystyle \max_{\substack{t_{0}-\alpha\le t\le t_{0}+\alpha\\
|y-y_{0}|\le r
}
}}\left|\partial_{y}F(t,y)\right|\in\rti$ and assume that
\begin{align*}
\alpha\cdot M & \le r,
\end{align*}
\begin{equation}
\lim_{n\to+\infty}\alpha^{n}L^{n}=0,\label{eq:limitAssPL}
\end{equation}
where the limit in \eqref{eq:limitAssPL} is clearly taken in the
sharp topology. Then there exists a unique solution $y\in\gcf{k+1}\left([t_{0}-\alpha,t_{0}+\alpha],\rti^{d}\right)$
of the Cauchy problem
\begin{equation}
\begin{cases}
y'(t)=F(t,y(t))\\
y(t_{0})=y_{0}.
\end{cases}\label{eq:ODE}
\end{equation}
This solution is given by
\begin{align*}
y & =\lim_{n\to+\infty}P^{n}(y_{0})\\
P(y)(t): & =y_{0}+\int_{t_{0}}^{t}F(s,y(s))\,\diff{s}\quad\forall t\in[t_{0}-\alpha,t_{0}+\alpha],
\end{align*}
and for all $n\in\N$ satisfies $\Vert y-P^{n}(y_{0})\Vert_{0}\le\alpha M\sum_{k=n}^{+\infty}\frac{\alpha^{n}L^{n}}{n!}$
and $\Vert y-y_{0}\Vert_{0}\le r$.
\end{thm}

Finally, we have the following Gr�nwall-Bellman inequality in integral
form:
\begin{thm}
\label{thm:Groenwall}Let $\alpha\in\RC{\rho}_{>0}$. Let $u$, $a$,
$b\in\gckf\left([0,\alpha],\RC{\rho}\right)$ and assume that $\left\Vert a\right\Vert _{0}\cdot\alpha<N\cdot\log\left(\diff{\rho}^{-1}\right)$
for some $N\in\mathbb{N}$. Assume that $a(t)\geq0$ for all $t\in[0,\alpha]$,
and that $u(t)\leq b(t)+\int_{0}^{t}a(s)u(s)\,\diff{s}$. Then
\begin{enumerate}
\item \label{enu: Gronw 1}For every $t\in[0,\alpha]$ we have
\[
u(t)\leq b(t)+\int_{0}^{t}a(s)b(s)e^{\int_{s}^{t}a(r)\,\diff{r}}\,\diff{s};
\]
\item \label{enu: Gronw 2}If $b(t)\le b(s)$ for all $t\le s$, i.e.~if
$b$ is non-decreasing, then for every $t\in[0,\alpha]$ we have
\[
u(t)\leq b(t)e^{\int_{0}^{t}a(s)\,\diff{s}};
\]
\end{enumerate}
\end{thm}

In this section, we always assume that the state equation $\phi$
satisfies the assumptions of Thm.~\ref{thm:PicLindFiniteContr},
i.e.~setting:
\begin{align}
M_{K} & :={\displaystyle \max_{\substack{t_{1}\le t\le t_{2}\\
|q-q_{1}|\le r\\
k\in K
}
}}|\phi(t,q,k)|\nonumber \\
L_{K} & :=\max_{\substack{t_{1}\le t\le t_{2}\\
|q-q_{1}|\le r\\
k\in K
}
}\left|\partial_{q}\phi(t,q,k)\right|,\label{eq:lipConst}
\end{align}
we assume that
\begin{align}
 & (t_{2}-t_{1})M_{K}\le r\nonumber \\
 & \lim_{n\to+\infty}(t_{2}-t_{1})^{n}L_{K}^{n}=0.\label{eq:PLT-LipCond}
\end{align}
Therefore, Thm.~\ref{thm:PicLindFiniteContr} allows us to state
that
\begin{equation}
\forall u\in\mathcal{A}\,\exists!q^{u}\in\mathcal{Q}:\ \text{ \eqref{cp} holds for all }t\in[t_{1},t_{2}].\label{eq:q^u}
\end{equation}
Note that Thm.~\ref{thm:PicLindFiniteContr} in particular yields
both $q^{u}\in\gsf([t_{1},t_{2}],\rti^{d})$ and $\Vert q^{u}-q_{1}\Vert_{0}\le r$,
so that $q^{u}(t)\in\overline{B_{r}(q_{1})}\subseteq H$ for all $t\in[t_{1},t_{2}]$,
and thereby $q^{u}\in\mathcal{Q}$ (see \eqref{eq:Q}). Therefore,
\eqref{eq:q^u} holds for all solid functionally compact sets $H\fcmp\rti^{d}$
and $K\fcmp\rti^{l}$ such that both $\overline{B_{r}(q_{1})}\subseteq H$
and \eqref{eq:PLT-LipCond} hold. Note also explicitly, the importance
in this deduction of the closure of GSF (of which, we recall, Sobolev-Schwartz
distributions are a particular case) with respect to composition (Thm.~\ref{thm:propGSF}.\ref{enu:category}).

Finally, observe that the constant $L_{K}\in\rti$ defined in \eqref{eq:lipConst}
and Taylor Thm.~\ref{thm:Taylor} yield the Lipschitz condition
\begin{equation}
\forall u\in\mathcal{A}\,\forall t\in[t_{1},t_{2}]\,\forall q,\tilde{q}\in\overline{B_{r}(q_{1})}:\ \left|\phi(t,q,u(t))-\phi(t,\tilde{q},u(t))\right|\le L_{K}\left|q-\tilde{q}\right|.\label{eq:lipAbs}
\end{equation}

\subsection{Weak Pontryagin Maximum Principle}

Based on \eqref{eq:q^u}, we can introduce the notation and the optimization
problem
\begin{align}
\forall u & \in\mathcal{A}:\ I[u]:=\int_{t_{1}}^{t_{2}}L\left(t,q^{u}(t),u(t)\right)\,\diff{t}\label{eq:JOu}\\
\text{Find }v & \in\mathcal{A}:\ \eqref{cp}\text{ and }\exists r\in\rti_{>0}\,\exists l\in\N\,\forall u\in\mathcal{A}\cap B_{r}^{l}(v):\ I[v]\le I[u].\label{eq:optCtrlProbl}
\end{align}
 In order to develop the necessary optimality condition for Problem
\eqref{eq:JOu}, we first assume that $u\in\accentset{\circ}{\mathcal{A}}$
in the sharp topology of the $\rcrho$-graded Fr�chet space $(\mathcal{A},\Vert-\Vert_{0})$,
see Sec.~\ref{subsec:EVTandFcmp}, i.e.
\[
\exists r\in\rti_{>0}:\ B_{r}(u)=\left\{ v\in\mathcal{A}\mid\max_{t_{1}\le t\le t_{2}}\left|v(t)-u(t)\right|<r\right\} \subseteq\mathcal{A}.
\]
Thereby, for each direction $\bar{u}\in\gsf_{0}(t_{1},t_{2})$, and
for some $\delta=\delta(u,\bar{u})\in(0,1)$ sufficiently small
\begin{equation}
\forall h\in(-\delta,\delta):\ u+h\bar{u}\in\mathcal{A},\label{eq:deltaSmall}
\end{equation}
and hence we can evaluate $I[u+h\bar{u}]\in\rti$. Note explicitly
that the control $u\in\mathcal{A}=\gsf([t_{1},t_{2}],K)$, whereas
the direction $\bar{u}$ lies in the $\rti$-module $\gsf_{0}(t_{1},t_{2})\subseteq\gsf([t_{1},t_{2}],\rti^{d})$;
this is indispensable to apply the fundamental Lem.~\ref{lem:fund_lem_calc_var}
(see e.g.~the proof of Thm.~\ref{ocu}).

To define the first variation of $I[-]:\mathcal{A}\ra\rcrho$ at $u\in\mathcal{A}$
in the direction $\bar{u}\in\gsf_{0}(t_{1},t_{2})$, we can intuitively
think at a sort of first order Taylor sum of $I[u+h\bar{u}]$ at $u$:
\begin{defn}
\label{def:incrRatio}Let $\delta\in\rti_{>0}$ be such that \eqref{eq:deltaSmall}
holds. We say that $R:(-\delta,\delta)\ra\rti$ is \emph{an} \emph{incremental
ratio of} $I[-]$ \emph{at} $u\in\mathcal{A}$ \emph{in the direction
}$\bar{u}\in\gsf_{0}(t_{1},t_{2})$, if there exists a function $\hat{R}:(-\delta,\delta)\ra\rti$
(called \emph{remainder}) such that:
\begin{enumerate}
\item \label{enu:diffFR}$\forall h\in(-\delta,\delta):\ I[u+h\bar{u}]=I[u]+h\cdot R(h)+\hat{R}(h)$
\item \label{enu:diffCont}$R$ is continuous at $0$ in the sharp topology
\item \label{enu:diffRem}$\exists A\in\rti_{>0}\,\forall h\in(-\delta,\delta):\ \left|\hat{R}(h)\right|\le A\cdot h^{2}$.
\end{enumerate}
\end{defn}

\noindent We first prove that $R(0)\in\rti$ is unique:
\begin{thm}
Let $u\in\accentset{\circ}{\mathcal{A}}$ and $\bar{u}\in\gsf_{0}(t_{1},t_{2})$.
If $R$, $S$ are incremental ratios of $I[-]$ at $u$ in the direction
$\bar{u}$, then $R(0)=S(0)$.
\end{thm}

\begin{proof}
We have $R:(-\delta_{1},\delta_{1})\ra\rti$ and $S:(-\delta_{2},\delta_{2})\ra\rti$
from Def.~\ref{def:incrRatio}. Set $\delta:=\delta_{1}\wedge\delta_{2}\wedge1$,
so that $0<\delta\le1$ by Lem.~\ref{lem:mayer}. Let $\hat{R}$
and $\hat{S}$ be remainders of $R$, $S$ respectively. Take an invertible
$h\in(-\delta,\delta)$; from Def.~\ref{def:incrRatio}.\ref{enu:diffFR},
we get $R(h)-S(h)=h^{-1}\cdot\left(\hat{R}(h)-\hat{S}(h)\right)$.
Therefore, $\left|R(h)-S(h)\right|\le(A+\hat{A})h$ from Def.~\ref{def:incrRatio}.\ref{enu:diffRem},
and hence $\lim_{\substack{h\to0\\
h\text{ invertible}
}
}\left|R(h)-S(h)\right|=0$. But invertible elements are dense in the sharp topology (Lem.~\ref{lem:invDense}),
and thus $\lim_{h\to0}\left|R(h)-S(h)\right|=0=R(0)-S(0)$ using Def.~\ref{def:incrRatio}.\ref{enu:diffCont}.
\end{proof}
\begin{defn}
\label{def:diffI^u}Assume that there exists an incremental ratio
of $I[-]$. Then, we define the \emph{first variation of $I[-]$ at
}$u\in\accentset{\circ}{\mathcal{A}}$ \emph{in the direction} $\bar{u}\in\gsf_{0}(t_{1},t_{2})$\emph{
}as \emph{$\delta I^{u}(u;\bar{u}):=R(0)$}, where $R$ is any incremental
ratio of $I[-]$. Moreover, we say that $u\in\mathcal{A}$ is \emph{a
weak extremal of $I[-]$} if for any $\bar{u}\in\gsf_{0}(t_{1},t_{2})$,
$\delta I(u;\bar{u})=0$.
\end{defn}

Using Def.~\ref{def:incrRatio} and similarly to Thm.~\ref{thm:necessCondsForMinimizer},
we can prove the following
\begin{thm}
\label{thm:necCond}If $v\in\accentset{\circ}{\mathcal{A}}$ solves
problem \eqref{eq:optCtrlProbl} and there exists an incremental ratio
of $I[-]$, then $v$ a weak extremal of $I[-]$.
\end{thm}

\begin{proof}
Let $\bar{u}\in\gsf_{0}(t_{1},t_{2})$ and let $\delta\in(0,1)$ be
such that \eqref{eq:deltaSmall} holds and such that $v+h\bar{u}\in B_{r}^{l}(v)$
for all $h\in(-\delta,\delta)$. From Def.~\ref{def:incrRatio}.\ref{enu:diffFR}
of incremental ratio, we have $I[v+h\bar{u}]=I[v]+hR(h)+\hat{R}(h)\ge I[v]$.
Therefore, for all $k\in(0,\delta)$, we get $kR(k)\ge-\hat{R}(k)$
and hence $-R(k)\le Ak=A|k|$ by Def.~\ref{def:incrRatio}.\ref{enu:diffRem}.
Similarly, for all $k\in(-\delta,0)$, we get $R(k)\le-Ak=A|k|$.
Thereby, $|R(k)|=R(k)\vee-R(k)\le A|k|$ for all invertible $k\in(-\delta,\delta)$,
which implies $R(0)=\delta I(v;\bar{u})=0$.
\end{proof}
We now prove that, under sufficiently weak conditions, an incremental
ratio of $I[-]$ always exists and $\delta I(u;\cdot)$ is an $\rti$-linear
continuous map. We first show the following continuity conditions
for the map $u\in\mathcal{A}\mapsto q^{u}\in\mathcal{Q}$ (see \cite{loic}
for a similar proof):
\begin{thm}[Stability of order 1 and 2]
\label{RE1}Let $u\in\accentset{\circ}{\mathcal{A}}$, $\bar{u}\in\gsf_{0}(t_{1},t_{2})$
and $\delta\in(0,1)$ as in \eqref{eq:deltaSmall}. Assume that
\begin{equation}
\exists N\in\N:\ L_{K}\cdot(t_{2}-t_{1})\le N\log(\diff{\rho}^{-1}).\label{eq:boundExp}
\end{equation}
Then, there exist constants $A$, $\bar{A}\in\rti$ (depending only
on $u$, $\bar{u}$ and clearly on $\phi$) such that for all $\mid h\mid<\delta$,
we have

\begin{align}
 & \left\Vert q^{u+h\bar{u}}-q^{u}\right\Vert _{0}\leq\left|h\right|A\label{re1}\\
 & \left\Vert q^{u+h\bar{u}}-q^{u}-h\bar{q}\right\Vert _{0}\leq\bar{A}h^{2},\label{eq:re2}
\end{align}
where $\bar{q}\in\gsf([t_{1},t_{2}],\rti^{d})$ is the unique global
solution of the following Cauchy's problem 
\begin{equation}
\tag{LCP\ensuremath{_{1}}}\left\lbrace \begin{array}{l}
\dot{\bar{q}}=\dfrac{\partial\varphi}{\partial q}(t,q^{u},u)\cdot\bar{q}+\dfrac{\partial\varphi}{\partial u}(t,q^{u},u)\cdot\bar{u}\\[10pt]
\bar{q}(t_{1})=0
\end{array}\right.\label{re8}
\end{equation}
\end{thm}

\begin{proof}
We prove in detail \eqref{re1}; the proof of \eqref{eq:re2} is similar.
Let $h\in(-\delta,\delta)$. From the evolution ODE \eqref{cp}, for
all $t\in[t_{1},t_{2}]$, we have
\begin{align}
\left|q^{u+h\bar{u}}(t)-q^{u}(t)\right| & \le\left|\int_{t_{1}}^{t}\left(\varphi(s,q^{u+h\bar{u}}(s),u(s)+h\bar{u}(s))-\varphi(s,q^{u}(s),u(s))\right)\,\diff{s}\right|\le\nonumber \\
 & \le\int_{t_{1}}^{t}\left|\varphi(s,q^{u+h\bar{u}}(s),u(s)+h\bar{u}(s))-\varphi(s,q^{u}(s),u(s))\right|\,\diff{s}\label{eq:re4Init}
\end{align}
for all $s\in[t_{1},t]$, we have 
\begin{multline}
\left|\varphi\big(s,q^{u+h\bar{u}}(s),u(s)+h\bar{u}(s)\big)-\varphi\big(s,q^{u}(s),u(s)\big)\right|\le\\
\leq\left|\varphi\big(s,q^{u+h\bar{u}}(s),u(s)+h\bar{u}(s)\big)-\varphi\big(s,q^{u}(s),u(s)+h\bar{u}(s)\big)\right|+\\
+\left|\varphi\big(s,q^{u}(s),u(s)+h\bar{u}(s)\big)-\varphi\big(s,q^{u}(s),u(s)\big)\right|.\label{re3}
\end{multline}
Now, we apply the Lipschitz condition \eqref{eq:lipAbs} to the first
summand, and a first order Taylor expansion with Lagrange remainder
(Thm.~\ref{thm:Taylor}.\ref{enu:LagrangeRest}) to the second one,
obtaining 
\begin{equation}
\left|\varphi\big(s,q^{u+h\bar{u}}(s),u(s)+h\bar{u}(s)\big)-\varphi\big(s,q^{u}(s),u(s)+h\bar{u}(s)\big)\right|\le L_{K}\cdot\left|q^{u+h\bar{u}}(s)-q^{u}(s)\right|\label{eq:re4Lip}
\end{equation}
\begin{equation}
\left|\varphi\big(s,q^{u}(s),u(s)+h\bar{u}(s)\big)-\varphi\big(s,q^{u}(s),u(s)\big)\right|\le\vert h\vert\left|\dfrac{\partial\varphi}{\partial u}\big(s,q^{u}(s),\xi)\cdot\bar{u}(s)\right|\label{eq:re4Taylor}
\end{equation}
where $\xi=\xi(u,h,\bar{u},s)\in[u(s),u(s)+h\bar{u}(s)]\subseteq\rti^{l}$.
The extreme value Thm.~\ref{thm:extremeValues} applied to $u$,
$\bar{u}$ on the functionally compact set $[t_{1},t_{2}]$ yields
the existence of constants $B_{1}$, $B_{2}\in\gf_{>0}$ such that
$[u(s),u(s)+h\bar{u}(s)]\subseteq[B_{1},B_{2}]^{l}$; since $|h|<\delta$,
we can always assume that these constants do not depend on $h$ but
only on $u$ and $\bar{u}$. In the same way, applying the extreme
value theorem with $\dfrac{\partial\varphi}{\partial u}$ on the functionally
compact set $[t_{1},t_{2}]\times\overline{B_{r}(q_{1})}\times[B_{1},B_{2}]^{l}$,
we get the existence of a constant $C\in\rti$ (depending only on
$u$, $\bar{u}$) such that 
\[
\forall s\in[t_{1},t_{2}]:\ \left|\dfrac{\partial\varphi}{\partial u}\big(s,q^{u}(s),\xi)\cdot\bar{u}(s)\right|\leq C.
\]
Thereby, considering \eqref{eq:re4Taylor}:
\begin{equation}
\int_{t_{1}}^{t}\left|\varphi\big(s,q^{u}(s),u(s)+h\bar{u}(s)\big)-\varphi\big(s,q^{u}(s),u(s)\big)\right|\,\diff{s}\leq|h|C\cdot(t_{2}-t_{1}).\label{re5}
\end{equation}
And, using inequalities \eqref{eq:re4Init}, \eqref{re3}, \eqref{eq:re4Lip},
and \eqref{re5}, we have: 
\begin{equation}
\forall t\in[t_{1},t_{2}]:\ \left|q^{u+h\bar{u}}(t)-q^{u}(t)\right|\leq L_{K}\cdot\int_{t_{1}}^{t}\left|q^{u+h\bar{u}}(s)-q^{u}(s)\right|\,\diff{s}+\vert h\vert C\cdot(t_{2}-t_{1}).\label{re6}
\end{equation}
Finally, we apply Gr�nwall-Bellman Thm.~\ref{thm:Groenwall} to \eqref{re6}
and use assumption \eqref{eq:boundExp} to obtain
\begin{equation}
\Vert q^{u+h\bar{u}}-q^{u}\Vert_{0}\leq\vert h\vert C\cdot(t_{2}-t_{1})e^{L_{K}(t_{2}-t_{1})}\le\vert h\vert C\cdot(t_{2}-t_{1})\diff{\rho}^{-N}.\label{re7}
\end{equation}
Setting $A:=C\cdot(t_{2}-t_{1})\diff{\rho}^{-N}$, \eqref{re7} proves
the claim, because also $N$ does not depend on $h$ but only on $K$,
$\phi$ and $q_{1}$.
\end{proof}
Note that the existence of a solution $\bar{q}$ of the Cauchy problem
\eqref{re8} can be directly proved by setting for all $t\in[t_{1},t_{2}]$
\begin{align}
a(t) & :=\dfrac{\partial\varphi}{\partial q}(t,q^{u}(t),u(t))\label{eq:a}\\
b(t) & :=\dfrac{\partial\varphi}{\partial u}(t,q^{u}(t),u(t))\label{eq:b}\\
\bar{q}(t) & :=\exp\left(\int_{t_{1}}^{t}a\right)\cdot\int_{t_{1}}^{t}b(s)\cdot\exp\left(-\int_{t_{1}}^{t}a\right)\,\diff{s}.\nonumber 
\end{align}
In fact, $\bar{q}\in\gsf([t_{1},t_{2}],\rti^{d})$ because primitives
of GSF are GSF (Thm.~\ref{thm:existenceUniquenessPrimitives}), because
of the closure with respect to composition (Thm.~\ref{thm:propGSF}.\ref{enu:category}),
and because of condition \eqref{eq:boundExp}, which implies that
$\exp\left(\int_{t_{1}}^{t}a\right)\in\rti$. The uniqueness of the
solution follows from the Picard-Lindel�f Thm.~\ref{thm:PicLindFiniteContr}
considering e.g.~$r:=2(t_{2}-t_{1})L_{K}$. Finally, note that assumption
\eqref{eq:boundExp} always holds if $L_{K}\cdot(t_{2}-t_{1})$ is
a finite number or if $t_{2}-t_{1}$ is sufficiently small.
\begin{thm}
\label{th:cno}In the assumptions of Thm.~\ref{RE1}, there always
exists an incremental ratio of $I[-]$, and the following equality
holds: 
\begin{equation}
\delta I(u;\bar{u})=\int_{t_{1}}^{t_{2}}\left(\dfrac{\partial L}{\partial q}(s,q^{u}(s),u(s))\cdot\bar{q}(s)+\dfrac{\partial L}{\partial u}(t,q^{u}(s),u(s))\cdot\bar{u}(s)\right)\,\diff{s}\label{re10}
\end{equation}
where $\bar{q}=\bar{q}(\bar{u})\in\mathcal{Q}$ is the unique global
solution of problem \eqref{re8}. Moreover, $\bar{u}\in\gsf_{0}(t_{1},t_{2})\mapsto\delta I(u;\bar{u})\in\rti$
is an $\rti$-continuous linear map, i.e.~it is $\rti$-linear and
satisfies
\begin{equation}
\exists J\in\rti_{>0}\,\exists r\in\rti_{>0}\,\forall\bar{u},\bar{v}\in B_{r}(0):\ \left|\delta I(u;\bar{u})-\delta I(u;\bar{v})\right|\le J\cdot\Vert\bar{u}-\bar{v}\Vert_{0}.\label{eq:IBO}
\end{equation}
\end{thm}

\begin{proof}
Since $u\in\accentset{\circ}{\mathcal{A}}$, there always exists $\delta\in(0,1)$
such that \eqref{eq:deltaSmall} holds. Set $v^{h}:=q^{u+h\bar{u}}-q^{u}-h\bar{q}$
for all $h\in(-\delta,\delta)$, so that
\begin{equation}
\Vert\upsilon^{h}\Vert_{0}\leq\bar{A}h^{2}\label{eq:v^h^2}
\end{equation}
by Thm.~\ref{RE1}. For $t\in[t_{1},t_{2}]$ and $k=(0,k_{q},k_{u})\in\rti\times\rti^{d}\times\rti^{l}$,
let
\begin{multline*}
R_{L}(t,q^{u}(t),u(t);k):=L(t,q^{u}(t)+k_{q},u(t)+k_{u})-L(t,q^{u}(t),u(t))\\
-\partial_{q}L(t,q^{u}(t),u(t))\cdot k_{q}-\partial_{u}L(t,q^{u}(t),u(t))\cdot k_{u}
\end{multline*}
be the remainder of the first order Taylor formula of $L$ at the
point $(t,q^{u}(t),u(t))$ with increment $k=(0,k_{q},k_{u})$ (see
Thm.~\ref{thm:Taylor}). Thereby, for all $h\in(-\delta,\delta)$,
we get
\begin{align*}
I[u+h\bar{u}]-I[u] & =\int_{t_{1}}^{t_{2}}\left\{ L(\cdot,q^{u+h\bar{u}},u+h\bar{u})-L(\cdot,q^{u},u)\right\} =\\
 & =\int_{t_{1}}^{t_{2}}\left\{ L(\cdot,q^{u}+h\bar{q}+v^{h},u+h\bar{u})-L(\cdot,q^{u},u)\right\} =\\
 & =\int_{t_{1}}^{t_{2}}\left\{ \partial_{q}L(\cdot,q^{u},u)\cdot(h\bar{q}+v^{h})+\partial_{u}L(\cdot,q^{u},u)\cdot h\bar{u}+\right.\\
 & \phantom{=\int_{t_{1}}^{t_{2}}\{}\left.+R_{L}(\cdot,q^{u},u;(\cdot,h\bar{q}+v^{h},h\bar{u}))\right\} =\\
 & =h\int_{t_{1}}^{t_{2}}\left\{ \partial_{q}L(\cdot,q^{u},u)\cdot\bar{q}+\partial_{u}L(\cdot,q^{u},u)\cdot\bar{u}\right\} +\\
 & \phantom{=}+\int_{t_{1}}^{t_{2}}\left\{ \partial_{q}L(\cdot,q^{u},u)\cdot v^{h}+R_{L}(\cdot,q^{u},u;(\cdot,h\bar{q}+v^{h},h\bar{u}))\right\} .
\end{align*}
Setting 
\begin{align*}
R & :=\int_{t_{1}}^{t_{2}}\left\{ \partial_{q}L(\cdot,q^{u},u)\cdot\bar{q}+\partial_{u}L(\cdot,q^{u},u)\cdot\bar{u}\right\} \\
\hat{R} & (h):=\int_{t_{1}}^{t_{2}}\left\{ \partial_{q}L(\cdot,q^{u},u)\cdot v^{h}+R_{L}(\cdot,q^{u},u;(\cdot,h\bar{q}+v^{h},h\bar{u}))\right\} ,
\end{align*}
Taylor Thm.~\ref{thm:Taylor} and \eqref{eq:v^h^2} yield the existence
of a constant $D\in\rti_{>0}$ such that $|\hat{R}(h)|\le Dh^{2}$.
This proves that $R=R(0)=\delta J(u;\bar{u})$ is the incremental
ratio of $I[-]$ with remainder $\hat{R}$ (Def.~\ref{def:incrRatio}),
which is claim \eqref{re10}.

Using the notation $\bar{q}=\bar{q}(\bar{u})$ and the uniqueness
of solution of problem \eqref{re8}, it follows that $\bar{q}(\alpha\bar{u}+\beta\bar{v})=\alpha\bar{q}(\bar{u})+\beta\bar{q}(\bar{v})$
for all $\alpha$, $\beta\in\rti$ and $\bar{u}$, $\bar{v}\in\gsf([t_{1},t_{2}],K)$.
Thereby, \eqref{re10} implies that $\delta I(u;\cdot):\gsf_{0}(t_{1},t_{2})\ra\rti$
is an $\rti$-linear map. Finally, with the simplified notations \eqref{eq:a}
and \eqref{eq:b}, we have
\begin{align*}
\bar{q}(\bar{u})-\bar{q}(\bar{v}) & =\int_{t_{1}}^{(\cdot)}a\cdot\left\{ \bar{q}(\bar{u})-\bar{q}(\bar{v})\right\} +\int_{t_{1}}^{(\cdot)}b\cdot(\bar{u}-\bar{v})\\
\Vert\bar{q}(\bar{u})-\bar{q}(\bar{v})\Vert_{0} & \le\Vert\bar{q}(\bar{u})-\bar{q}(\bar{v})\Vert_{0}\cdot(t_{2}-t_{1})L_{K}+\Vert b\Vert_{0}(t_{2}-t_{1})\Vert\bar{u}-\bar{v}\Vert_{0}.
\end{align*}
Note that $\Vert a\Vert_{0}\le L_{K}$ by \eqref{eq:lipConst}. But
assumption \eqref{eq:PLT-LipCond} implies that $(t_{2}-t_{1})^{N}L_{K}^{N}<1$
for some $N\in\N$ and hence that $1-(t_{2}-t_{1})L_{K}$ is invertible,
yielding
\[
\Vert\bar{q}(\bar{u})-\bar{q}(\bar{v})\Vert_{0}\le(1-(t_{2}-t_{1})L_{K})^{-1}\Vert b\Vert_{0}(t_{2}-t_{1})\Vert\bar{u}-\bar{v}\Vert_{0}.
\]
Now, the conclusion \eqref{eq:IBO} follows from this Lipschitz property
and \eqref{re10}.
\end{proof}
Note that the Lipschitz constant $J\in\rti$ in \eqref{eq:IBO} can
be an infinite number, e.g.~if the Lagrangian $L$ shows some kind
of singularity as a generalized function.
\begin{defn}
\label{def:Hamiltonian}Let $\mathcal{H}$ be the \emph{Hamiltonian}
associated to the Lagrangian $L$ and the state equation \eqref{cp},
i.e.~the GSF
\begin{equation}
\fonction{\mathcal{H}}{[t_{1},t_{2}]\times H\times K\times\rti^{d}}{\gf}{(t,q,u,p)}{L(t,q,u)+p\cdot\varphi(t,q,u).}\label{eq:Ham}
\end{equation}
Moreover, for any control $u\in\mathcal{A}$, let $p^{u}\in\gsf([t_{1},t_{2}],\rti^{d})$
denote the \emph{adjoint variable (generalized momentum)}, i.e.~the
unique GSF solution of the Cauchy problem
\begin{equation}
\tag{LCP\ensuremath{_{2}}}\left\lbrace \begin{array}{l}
\dot{p}^{u}=-\dfrac{\partial\mathcal{H}}{\partial q}(t,q^{u},u,p^{u})=-\dfrac{\partial L}{\partial q}(t,q^{u},u)-\left(\dfrac{\partial\varphi}{\partial q}(t,q^{u},u)\right)^{\text{T}}\cdot p^{u}\\[10pt]
p^{u}(t_{2})=0.
\end{array}\right.\label{re11}
\end{equation}
As we showed above, this problem has a unique solution on $[t_{1},t_{2}]$
because of our assumptions \eqref{eq:PLT-LipCond}.
\end{defn}

We want to close this section, giving a proof of the \emph{weak Pontryagin
Maximum Principle}, i.e.~a theorem where instead of the usual condition
\[
\mathcal{H}(t,q^{v}(t),v(t),p^{v}(t))=\min_{k\in K}\mathcal{H}(t,q^{v}(t),k,p^{v}(t))\quad\forall t\in[t_{1},t_{2}]
\]
(if $v$ is an optimal control, i.e.~it solves problem \eqref{eq:optCtrlProbl};
see e.g.~\cite{Shv05,BeLeVi01,Hes66,CD:MR29:3316b}) we have instead
the necessary condition $\partial_{u}\mathcal{H}(t,q^{v}(t),v(t),p^{v}(t))=0$
assuming that $v\in\mathcal{A}$ is a local minimum of the functional
$I[-]$.

Directly from the definition of Hamiltonian, we get that, for any
control $u\in\mathcal{A}$, the pair $(q^{u},p^{u})$ is a solution
of the following \textit{\emph{Hamiltonian system}}:
\begin{equation}
\tag{HS}\left\lbrace \begin{array}{l}
\dot{q}=\dfrac{\partial\mathcal{H}}{\partial p}(t,q,u,p)\\[10pt]
\dot{p}=-\dfrac{\partial\mathcal{H}}{\partial q}(t,q,u,p).
\end{array}\right.\label{re12}
\end{equation}

Next, we prove the following optimality condition for the functional
\eqref{eq:JOu}:
\begin{thm}
\label{ocu}In the assumptions of Thm.~\ref{RE1}, let $u\in\accentset{\circ}{\mathcal{A}}$
and $\bar{u}\in\gsf([t_{1},t_{2}],K)$. Then
\[
\delta I(u,\bar{u})=\int_{t_{1}}^{t_{2}}\dfrac{\partial\mathcal{H}}{\partial u}(\cdot,q^{u},u,p^{u})\cdot\bar{u}.
\]
Therefore $u$ is a weak extremal of $I[-]$ if and only if $(q^{u},u,p^{u})$
satisfy the equation: 
\begin{equation}
\tag{GSE}\dfrac{\partial H}{\partial u}(t,q^{u},u,p^{u})=0.\label{eqse}
\end{equation}
\end{thm}

\begin{proof}
Let $\bar{u}\in\gsf_{0}(t_{1},t_{2})$. Thm.~\ref{th:cno} asserts
that 
\begin{equation}
\delta I(u;\bar{u})=\int_{t_{1}}^{t_{2}}\left(\dfrac{\partial L}{\partial q}(\cdot,q^{u},u)\cdot\bar{q}+\dfrac{\partial L}{\partial u}(\cdot,q^{u},u)\cdot\bar{u}\right).\label{re13}
\end{equation}
Equation \eqref{re13} can be written as 
\begin{multline}
\delta I(u;\bar{u})=\int_{t_{1}}^{t_{2}}\left(\left(\dfrac{\partial L}{\partial q}(\cdot,q^{u},u)+\left(\dfrac{\partial\varphi}{\partial q}(\cdot,q^{u},u)\right)^{\text{T}}\cdot p^{u}\right)\cdot\bar{q}\right.\\
\left.-\left(\dfrac{\partial\varphi}{\partial q}(\cdot,q^{u},u)\cdot\bar{q}\right)\cdot p^{u}+\dfrac{\partial L}{\partial u}(\cdot,q^{u},u)\cdot\bar{u}\right).\label{re14}
\end{multline}

\noindent By \eqref{re8} and \eqref{re11}, we have

\begin{equation}
\bar{q}=\int_{t_{1}}^{(\cdot)}\left(\dfrac{\partial\varphi}{\partial q}(\cdot,q^{u},u)\cdot\bar{q}+\dfrac{\partial\varphi}{\partial u}(\cdot,q^{u},u)\cdot\bar{u}\right)\label{re15}
\end{equation}
and

\begin{equation}
p^{u}=-\int_{t_{1}}^{(\cdot)}\left(\dfrac{\partial L}{\partial q}(\cdot,q^{u},u)+\left(\dfrac{\partial\varphi}{\partial q}(\cdot,q^{u},u)\right)^{\text{T}}\cdot p^{u}\right).\label{re16}
\end{equation}
Now, using in \eqref{re14} equalities \eqref{re15}, \eqref{re16}
and integrating by parts (Thm.~\ref{thm:intRules}.\ref{enu:intByParts})
with $\bar{q}(t_{1})=0$, $p^{u}(t_{2})=0$, we get

\begin{multline*}
\delta I(u;\bar{u})=\int_{t_{1}}^{t_{2}}\left(p^{u}\cdot\left(\dfrac{\partial\varphi}{\partial q}(\cdot,q^{u},u)\cdot\bar{q}+\dfrac{\partial\varphi}{\partial u}(t,q^{u},u)\cdot\bar{u}\right)\right.\\
\left.-\left(\dfrac{\partial\varphi}{\partial q}(\cdot,q^{u},u)\cdot\bar{q}\right)\cdot p^{u}+\dfrac{\partial L}{\partial u}(\cdot,q^{u},u)\cdot\bar{u}\right)=\\
=\int_{t_{1}}^{t_{2}}\left(\left(\dfrac{\partial\varphi}{\partial u}(\cdot,q^{u},u)\right)^{\text{T}}\cdot p^{u}+\dfrac{\partial L}{\partial u}(\cdot,q^{u},u)\right)\cdot\bar{u}=\\
=\int_{t_{1}}^{t_{2}}\dfrac{\partial\mathcal{H}}{\partial u}(\cdot,q^{u},u,p^{u})\cdot\bar{u}.
\end{multline*}

\noindent The proof is now completed applying the fundamental lemma
Cor.~\ref{cor:fundLem}.
\end{proof}
Summarizing these results yields the weak Pontryagin Maximum Principle:
\begin{cor}[Weak Pontryagin Maximum Principle]
\label{th:P}In the assumptions of Thm.~\ref{RE1}, if $v\in\accentset{\circ}{\mathcal{A}}$
is a local extremal of $I[-]$, i.e.~it solves problem \eqref{eq:optCtrlProbl},
then $(q^{v},v,p^{v})\in\gsf([t_{1},t_{2}],\rti^{d})$ are solution
of the system
\begin{equation}
\tag{WPS}\left\lbrace \begin{array}{l}
\dot{q}=\dfrac{\partial\mathcal{H}}{\partial p}(t,q,v,p)\\[10pt]
\dot{p}=-\dfrac{\partial\mathcal{H}}{\partial q}(t,q,v,p)\\[10pt]
\dfrac{\partial\mathcal{H}}{\partial u}(t,q,v,p)=0\\[10pt]
\big(q(t_{1}),p(t_{2})\big)=(0,0).
\end{array}\right.\label{wps}
\end{equation}
Moreover, if $(q,u,p)\in\gsf([t_{1},t_{2}],\rti^{d})$ solve \eqref{wps},
then necessarily $(q,p)=(q^{u},p^{u})$.
\end{cor}

\begin{proof}
By Thm.~\ref{thm:necCond}, we get that $v$ is also a weak extremal
of the functional $I[-]$. Therefore, the conclusion follows by \eqref{re12},
Thm.~\ref{ocu} and \eqref{re8}, \eqref{re11}.
\end{proof}
\begin{defn}
\label{def:extPont} Any triple $(q,u,p)\in\gsf([t_{1},t_{2}],\rti^{d})$
satisfying the system \eqref{wps} is called a \emph{weak Pontryagin
extremal}.
\end{defn}

The next proposition generalizes the du Bois\textendash Reymond necessary
optimality condition Thm\@.~\ref{thm:cDRifm} (as usual, set in
the following theorem: $m=1$, $\phi(t,q,u)=u$ so that $\dot{q}=u$).
Its proof follows directly from \eqref{eq:Ham} and the system \eqref{wps}.
\begin{thm}
\label{prp:tdH}In the previous assumptions, the following property
holds for any weak Pontryagin extremals $(q,u,p)$:
\begin{equation}
\frac{\diff{}}{\diff{t}}\left(\mathcal{H}(t,q(t),u(t),p(t))\right)=\frac{\partial{\mathcal{H}}}{\partial{t}}(t,q(t),u(t),p(t))\quad\forall t\in[t_{1},t_{2}].\label{eq:H8}
\end{equation}
\end{thm}

\subsection{Noether's Principle for optimal control}

In classical optimal control, (see e.g.~\cite{CD:Djukic:1972,Gogodze88,Sag68})
the optimal problem \eqref{eq:optCtrlProbl} is equivalent to minimize
the augmented functional $I[q,u,p]$ defined by 
\begin{equation}
I[q,u,p]:=\int_{t_{1}}^{t_{2}}\left({\mathcal{H}}\left(t,q(t),u(t),p(t)\right)-p(t)\cdot\dot{q}(t)\right)\,\diff{t},\label{eq:COA1}
\end{equation}

\noindent with $\mathcal{H}$ given by \eqref{eq:Ham}. The notion
of variational invariance for problem \eqref{eq:optCtrlProbl} is
then defined with the help of the augmented functional \eqref{eq:COA1}.
In the GSF setting, we follow the same approach:
\begin{defn}
\label{def:inv:gt1}Let $T=\left\{ \tau(s,\cdot)\right\} _{s\in P}\in\gsf(T',T')$,
$S=\left\{ \sigma(s,\cdot)\right\} _{s\in P}\in\gsf(S',S')$, $U=\left\{ \upsilon(s,\cdot)\right\} _{s\in P}\in\gsf(U',U')$,
and $A=\left\{ \pi(s,\cdot)\right\} _{s\in P}\in\gsf(A',A')$ be one
parameter groups of diffeomorphisms of the open sets $T'\subseteq\rti$,
$S'$, $A'\subseteq\rti^{d}$ and $U'\subseteq\rti^{l}$. The augmented
functional \eqref{eq:COA1} is said to be \emph{invariant under the
action of} $T$, $S$, $U$, $A$, if for any \emph{weak Pontryagin
extremal} $(q,u,p)$ such that $q\in\gsf(T',S')$, $u\in\gsf(T',U')$
and $p\in\gsf(T',A')$, the following equality holds
\begin{multline}
\left\{ \mathcal{H}\left(\tau(s,t),\sigma(s,q(t)),\upsilon(s,u(t)),\pi(s,p(t))\right)-\pi(s,p(t))\cdot\frac{\diff{\sigma}(s,q(t))}{\diff{\tau}(s,t)}\right\} \frac{\partial\tau}{\partial t}(s,t)=\\
=\mathcal{H}\left(t,q(t),u(t),p(t)\right)-p(t)\cdot\dot{q}(t).\label{eq:condInv}
\end{multline}
for all $s\in P$ and all $t\in T'$.
\end{defn}

\begin{thm}[Necessary condition of invariance for problem \eqref{eq:optCtrlProbl}]
\label{th:cnsi}If the augmented functional \eqref{eq:COA1} is invariant
in the sense of Def.~\ref{def:inv:gt1}, then for all weak Pontryagin
extremals $(q,u,p)$ such that $q\in\gsf(T',S')$, $u\in\gsf(T',U')$
and $p\in\gsf(T',A')$, we have
\begin{equation}
\frac{\partial{\mathcal{H}}}{\partial t}(\cdot,q,u,p)\frac{\partial\tau}{\partial s}(0,\cdot)+\frac{\partial{\mathcal{H}}}{\partial q}(\cdot,q,u,p)\cdot\frac{\partial\sigma}{\partial s}(0,q)+{\mathcal{H}}(\cdot,q,u,p)\frac{\diff{}}{\diff{t}}\frac{\partial\tau}{\partial s}(0,\cdot)=p\cdot\frac{\diff{}}{\diff{t}}\frac{\partial\sigma}{\partial s}(0,q)\label{eq:H4}
\end{equation}
\end{thm}

\begin{proof}
We obtain \eqref{eq:H4} by differentiating \eqref{eq:condInv} with
respect to $s$ at $s=0$ and then considering Rem.~\ref{rem:derSigTau}
and \eqref{wps}.
\end{proof}
Note that for the particular case of calculus of variations ($\varphi(t,q,u)=u$
and hence $\mathcal{H}=L+p\cdot\dot{q}$) one obtains from \eqref{eq:H4}
the necessary condition of invariance Lem.~\ref{lem:cnsi}.
\begin{defn}[Constants of Motion/conservation law for problem \eqref{eq:optCtrlProbl}]
\label{def:NCLs} We say that a function $t\in T'\mapsto C(t,q(t),u(t),p(t))\in\rti$
is a \emph{constant of motion on $T'$, $S'$, $U'$, $A'$} if along
any weak Pontryagin extremal $(q,u,p)$ such that $q\in\gsf(T',S')$,
$u\in\gsf(T',U')$ and $p\in\gsf(T',A')$, we have
\begin{equation}
\frac{\diff{}}{\diff{t}}C(t,q(t),u(t),p(t))=0\quad\forall t\in T'.\label{eq:CL:NC}
\end{equation}
\end{defn}

\noindent By the uniqueness of Thm\@.~\ref{thm:existenceUniquenessPrimitives},
condition \eqref{eq:CL:NC} implies that
\[
t\in J\mapsto C(t,q(t),u(t),p(t))\in\rti
\]
is constant along any interval $J\subseteq T'$.
\begin{thm}[Noether's theorem for optimal control]
\label{th:TNNC}If the augmented functional \eqref{eq:COA1} is invariant
in the sense of Def.~\ref{def:inv:gt1}, then the quantity defined
for all $q\in\gsf(T',S')$, $u\in\gsf(T',U')$, $p\in\gsf(T',A')$
and $t\in T'$ by
\begin{equation}
C(t,q(t),u(t),p(t)):={\mathcal{H}}(t,q(t),u(t),p(t))\frac{\partial\tau}{\partial s}(0,t)-p(t)\cdot\frac{\partial\sigma}{\partial s}(0,q(t))\quad\forall t\in T'\label{eq:H7}
\end{equation}
is a constant of motion \emph{on $T'$, $S'$, $U'$, $A'$}.
\end{thm}

\begin{proof}
By Thm.~\ref{prp:tdH} we have
\begin{align*}
\frac{\diff{}}{\diff{t}}\left({\mathcal{H}}(\cdot,q,u,p)\frac{\partial\tau}{\partial s}(0,\cdot)-p\cdot\frac{\partial\sigma}{\partial s}(0,q)\right) & =\frac{\partial{\mathcal{H}}}{\partial t}(\cdot,q,u,p)\frac{\partial\tau}{\partial s}(0,\cdot)+{\mathcal{H}}(\cdot,q,u,p)\frac{\diff{}}{\diff{t}}\frac{\partial\tau}{\partial s}(0,\cdot)-\\
 & \phantom{=}-\dot{p}\cdot\frac{\partial\sigma}{\partial s}(0,q)-p\cdot\frac{\diff{}}{\diff{t}}\frac{\partial\sigma}{\partial s}(0,q).
\end{align*}
Considering \eqref{eq:H4} and \eqref{wps}, we obtain:
\begin{align*}
\frac{\diff{}}{\diff{t}}\left({\mathcal{H}}(\cdot,q,u,p)\frac{\partial\tau}{\partial s}(0,\cdot)-p\cdot\frac{\partial\sigma}{\partial s}(0,q)\right) & =\frac{\partial\mathcal{H}}{\partial q}(\cdot,q,u,p)\cdot\frac{\partial\sigma}{\partial s}(0,q)+\dot{p}\cdot\frac{\partial\sigma}{\partial s}(0,q)=\\
 & =\left(\frac{\partial\mathcal{H}}{\partial q}(\cdot,q,u,p)+\dot{p}\right)\frac{\partial\sigma}{\partial s}(0,q)=0.
\end{align*}
\end{proof}
Repeating the usual calculations, it is possible to prove Thm\@.~\ref{thm:tn}
also as a particular case of Thm.~\ref{th:TNNC} (cf.~e.g.~\cite{MR2351637}).

\section{\label{sec:Examples-and-applications}Examples and applications}

\subsection{Modeling singular dynamical systems}

In this section we show several applications of the calculus of variations
with GSF we introduced above. As we already mentioned in the introduction,
we will \emph{not} consider mathematical models of singular dynamical
systems at the times when singularities occur. Indeed, this would
clearly require new physical ideas, e.g.~in order to consider the
nonlinear behavior of objects or materials for the entire duration
of the singularity. Like in every mathematical model, the correct
point of view concerns J.~von Neumann's \emph{reasonably wide area
}of applicability of a mathematical model:
\begin{quotation}
To begin, we must emphasize a statement which I am sure you have heard
be-fore, but which must be repeated again and again. It is that the
sciences do not try to explain, they hardly ever try to interpret,
they mainly make models. By a model is meant a mathematical construct
which, with the addition of some verbal interpretations, describes
observed phenomena. The justification of such a mathematical construct
is solely and precisely that it is expected to work - that is correctly
to describe phenomena from a reasonably wide area. \cite[pag.~492]{vonN}
\end{quotation}
Therefore, it is not epistemologically correct to use the theory described
in the present article to deduce a physical property of our modeled
systems when a singularity occurs. Stating it with a language typically
used in physics, \emph{we consider physical systems where the duration
of the singularity is negligible with respect to the durations of
the other phenomena that take place in the system}. Mathematically,
this means to consider as infinitesimal the duration of the singularities.
As a consequence, several quantities changing during this infinitesimal
interval of time have infinite derivatives. We can hence paraphrase
the latter sentence saying that the amplitude (of the derivatives)
of these physical quantities is much larger than all the other quantities
we can estimate in the system. However, this is a logical consequence
of our lacking of interest to include in our mathematical model what
happens during the singularity, constructing at the same time a beautiful
and sufficiently powerful mathematical model, and not because these
quantities really become infinite.

As we will see below for the singularly variable length pendulum (Sec.~\ref{subsec:SingVarLenPendulum}),
only a rigorous mathematical theory of infinitesimal quantities is
able to resolve, e.g., the apparent inconsistency of considering infinitesimal
oscillations and at the same time neglecting the infinitesimal duration
of the singularity.

On the other hand, the aforementioned ``wide area'' is now able
to include in a single equation the dynamical properties of our modeled
systems, without being forced to subdivide into cases of the type
``before/after the occurrence of each singularity''. Which is not
even reasonable in several cases, e.g.~in the motion of a particle
in a granular medium or of a ray of light in an optical fiber.

Finally, note that remaining far from the singularity (from the point
of view of the physical interpretation), is what allow us to state
that in several cases this kind of models are already experimentally
validated.

\subsection{How to use numerical solutions}

The applications we are going to present always end up with an ODE.
Existence and uniqueness of the solution is therefore guaranteed by
Thm\@.~\ref{thm:PicLindFiniteContr}. Clearly, if an explicit analytic
solution is possible, this is preferable, but this is a rare event.
On the other hand, in numerous cases we have to deal with a differential
equation whose singularities are generated by Heaviside's functions
or Dirac's deltas and linear or nonlinear operations with them. Embeddings
of these generalized functions as GSF are studied and quite well-known,
see Sec.~\ref{subsec:Embedding}. Their pictures, Fig.~\ref{fig:MollifierHeaviside},
are clearly obtained by numerical methods, but their properties can
be fully justified by suitable theorems, see e.g.~example \ref{enu:deltaCompDelta}.
In the same way, we can consider numerical solutions of our differential
equations as empirical laboratories helping us to guess suitable properties
and hence conjectures on the solutions. In principle, these properties
must be justified by corresponding theorems. From this point of view,
the fact that GSF share with ordinary smooth functions a lot of classical
theorems (such as the intermediate value, the extreme value, the mean
value, Taylor theorems, etc.) is usually of great help.

\subsection{\label{subsec:SingVarLenPendulum}Singularly variable length pendulum}

As a first example, we want to study the dynamics of a pendulum with
singularly variable length, e.g.~because it is wrapping on a parallelepiped
(see Fig.~\ref{fig:variable_length_pendulum}; see \cite{MaHoAh12}
for a similar but non-singular case).

The pendulum length function is therefore $\Lambda(\theta)=H(\theta_{0}-\theta)L_{1}+L_{2}$,
where $H$ is the (embedding of the) Heaviside function. We always
assume that $L_{1}$, $L_{2}\in\rti_{>0}$ are finite and non-infinitesimal
numbers. From this it follows that for all $\theta$, $H(\theta_{0}-\theta)>\frac{\diff{\rho}-L_{2}}{L_{1}}\approx-\frac{L_{2}}{L_{1}}$
and hence that also $\Lambda(\theta)>\diff{\rho}$ is invertible (recall
that $H$ has negative infinitesimal oscillations in an infinitesimal
neighborhood of the origin, see Fig.~\ref{fig:MollifierHeaviside}).

The kinetic energy is given by: 
\begin{equation}
T(\theta,\dot{\theta})=\frac{1}{2}m\dot{\theta}^{2}\Lambda(\theta)^{2}.\label{e1c}
\end{equation}
The potential energy (the zero level being the suspension point of
the pendulum) is: 
\begin{equation}
U(\theta)=-mg\Lambda(\theta)\cos\theta-mg(1-H(\theta_{0}-\theta))L_{1}\cos\theta_{0}.\label{e1d}
\end{equation}

\noindent Let us define the Lagrangian \textit{$L$} for this problem
as 
\begin{equation}
L(\theta,\dot{\theta}):=T(\theta,\dot{\theta})-U(\theta).\label{e1a}
\end{equation}
\begin{figure}
\centering{}\includegraphics[scale=0.2]{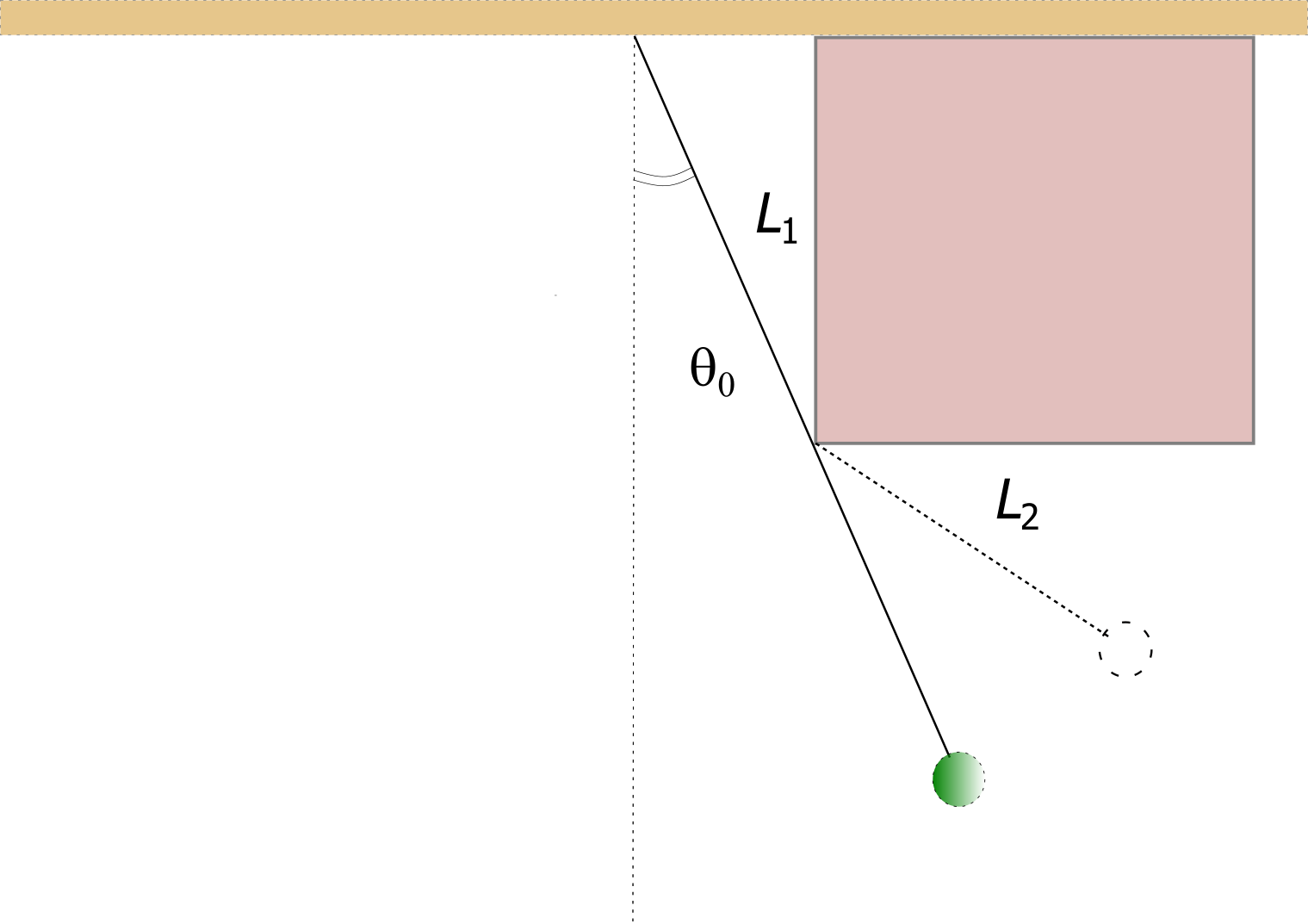}\caption{\label{fig:variable_length_pendulum} Oscillations of a pendulum wrapping
on a parallelepiped}
\end{figure}

\noindent The equation of motion is given by the Euler\textendash Lagrange
equation, Cor.~\ref{cor:16}, and can be written as: 
\begin{equation}
\frac{\partial L}{\partial\theta}=\frac{\diff{}}{\diff{t}}\frac{\partial L}{\partial\dot{\theta}}.\label{e1b}
\end{equation}

\noindent Thereby
\begin{equation}
\frac{\diff{}}{\diff{t}}\frac{\partial L}{\partial\dot{\theta}}=\frac{\diff{}}{\diff{t}}\frac{\partial}{\partial\dot{\theta}}\left(\frac{1}{2}m\dot{\theta}^{2}\Lambda(\theta)^{2}\right)=\frac{\diff{}}{\diff{t}}\left(m\dot{\theta}\Lambda(\theta)^{2}\right)=m\Lambda(\theta)^{2}\ddot{\theta}+2m\dot{\theta}\Lambda(\theta)\dot{\Lambda}(\theta),\label{e1e}
\end{equation}
where $\dot{\Lambda}(\theta):=\frac{\diff{}}{\diff{t}}\Lambda(\theta(t))$.
From \eqref{e1d}, the left side of the Euler\textendash Lagrange
equation \eqref{e1b} reduces to 
\begin{equation}
\frac{\partial L}{\partial\theta}=\frac{\partial T}{\partial\theta}+\frac{\partial(-U)}{\partial\theta}=m\dot{\theta}^{2}\Lambda(\theta)\Lambda'(\theta)+mg\Lambda'(\theta)\left(\cos\theta-\cos\theta_{0}\right)-mg\Lambda(\theta)\sin\theta,\label{e1f}
\end{equation}
where
\begin{equation}
\Lambda'(\theta)=\frac{\diff{}}{\diff{\theta}}\left(H(\theta_{0}-\theta)L_{1}+L_{2}\right)=-\delta(\theta_{0}-\theta)L_{1},\label{e1g}
\end{equation}
and $\delta$ is the Dirac delta function. We then obtain the following
equation of motion: 
\begin{equation}
m\dot{\theta}^{2}\Lambda(\theta)\Lambda'(\theta)+mg\Lambda'(\theta)\left(\cos\theta-\cos\theta_{0}\right)-mg\Lambda(\theta)\sin\theta=m\Lambda(\theta)^{2}\ddot{\theta}+2m\dot{\theta}\Lambda(\theta)\dot{\Lambda}(\theta).\label{e1h}
\end{equation}
Taking into account that $\dot{\Lambda}(\theta)=\Lambda'(\theta)\dot{\theta}$,
we finally obtain the equation of motion for the variable length pendulum:
\begin{equation}
\ddot{\theta}+\dot{\theta}\frac{\dot{\Lambda}(\theta)}{\Lambda(\theta)}-g\frac{\dot{\Lambda}(\theta)}{\dot{\theta}\Lambda(\theta)^{2}}\left(\cos\theta-\cos\theta_{0}\right)+\frac{g}{\Lambda(\theta)}\sin\theta=0.\label{e1j}
\end{equation}
Note in \eqref{e1j} the nonlinear operations on the Sobolev-Schwartz
distribution $\Lambda$, on the GSF $\theta$ and the composition
$t\mapsto\Lambda(\theta(t))$. Since the Lagrangian $L$ is autonomous,
in the usual way Noether's Thm.~\ref{thm:tn} implies that the mechanical
energy of the system is a constant of motion: 
\begin{equation}
E(\theta,\dot{\theta})=T(\theta,\dot{\theta})+U(\theta)=\frac{1}{2}m\dot{\theta}^{2}\Lambda(\theta)^{2}-mg\Lambda(\theta)\cos\theta-mg(1-H(\theta_{0}-\theta))L_{1}\cos\theta_{0}=\textnormal{constant}.\label{e1k}
\end{equation}

Even if local (in time) existence and uniqueness of the solution $\theta$
of \eqref{e1k} is easily granted by Thm\@.~\ref{thm:PicLindFiniteContr}
and by the extreme value Thm.~\ref{thm:extremeValues}, the existence
of a global solution is not so easy to show (see in \cite{baglinipaolo}
the general theory) and is out of the scope of the present work. Before
showing the numerical solution of \eqref{e1j}, let us consider the
simplest case of the dynamics far from the singularity and that of
small oscillations. The former, as we mentioned above, is the only
physically meaningful one.

\subsubsection{Description far from singularity and small oscillations}

For simplicity, let us consider the simplest case $\theta_{0}=0$.
Furthermore, we consider that the pendulum is initially at rest and
starts its movement at $t_{1}\in\rti$. The initial conditions we
use are hence: 
\begin{equation}
\begin{cases}
\theta(t_{1})=\theta_{1};\\
\dot{\theta}(t_{1})=0,
\end{cases}\label{e1l}
\end{equation}
with $\theta_{1}<0$. Assuming that at some time $t_{3}\in\rti$ we
have $\theta(t_{3})>0$, by the intermediate value theorem for GSF,
there exists $t_{2}\in\rti$ where we have the singularity, i.e.~$\theta(t_{2})=0$
and the length of the pendulum \emph{smoothly} changes from $L_{1}+L_{2}$
to $L_{2}$ after the rope touches the parallelepiped. By example
\ref{enu:deltaCompDelta}, it follows that this change happens in
an infinitesimal interval, because by contradiction it is possible
to prove that if $\Lambda(\theta)\in(L_{2},L_{1}+L_{2})$, then $|\theta|\le\frac{-1}{\log\diff{\rho}}\approx0$.
\begin{defn}
Let $x$, $y\in\rti$. We say that $x$ \emph{is far from $y$ }if
$|x-y|\ge\diff{\rho}^{a}$ for all $a\in\R_{>0}$. More generally,
we say that $x$ \emph{is far from} $y$ \emph{with respect to the
class of infinitesimals $\mathcal{I}\subset\rti$}, if $|x-y|\ge i$
for all $i\in\mathcal{I}$.
\end{defn}

\noindent For example, if $|x|\ge r$ for some $r\in\R_{>0}$, then
$x$ is far from $0$, but also the infinitesimal number $x=\frac{-1}{k\log\diff{\rho}}$
($k\in\R_{>0})$ is far from $0$; similarly, the infinitesimal $x=\frac{-1}{k\log\log\diff{\rho}}$
if far from $0$ with respect to all the infinitesimals of the type
$\frac{-1}{h\log\diff{\rho}}$ for $h\in\R_{>0}$.

If $\theta$ is far from $0$ and $b\ge\diff{\rho}^{-a}$, $a\in\R_{>0}$,
then $|b\theta|\ge\diff{\rho}^{-a}|\theta|\ge\diff{\rho}^{-a/2}\ge1$.
Therefore, from example \ref{enu:deltaCompDelta}, it follows that
$H(-\theta)\in\{0,1\}$ and hence $\dot{\Lambda}(\theta(t))=0$. Equation
\eqref{e1j} becomes
\begin{equation}
\theta(t)\text{ is far from }0\ \Rightarrow\ \begin{cases}
\ddot{\theta}+\frac{g}{L_{1}+L_{2}}\sin\theta(t)=0 & \text{if }\theta(t)<0,\\
\ddot{\theta}+\frac{g}{L_{2}}\sin\theta(t)=0 & \text{if }\theta(t)>0.
\end{cases}\label{eq:ODEfar}
\end{equation}
If we assume that $\theta(t_{1})=\theta_{1}<0$ and $\theta(t_{3})>0$
are far from $0$, the sharp continuity of $\theta$ yields the existence
of $\delta_{1}$, $\delta_{3}\in\rti_{>0}$ such that
\begin{align}
\forall t & \in[t_{1},t_{1}+\delta_{1})\cup(t_{3}-\delta_{3},t_{3}]:\ \theta(t)\text{ is far from }0\nonumber \\
\forall t & \in[t_{1},t_{1}+\delta_{1}):\ \theta(t)<0\label{eq:intFar}\\
\forall t & \in(t_{3}-\delta_{3},t_{3}]:\ \theta(t)>0\nonumber 
\end{align}
(and hence $t_{2}\notin[t_{1},t_{1}+\delta_{1})\cup(t_{3}-\delta_{3},t_{3}]$
because $\theta(t_{2})=0$). Assuming that $t_{1}$, $t_{3}$ are
far from $t_{2}$, without loss of generality we can also assume to
have taken $\delta_{i}$ so small that also $t_{1}+\delta_{1}$ and
$t_{3}-\delta_{3}$ are far from $t_{2}$.

We now employ the non Archimedean framework of $\rti$ in order to
formally consider small oscillations, i.e.~$\theta_{1}\approx0$.
We first note that we cannot only assume $\theta_{1}$ infinitesimal,
because if $\theta_{1}$ is not far from $0$ then our solution will
not be physically meaningful. However, we already have seen that we
can take $\theta_{1}$ far from $0$ and infinitesimal at the same
time, e.g.~$\theta_{1}=\frac{-1}{\log\diff{\rho}}$. In other words,
$\theta_{1}$ is a ``large'' infinitesimal with respect to all the
infinitesimals of the form $\diff{\rho}^{a}$. Let $\vartheta_{1}$,
$\vartheta_{3}$ be the solution of the linearized problems
\begin{equation}
\begin{cases}
\ddot{\vartheta}_{1}+\frac{g}{L_{1}+L_{2}}\vartheta_{1}=0, & t_{1}\le t<t_{1}+\delta_{1}\\
\dot{\vartheta}_{1}(t_{1})=0,\ \vartheta_{1}(t_{1})=\theta_{1}
\end{cases}\quad\begin{cases}
\ddot{\vartheta}_{3}+\frac{g}{L_{2}}\vartheta_{3}=0, & t_{3}-\delta_{3}<t\le t_{3}\\
\dot{\vartheta}_{1}(t_{3})=\dot{\theta}(t_{3}),\ \vartheta_{1}(t_{3})=\theta(t_{3}),
\end{cases}\label{eq:linODE}
\end{equation}
i.e.~$\vartheta_{1}(t)=\theta_{1}\cos\left(\omega(t-t_{1})\right)$,
$\omega:=\sqrt{\frac{g}{L_{1}+L_{2}}}$, and $\vartheta_{3}(t)=\theta(t_{3})\cos\left(\omega'(t_{3}-t)\right)-\frac{\dot{\theta}(t_{3})}{\omega'}\sin\left(\omega'(t_{3}-t)\right)$,
$\omega'=\sqrt{\frac{g}{L_{2}}}$. We want to show that $\theta(t)\approx\vartheta_{i}(t)$
at least in an infinitesimal neighborhood of $t_{1}$ and $t_{3}$
exactly because $\theta_{1}\approx0$. For simplicity, we proceed
only for $\vartheta_{1}$, the other case being similar. For any $t\in[t_{1},t_{1}+\delta_{1})$,
we have that $\theta(t)<0$ is far from $0$ from \eqref{eq:intFar},
and hence $\ddot{\theta}+\frac{g}{L_{1}+L_{2}}\sin\theta(t)=0$ from
\eqref{eq:ODEfar}. Recalling the initial conditions, we obtain
\[
\theta(t_{1}+h)-\theta_{1}=-\omega^{2}\int_{t_{1}}^{t_{1}+h}\sin\theta(s)\,\diff{s}\quad\forall h\in(0,\delta_{1}).
\]
Similarly, integrating \eqref{eq:linODE}, we get
\[
\vartheta_{1}(t_{1}+h)-\theta_{1}=-\omega^{2}\int_{t_{1}}^{t_{1}+h}\vartheta(s)\,\diff{s}\quad\forall h\in(0,\delta_{1}).
\]
Using Taylor Thm.~\ref{thm:Taylor} at $t_{1}$ with increment $h$
of these integral GSF, we obtain
\begin{align*}
\theta(t_{1}+h)-\vartheta(t_{1}+h) & =-\omega^{2}\left\{ \sin\theta_{1}-\theta_{1}+h\cos\theta_{1}\cdot\dot{\theta}(t_{1})-h\dot{\vartheta}_{1}(t_{1})+h^{2}R(h)\right\} =\\
 & =-\omega^{2}\left\{ \sin\theta_{1}-\theta_{1}+h^{2}R(h)\right\} ,
\end{align*}
where $R(-)$ is a suitable GSF. Thereby, $\theta(t_{1}+h)-\vartheta(t_{1}+h)\approx-\omega^{2}h^{2}R(h)\approx0$
for all $h\approx0$ sufficiently small because $\sin\theta_{1}\approx\theta_{1}$
since $\theta_{1}\approx0$.

Since each $t\in[t_{1},t_{1}+\delta_{1})\cup(t_{3}-\delta_{3},t_{3}]$
is far from $t_{2}$, we can also formally join the two solutions
$\vartheta_{i}$ using the Heaviside's function:
\begin{equation}
\theta(t)\approx\vartheta_{1}(t)+H(t_{2}-t)\left(\vartheta_{3}(t)-\vartheta_{1}(t)\right)\quad\forall t\in[t_{1},t_{1}+h)\cup(t_{3}-h,t_{3}].\label{eq:join}
\end{equation}
For the motivations previously stated, this infinitesimal approximation
cannot be extended to a neighborhood of $t_{2}$.

We close this section noting that all these deductions can be repeated
using any GSF $H\in\gsf(\rti,\rti)$ satisfying for all $x$ far from
zero $H(x)=1$ if $x>0$ and $H(x)=0$ if $x<0$. This allows us to
consider e.g.~a GSF that does not show the infinitesimal oscillations
of the Heaviside's function in an infinitesimal neighborhood of the
origin.

\subsubsection{Numerical Solution}

The numerical solution of equation \eqref{e1j} has been computed
using Mathematica Solver NDSolve (see \cite{NDSolve}). Initial conditions
we used are: 
\begin{equation}
\begin{cases}
\theta(0)=0\text{ rad},\\
\dot{\theta}(0)=1\text{ rad/s}.
\end{cases}\label{eq:initial_cond_case_1}
\end{equation}
The graph of $\theta(t)$, and its derivative $\dot{\theta}(t)$,
based on the Mathematica definitions of $H(x)$ and $\delta(x)$ (see
\cite{Hdelta}) are shown in Figure \ref{fig:damping osc - solution and first derivative}.
\begin{figure}
\centering{}\includegraphics[width=0.75\textwidth]{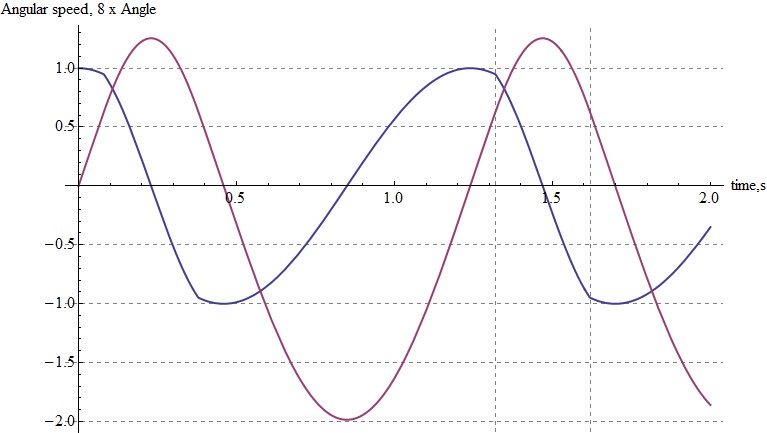}\caption{\label{fig:damping osc - solution and first derivative}8 times re-scaled
solution (violet line) in radians and its derivative in rad/s (at
$\theta=\theta_{0}=\pi/40$ rad we can see a corner point). Parameters
used: $L_{1}=0.4\text{ m}$, $L_{2}=0.2\text{ m}$, $g=9.8\text{ m/\ensuremath{\text{s}^{2}}}$.}
\end{figure}

In Figure \ref{fige2} we show the second derivative graph. Directly
from \eqref{e1j} and \eqref{e1g} we can prove that when $\theta(t)=\theta_{0}$,
$\ddot{\theta}(t)$ is an infinite number and hence $\dot{\theta}(t)$
has a corner point. Note finally that no infinitesimal oscillations
are represented in an infinitesimal neighborhood of the singularities.
This is due to the Mathematica implementation of $H$ and $\delta$:
we could hence say that these graphs represent the solution far from
the singularities.
\begin{figure}
\centering{}\includegraphics[width=0.75\textwidth]{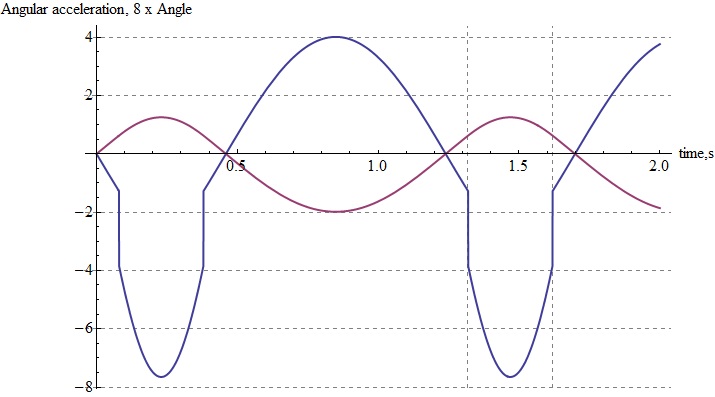}\caption{\label{fige2}8 times re-scaled solution in radians (violet line)
and its second derivative in rad/$\text{s}^{2}$}
\end{figure}

\subsection{Oscillations damped by two media}

\begin{figure}
\centering{}\includegraphics[scale=0.2]{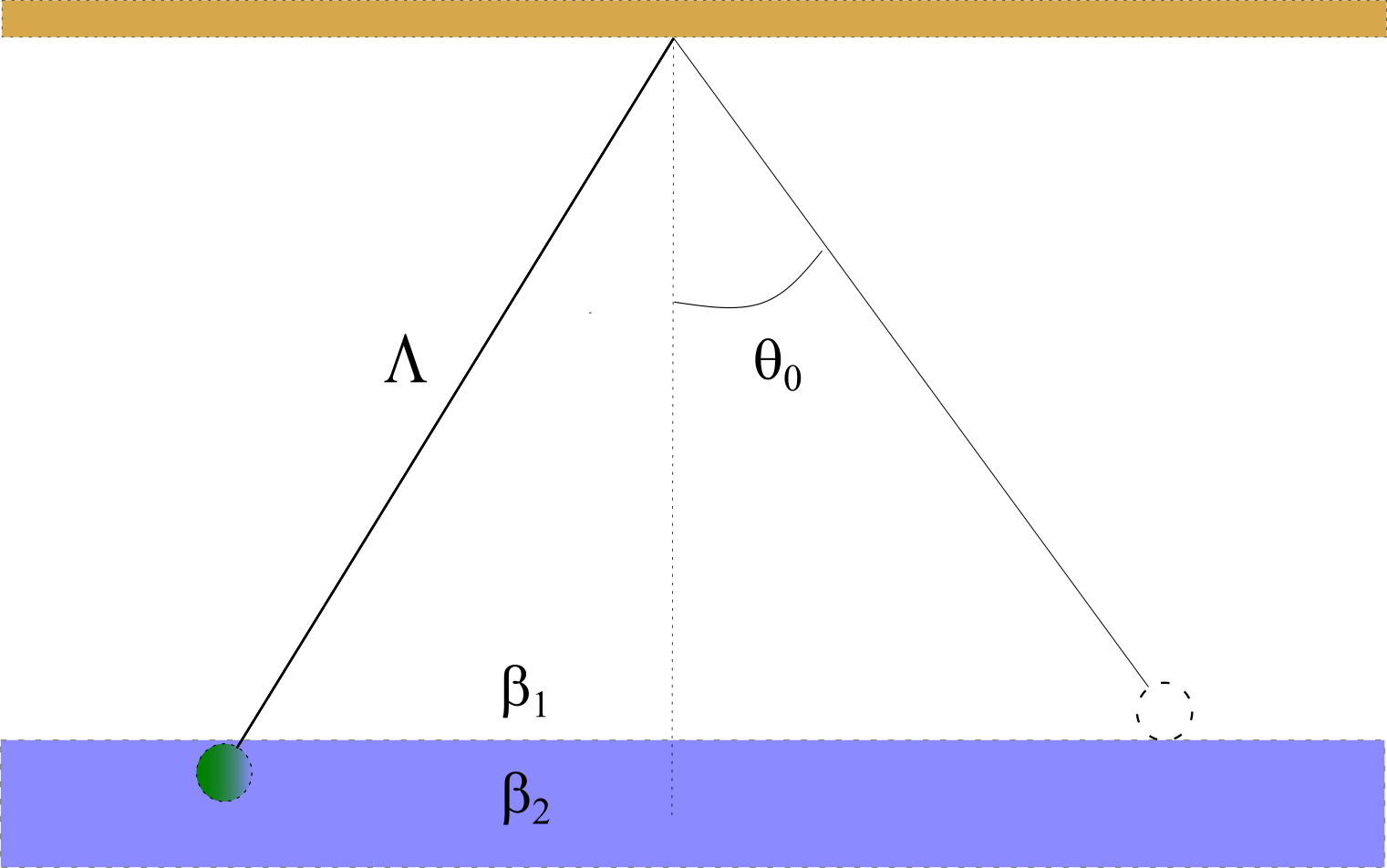}\caption{\label{fig:water_oscillations}Simple pendulum moving in two media}
\end{figure}

The second example concerns oscillations of a pendulum in the interface
of two media. Since we have to neglect the dynamics occurring at singular
times (i.e.~at the changing of the medium), this can be considered
only a toy model approximating the case of a very small but sufficiently
heavy moving particle.

We hence want to model the system employing a ``jump'' in the damping
coefficient $\beta$, i.e.~a finite change occurring in an infinitesimal
interval of time, see Fig.~\ref{fig:water_oscillations}. Since the
frictional forces acting in this case are not conservative, it is
well-known that the Euler-Lagrange equations cannot be assumed to
describe the dynamics of the system and we have to use the D'Alembert
principle Cor.~\ref{cor:DalembertDiff}; see \cite{LazoCesar} for
a deeper study.

The kinetic energy is given by: 
\begin{equation}
T(\dot{\theta})=\frac{1}{2}m\dot{\theta}^{2}\Lambda^{2},\label{e2c}
\end{equation}
and the potential energy (the zero level is the suspension point of
the pendulum) is: 
\begin{equation}
U(\theta)=-mg\Lambda\cos\theta.\label{e2d}
\end{equation}

\noindent In case of fluid resistance proportional to the velocity,
we can introduce the generalized forces $Q$ as: 
\begin{equation}
Q(\dot{\theta})=-r\Lambda^{2}\dot{\theta},\label{e2e}
\end{equation}
where $r$ is a proportional coefficient depending on the media. Let's
define the Lagrangian \textit{$L$} as 
\begin{equation}
L(\theta,\dot{\theta}):=T(\dot{\theta})-U(\theta).\label{e2a}
\end{equation}
We hence assume that the equation of motion for this non-conservative
system is given by the D'Alembert's principle, so that Cor.~\ref{cor:DalembertDiff}
gives
\begin{equation}
\frac{\diff{}}{\diff{t}}\frac{\partial L}{\partial\dot{\theta}}-\frac{\partial L}{\partial\theta}=Q.\label{e2b}
\end{equation}
Inserting \eqref{e2c}, \eqref{e2d} and \eqref{e2e} into \eqref{e2b}
we obtain the following equation of motion: 
\begin{equation}
m\Lambda^{2}\ddot{\theta}+mg\Lambda\sin\theta=-r\Lambda^{2}\dot{\theta}.\label{e2f}
\end{equation}
By introducing the damping coefficient $\beta(\theta):=r(\theta)/(2m)$
(we clearly assume that the mass $m\in\rti>0$ is invertible) we obtain
the classical form of the equation of motion for damped oscillations:
\begin{equation}
\ddot{\theta}+2\beta(\theta)\dot{\theta}+\frac{g\sin\theta}{\Lambda}=0.\label{eq:damping_equation}
\end{equation}

If the pendulum crosses the boundary between two media (see Fig.~\ref{fig:water_oscillations})
with damping coefficients $\beta_{1}$ and $\beta_{2}$, we can model
the system using the Heaviside function $H$:
\begin{equation}
\beta(\theta)=\beta_{1}+\left(H(\theta+\theta_{0})-H(\theta-\theta_{0})\right)(\beta_{2}-\beta_{1}),\label{eq:beta_dependence}
\end{equation}
where $\theta=\pm\theta_{0}$ are the angles at which we have the
changing of the medium (singularities).

The numerical solution of \eqref{eq:damping_equation} with $\beta$
defined by \eqref{eq:beta_dependence} and initial conditions \eqref{eq:initial_cond_case_1}
is presented in Fig.~\ref{fig:Solution-beta-gen}. The numerical
solution has been computed using Mathematica Solver NDSolve, but with
an implementation of the Heaviside's function $H$ corresponding to
Thm\@.~\ref{thm:embeddingD'}, i.e.~as represented in Fig.~\ref{fig:MollifierHeaviside}.

\begin{figure}
\centering{}\includegraphics[scale=0.3]{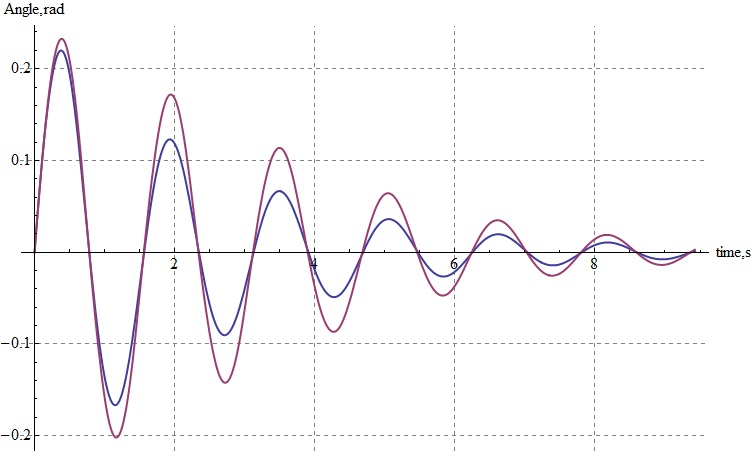}\caption{\label{fig:Solution-beta-gen}Solution $\theta$ of \eqref{eq:damping_equation}
(blue line). For comparison, the violet line is the case $\beta=\text{const}.=\beta_{1}$.
Used parameters: $\beta_{1}=0.0064$ (air), $\beta_{2}=0.3859$ (water),
$\theta_{0}=\pi/40\text{ rad}$, $\Lambda=0.6\text{ m}$, $g=9.8\text{ m/\ensuremath{\text{s}^{2}}}$.}
\end{figure}

We also include the graphs of the angular frequency $\dot{\theta}$
(which shows corner points) and of the angular acceleration $\ddot{\theta}$
(which shows ``jumps'', i.e.~infinite derivatives at singular times,
as we can directly see from \eqref{eq:damping_equation} and \eqref{eq:beta_dependence}).

Note that, to simplify our analysis, we considered a fixed length
$\Lambda$ for the pendulum. However, in the framework of the suggested
model, we could also consider a singularly variable length pendulum
dumped by two media.

Once again, we note that the model can be refined by considering,
instead of the Heaviside's function, any GSF $H\in\gsf(\rti,\rti)$
satisfying for all $x$ far from zero $H(x)=1$ if $x>0$ and $H(x)=0$
if $x<0$; e.g.~this refinement could allow one to consider different
``large'' infinitesimal neighborhoods of the origin in order to
take into account a better modeling near the singularities.

\begin{figure}
\centering{}\includegraphics[scale=0.3]{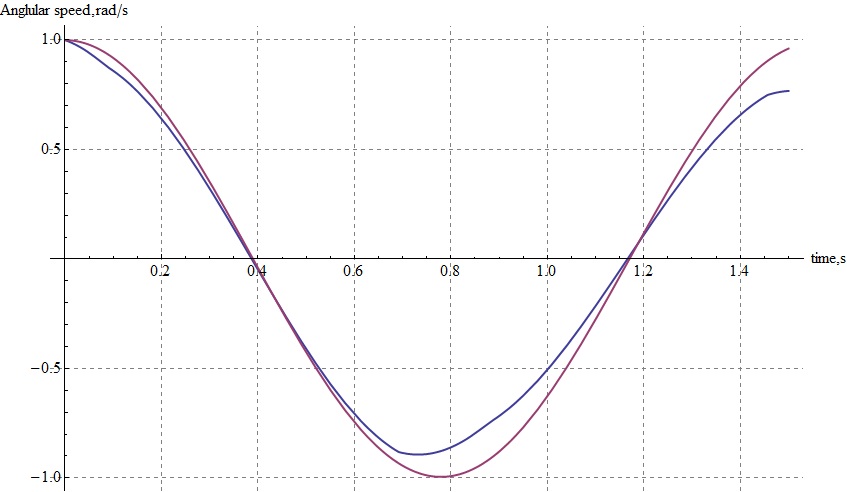}\includegraphics[scale=0.3]{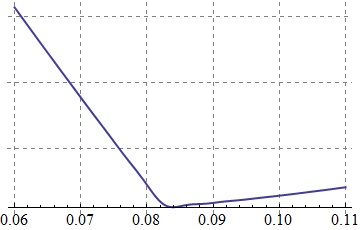}\caption{\label{fig:Difference-in-Derivative-Damping}First derivative $\dot{\theta}$
of the solution of \eqref{eq:damping_equation} (blue line). The case
with $\beta=\text{const}=\beta_{1}$ is also shown for comparison
(violet line). Note the corner points at the singular moments, for
example at $t=0.083\text{ s}$ (scaled in the right figure).}
\end{figure}

\begin{figure}
\begin{centering}
\includegraphics[scale=0.3]{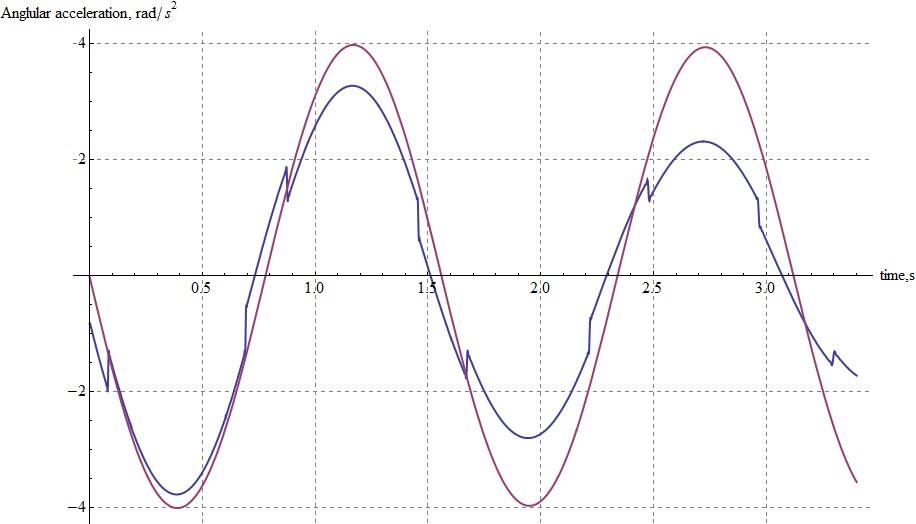}\includegraphics[scale=0.3]{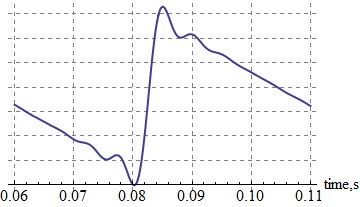}
\par\end{centering}
\caption{Second derivative $\ddot{\theta}$ of the solution of \eqref{eq:damping_equation}
(blue line). The case with $\beta=\text{const}=\beta_{1}$ is also
shown for comparison (violet line). Note the \textquotedblleft jumps\textquotedblright{}
at the singular moments, for example at $t=0.083\text{ s}$ (scaled
in the right figure). The infinitesimal oscillations are caused by
the embedding as GSF of the Heaviside function.}
\end{figure}

\subsection{\label{subsec:Pais=002013Uhlenbeck-oscillator}Pais\textendash Uhlenbeck
oscillator with singular frequencies}

An interesting higher order system with several applications in Physics
is the Pais\textendash Uhlenbeck (PU) oscillator \cite{PU}. The PU
oscillator is usually studied as a toy model of recent higher-derivative
theories like gravity, quantum-mechanical field theories and others
(see for example \cite{PU0,PU1,PU2,PU3}), but it can also describe
simple real word systems such as, for example, electrical circuits
\cite{PU3}. Actually, the PU oscillator is a bi-harmonic oscillator
with two natural frequencies and is described by a fourth-order equation
of motion. An example of PU GSF oscillator would be a PU electric
circuit with singularly variable frequencies. We can think at an (exogenous
or endogenous) force acting on the system and rapidly changing these
frequencies.

The oscillator we consider is given by the following Lagrangian function
\begin{equation}
L(t,q,\dot{q},\ddot{q})=\frac{m}{2}\left[\ddot{q}^{2}-\left(\omega_{1}^{2}(t)+\omega_{2}^{2}(t)\right)\dot{q}^{2}+\omega_{1}^{2}(t)\omega_{2}^{2}(t)q^{2}\right]\quad\forall t\in[t_{1},t_{3}]\label{PU1}
\end{equation}
where the natural frequencies 
\begin{equation}
\begin{split}\omega_{1}(t) & =\omega'_{1}+H(t_{\text{s}}-t)\left(\hat{\omega}_{1}-\omega'_{1}\right),\\
\omega_{2}(t) & =\omega'_{2}+H(t_{\text{s}}-t)\left(\hat{\omega}_{2}-\omega'_{2}\right)
\end{split}
\label{PU1b}
\end{equation}
changes from $\omega'_{i}$ to $\hat{\omega}_{i}$ ($i=1,2$) at $t=t_{\text{s}}\in(t_{1},t_{2})$,
and $m$, $\omega'_{1}$, $\omega'_{2}$, $\hat{\omega}_{1}$, $\hat{\omega}_{2}\in\rti_{>0}$
are constants. The resulting equation of motion for the PU oscillator
can be obtained from the Euler-Lagrange equation Cor.~\ref{cor:16}.
We have: 
\begin{equation}
q^{(4)}+\left(\omega_{1}^{2}+\omega_{2}^{2}\right)\ddot{q}+2\left(\omega_{1}\dot{\omega}_{1}+\omega_{2}\dot{\omega}_{2}\right)\dot{q}+\omega_{1}^{2}\omega_{2}^{2}q=0,\label{PU2}
\end{equation}
where 
\begin{equation}
\dot{\omega}_{i}(t)=-\delta(t_{\text{s}}-t)\left(\hat{\omega}_{i}-\omega'_{i}\right)\qquad(i=1,2).\label{PU2b}
\end{equation}
Once again, note the nonlinear operations in \eqref{PU2} on the Sobolev-Schwartz
distributions $\omega_{i}$.

We now consider an open set of the form $T':=(t_{1},t_{\text{s}}-\delta)\cup(t_{\text{s}}+\delta,t_{2})$
and assume that
\[
\exists b\in\R_{>0}\,\forall t\in T'\,\forall s\in(-\diff{\rho}^{b},\diff{\rho}^{b}):\ t+s\text{ is far from }t_{\text{s}}
\]
(in particular any $t\in T'$ is far from $t_{\text{s}}$). For example,
it suffices to take $\delta\ge-\frac{1}{h\log\diff{\rho}}$, $h\in\R_{>1}$,
so that if $t\in(t_{1},t_{\text{s}}-\delta)$, then $t_{\text{s}}-t-s\ge\delta-s\ge-\frac{1}{h\log\diff{\rho}}-\diff{\rho}^{b}\ge-\frac{1}{(h-1)\log\diff{\rho}}\ge\diff{\rho}^{a}$
for any $a$, $b\in\R_{>0}$; similarly we can proceed if $t\in(t_{\text{s}}+\delta,t_{2})$.
Then $\omega_{i}(t+s)=\omega_{i}(t)$ for all $t\in T'$ and all $s\in(-\diff{\rho}^{b},\diff{\rho}^{b})=:P$,
and hence the Lagrangian \eqref{PU1} is invariant for the translation
$\tau(t,s):=t+s$ defined only for $t\in T'$ and with parameter $s\in P$.
From the Noether's Thm.~\ref{thm:tn}, we hence obtain the following
conserved quantity on the intervals $(t_{1},t_{\text{s}}-\delta)$
and $(t_{\text{s}}+\delta,t_{2})$: 
\begin{equation}
E=\frac{m}{2}\left[2\dot{q}q^{(3)}-\ddot{q}^{2}+\left(\omega_{1}^{2}+\omega_{2}^{2}\right)\dot{q}^{2}+\omega_{1}^{2}\omega_{2}^{2}q^{2}\right]\label{PU3}
\end{equation}
that coincides with the total energy for the PU oscillator \cite{PU2,PU}.
As it is shown in Fig.~\ref{fig:energy}, the total energy \eqref{PU3}
is not preserved on the entire $T'=(t_{1},t_{\text{s}}-\delta)\cup(t_{\text{s}}+\delta,t_{2})$
and we have a decreasing of energy switching from $\omega'_{i}$ to
$\hat{\omega}_{i}>\omega'_{i}$.

For the numerical solution of the PU oscillator's equation of motion
\eqref{PU2}, we used Mathematica Solver NDSolve (see \cite{NDSolve})
and Mathematica implementation of $H$ and $\delta$ (see \cite{Hdelta}),
and considered the initial conditions:

\begin{equation}
\begin{cases}
q(0)=1\text{ rad},\\
\dot{q}(0)=2\text{ rad/s},\\
\ddot{q}(0)=0\text{ rad/s}^{2},\\
q^{(3)}(0)=1\text{ rad/s}^{3}.
\end{cases}\label{eq:PU-IC}
\end{equation}
See Fig.~\ref{fig:PU} and Fig.~\ref{fig:PUder34}, where we can
note the corner point for the third derivative $q^{(3)}$ and the
jump for the fourth derivative $q^{(4)}$ at $t=t_{\text{s}}$. The
analytical solution of \eqref{PU2} for constant $\omega_{i}$ is
given by $q(t;\omega_{1},\omega_{2}):=A_{1}\sin(\omega_{1}t+\phi_{1})+A_{2}\sin(\omega_{2}t+\phi_{2})$,
where $A_{i}$ and $\phi_{i}$ are integration constants that for
our initial conditions \eqref{eq:PU-IC} are given by $A_{1}=\pm6.02827$,
$A_{2}=\pm1.81181$, $\phi_{1}=-2.88742$ or $\phi_{1}=0.254175$,
$\phi_{2}=0.288674$ or $\phi_{2}=-2.85292=0.288674-\pi$. Exactly
as we did in Sec.~\ref{subsec:SingVarLenPendulum}, we can hence
prove that $q(t)=q(t;\omega'_{1},\omega'_{2})$ for $t\in(t_{1},t_{\text{s}}-\delta)$
and $q(t)=q(t;\hat{\omega}_{1},\hat{\omega}_{2})$ for $t\in(t_{\text{s}}+\delta,t_{2})$.
The non-singular solution $q(-;\omega'_{1},\omega'_{2})$ is also
presented as a violet line in Fig.~\ref{fig:PU}.

\begin{figure}
\centering{}%
\begin{tabular}{|c|c|}
\hline 
\includegraphics[scale=0.25]{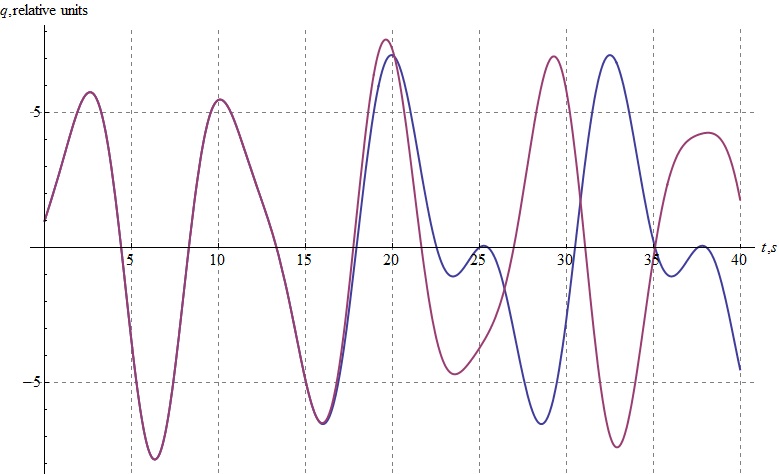} & \includegraphics[scale=0.25]{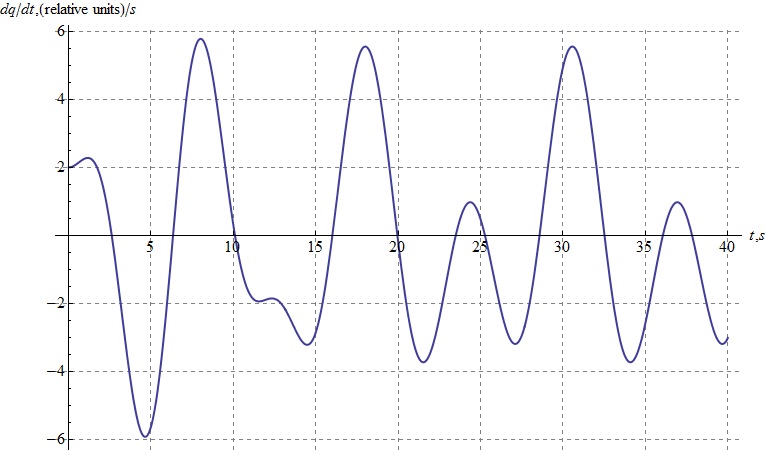}\tabularnewline
\hline 
\hline 
\multicolumn{2}{|c|}{\includegraphics[scale=0.25]{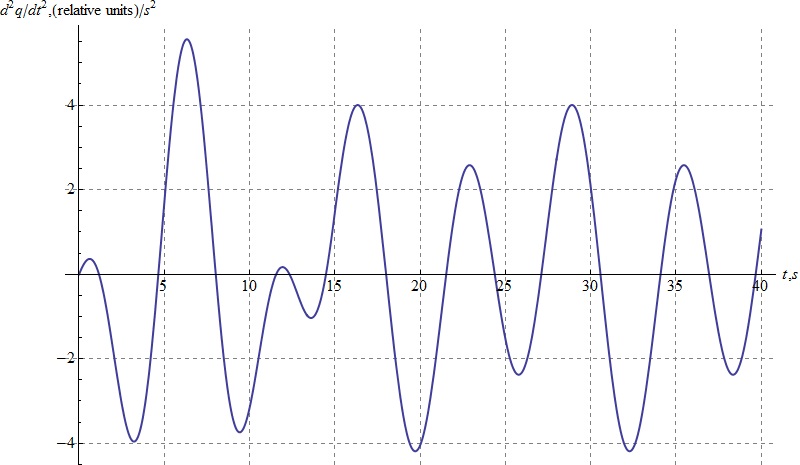}}\tabularnewline
\hline 
\end{tabular}\caption{\label{fig:PU}Numerical solution of the PU oscillator and its derivatives
$\dot{q}$, $\ddot{q}$. Switching time is at $t_{\text{s}}=15\text{ s}$,
$m=1\text{ kg}$, $\omega'_{1}=0.5\text{ rad/s}$, $\hat{\omega}_{1}=0.7\text{ rad/s}$,
$\omega'_{2}=1\text{ rad/s}$, $\hat{\omega}_{2}=1.2\text{ rad/s}$.}
\end{figure}

\begin{figure}
\centering{}%
\begin{tabular}{|c|c|}
\hline 
\includegraphics[scale=0.25]{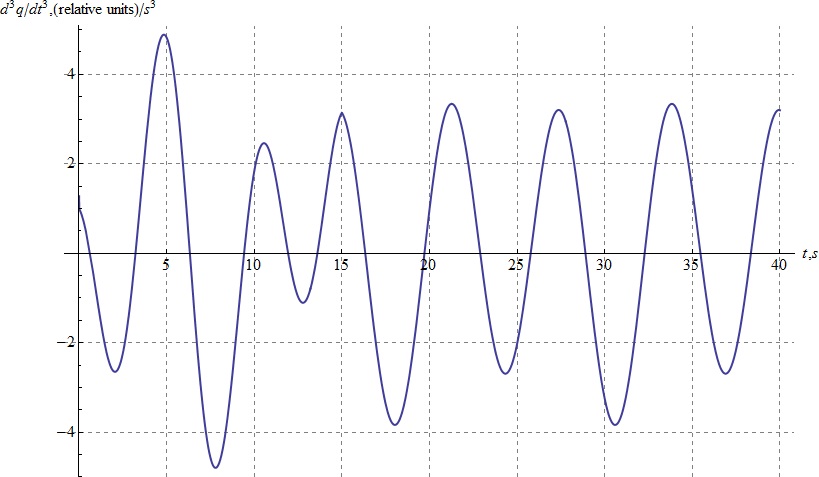} & \includegraphics[scale=0.25]{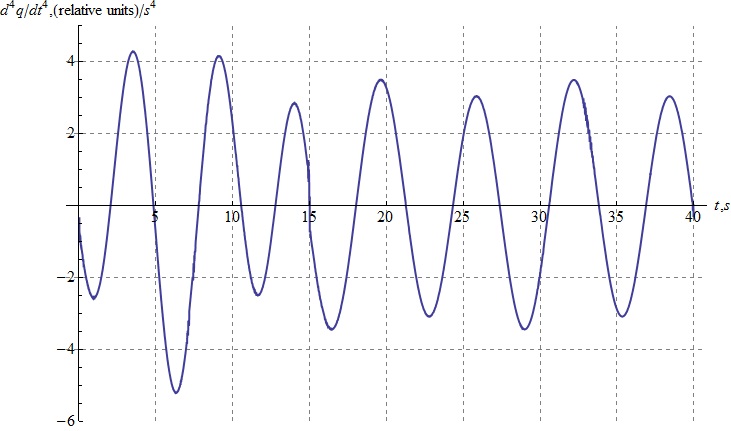}\tabularnewline
\hline 
\end{tabular}\caption{\label{fig:PUder34}Third and fourth derivatives of the numerical
solution of the PU oscillator.}
\end{figure}

\begin{figure}
\centering{}\includegraphics[scale=0.6]{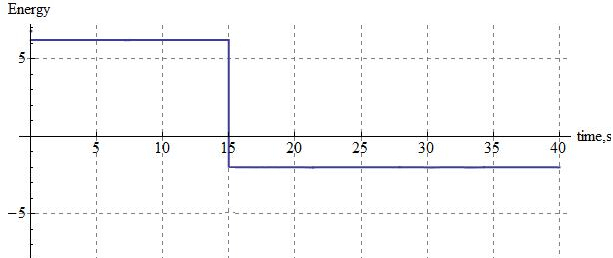}\caption{\label{fig:energy}Energy of the PU oscillator.}
\end{figure}

\section{Conclusions}

The present paper concretely shows the possibility to develop calculus
of variations and optimal control for a class of generalized functions
extending Sobolev-Schwartz distributions. On the one hand, we have
clearly only presented basic results, but, on the other hand, this
paves the way for a lot of further results in pure and applied mathematics,
theoretical physics and engineering.

We have also shown that generalized smooth function theory has features
that closely resemble classical smooth functions. In contrast, some
differences have to be carefully considered, such as the fact that
the new ring of scalars $\rti$ is not a field, it is not totally
ordered, it is not ordered complete, so that its theory of supremum
and infimum is more involved (see \cite{MTAG}), and its intervals
are not connected in the sharp topology because the set of all the
infinitesimals is a clopen set. Almost all these properties are necessarily
shared by other non-Archimedean rings because their opposite are incompatible
with the existence of infinitesimal numbers.

Conversely, the ring of Robinson-Colombeau generalized numbers $\rti$
is a framework where the use of infinitesimal and infinite quantities
is available, it is defined using elementary mathematics, and with
a strong connection with infinitesimal and infinite functions of classical
analysis (see Def.~\ref{def:nonArchNumbs}). As we have shown in
Sec.~\ref{subsec:SingVarLenPendulum} with the singularly variable
length pendulum, this leads to a better understanding and opens the
possibility to define new models of physical systems.


\begin{thebibliography}{10}
\bibitem{PU2}Abakumova, V.A., Kaparulin, D.S., Lyakhovich, S.L.,
\textit{\emph{Conservation laws and stability of higher derivative
extended Chern-Simons}}, Journal of Physics: Conf. Series 1337, 012001,
2019.

\bibitem{ArGaJu09}Aragona, J., Garcia, A.R.G., Juriaans, S.O., Generalized
solutions of a nonlinear parabolic equation with generalized functions
as initial data. Nonlinear Analysis, Theory, Methods and Applications,
Volume 71, Issue 11, 1 December 2009, Pages 5187-5207.

\bibitem{BeLuSq20}Benci, V., Luperi Baglini, L., Squassina, M., Generalized
solutions of variational problems and applications, Adv. Nonlinear
Anal. 9, pp. 124-147, 2020.

\bibitem{BeLeVi01}Bessis, D.N., Ledyaev, Yu.S., Vinter, R.B., Dualization
of the Euler and Hamiltonian inclusions, Nonlinear Analysis, Theory,
Methods and Applications, Volume 43, Issue 7, March 2001, Pages 861-882.

\bibitem{PU3}Biolek, Z., Biolek, D., Biolkova, V., \textit{\emph{Lagrangian
for Circuits with Higher-Order Elements}}, Entropy 21(11), 1059, 2019.

\bibitem{Bro99}Brogliato, B., \emph{Nonsmooth Mechanics. Models,
Dynamics and Control}. 2nd ed. Springer Verlag London, 1999.

\bibitem{ChMi94}Cheng, C.-W., Mizel, V.J., On the Lavrentiev phenomenon
for autonomous second-order integrands, Arch. Rational Mech. Anal.
126 (1994) 21-33.

\bibitem{C1}Colombeau, J.F., \emph{New generalized functions and
multiplication of distributions.} North-Holland, Amsterdam, 1984.

\bibitem{Col92}Colombeau, J.F.\emph{, Multiplication of Distributions;
A Tool in Mathematics, Numerical Engineering and Theoretical Physics},
Lecture Notes in Mathematics, 1532, Springer-Verlag, Berlin, 1992.

\bibitem{Col07}Colombeau, J.F., Mathematical problems on generalized
functions and the canonical Hamiltonian formalism, 2007 see https://arxiv.org/abs/0708.3425

\bibitem{loic}Cresson, J., (editor) \emph{Fractional calculus in
Analysis, dynamics and optimal control}, Mathematics research developments,
Nova Publishers, New York, 2014.

\bibitem{CKOPW}Cs�rnyei, M., Kirchheim, B., O\textquoteright Neil,
T.C., Preiss, D., Winter, S., Universal Singular Sets in the Calculus
of Variations, Arch. Rational Mech. Anal. 190 (2008) 371\textendash 424.

\bibitem{Dav88}Davie, A.M., Singular minimisers in the calculus of
variations in one dimension. Arch. Rational Mech. Anal. 101(2), 161\textendash 177,
1988.

\bibitem{Dir26}Dirac, P.A.M., \emph{The physical interpretation of
the quantum dynamics}, Proc. R. Soc. Lond. A, 113, 1926-27, 621-641.

\bibitem{CD:Djukic:1972}Djuki\'{c}, D.S.\textit{, }\textit{\emph{Noether's
theorem for optimum control systems}}, Internat. J. Control 1, no
18, 667\textendash 672, 1973.

\bibitem{ErlGross}Erlacher, E., Grosser, M., Ordinary Differential
Equations in Algebras of Generalized Functions, in Pseudo-Differential
Operators, Generalized Functions and Asymptotics, S. Molahajloo, S.
Pilipovi\'{c}, J. Toft, M. W. Wong eds, Operator Theory: Advances
and Applications Volume 231, 2013, 253\textendash 270.

\bibitem{MR2351636}Frederico, G.S.F., Torres, D.F.M., \textit{\emph{Conservation
laws for invariant functionals containing compositions}}\emph{,} Appl.
Anal. 86, no.~9, 1117\textendash 1126, 2007. \texttt{arXiv:0704.0949}

\bibitem{MR2405377}Frederico, G.S.F., Torres, D.F.M., \textit{\emph{Non-conservative
Noether's theorem for fractional action-like variational problems
with intrinsic and observer times}}, Int. J. Ecol. Econ. Stat. 9,
no.~F07, 74\textendash 82, 2007. \texttt{arXiv:0711.0645}

\bibitem{MR2351637}Frederico, G.S.F., Torres, D.F.M., \textit{\emph{A
non-differentiable quantum variational embedding in presence of time
delays}}, Int. J. Difference Equ. 8, no. 1, 49\textendash 62, 2013.

\bibitem{Torres:proper}Frederico, G.S.F., Odzijewicz, T., Torres,
D.F.M., \textit{\emph{Noether's theorem for nonsmooth extremals of
variational problems with time delay}}, Applicable Analysis 93, no.
1, 153\textendash 170, 2014.

\bibitem{GeFo00}Gelfand, I.M., Fomin, S.V., Calculus of variations,
Dover Publications, 2000.

\bibitem{GiKu16}Giordano, P., Kunzinger, M., Inverse Function Theorems
for Generalized Smooth Functions. Invited paper for the Special issue
ISAAC - Dedicated to Prof. Stevan Pilipovic for his 65 birthday. Eds.
M. Oberguggenberger, J. Toft, J. Vindas and P. Wahlberg, Springer
series Operator Theory: Advances and Applications, Birkhaeuser Basel,
2016.

\bibitem{GiKu18}Giordano, P., Kunzinger, M., A convenient notion
of compact sets for generalized functions. Proceedings of the Edinburgh
Mathematical Society, Volume 61, Issue 1, February 2018, pp. 57-92.

\bibitem{GiKuVe15}Giordano, P., Kunzinger, M., Vernaeve, H., Strongly
internal sets and generalized smooth functions. Journal of Mathematical
Analysis and Applications, volume 422, issue 1, 2015, pp. 56-71.

\bibitem{GIO1}Giordano, P., Kunzinger, M., Vernaeve, H., A Grothendieck
topos of generalized functions I: basic theory. Preprint. See: \url{https://www.mat.univie.ac.at/~giordap7/ToposI.pdf}

\bibitem{GiLu16}Giordano, P., Luperi Baglini, L., Asymptotic gauges:
Generalization of Colombeau type algebras. Math. Nachr. Volume 289,
Issue 2-3, pages 247\textendash 274, 2016.

\bibitem{Giu03}Giunashvili, Z., Bott Connection and Generalized Functions
on Poisson Manifold, 2003. See \url{https://arxiv.org/abs/math/0301364}

\bibitem{Gogodze88}Gogodze, I.K., \textit{Symmetry in problems of
optimal control} (in Russian). \emph{Proc. of extended sessions of
seminar of the Vekua Institute of Applied Mathematics}, Tbilisi University,
Tbilisi, 3, no 3, 39\textendash 42, 1988.

\bibitem{GrPr11}Gratwick, R., Preiss, D., A One-Dimensional Variational
Problem with Continuous Lagrangian and Singular Minimizer, Arch. Rational
Mech. Anal. 202 (2011) 177\textendash 211.

\bibitem{Gra30}Graves, L.M., Discontinuous solutions in the calculus
of variations. Bull. Amer. Math. Soc. 36, 831\textendash 846, 1930.

\bibitem{GKOS}Grosser, M., Kunzinger, M., Oberguggenberger, M., Steinbauer,
R., \emph{Geometric theory of generalized functions}, Kluwer, Dordrecht
(2001).

\bibitem{PU0}Hawking, S.W., Hertog, T., \textit{\emph{Living with
ghosts}}, Phys. Rev. D 65, 103515, 2002.

\bibitem{Hes66}Hestenes, M.R., \emph{Calculus of variations and optimal
control theory}, John Wiley \& Sons,1966.

\bibitem{Jost:book}Jost, J., Li-Jost, X., \textit{Calculus of variations},
Cambridge Studies in Advanced Mathematics, 64, Cambridge Univ. Press,
Cambridge, 1998.

\bibitem{Kat-Tal12}Katz, M.G., Tall, D., A Cauchy-Dirac delta function.
\emph{Foundations of Science}, 2012. See \url{http://arxiv.org/abs/1206.0119}

\bibitem{KKO:08}Konjik, S., Kunzinger, M., Oberguggenberger, M.:
Foundations of the Calculus of Variations in Generalized Function
Algebras. Acta Applicandae Mathematicae 103 n.~2, 169\textendash 199
(2008).

\bibitem{Kuh91}Kuhn, S., The derivative � la Carath�odory, The American
Mathematical Monthly, Vol. 98, No. 1 (Jan., 1991), pp. 40-44.

\bibitem{KuObStVi}Kunzinger, M., Oberguggenberger, M., Steinbauer,
R., Vickers, J.A., Generalized Flows and Singular ODEs on Differentiable
Manifolds, Acta Applicandae Mathematicae, 80(2), pp 221-241, 2004.

\bibitem{Kunzle96}K�nzle, A.F., Singular Hamiltonian systems and
symplectic capacities, Singularities and Differential Equations Banach
Center Publications, 33, pp.171-187, 1996. See \url{http://pldml.icm.edu.pl/pldml/element/bwmeta1.element.bwnjournal-article-bcpv33z1p171bwm}

\bibitem{Vasudevan}Lakshminarayanan, V., Ghatak, A.K. , Thygarajan,
K.\textit{, Lagrangian Optics, }Springer Science Business Media, LLC,
2002.

\bibitem{Lau89}Laugwitz, D., Definite values of infinite sums: aspects
of the foundations of infinitesimal analysis around 1820, Arch. Hist.
Exact Sci. 39 (3): 195-245, 1989.

\bibitem{LazoCesar}Lazo, M.J., Krumreich, C.E., \textit{\emph{The
action principle for dissipative systems}}, Journal of Mathematical
Physics 55, 122902, 2014.

\bibitem{GIO}Lecke, A., Luperi Baglini, L., Giordano, P., \textit{\emph{The
classical theory of calculus of variations for generalized functions}},
Advances in Nonlinear Analysis, 779\textendash 808, 2019.

\bibitem{LeMoSj91}Lerman, E., Montgomery, R., Sjamaar, R., Examples
of singular reduction. Symplectic geometry, London Math. Soc. Lecture
Note Ser., 192, Cambridge Univ. Press, Cambridge, 1993.

\bibitem{LiHuZh14}Li, F., Hua, Q., Zhang, S., New periodic solutions
of singular Hamiltonian systems with fixed energies, J Inequal Appl
2014: 400. See \url{https://doi.org/10.1186/1029-242X-2014-400}

\bibitem{Lim89}Lim, C.C., On singular Hamiltonians: the existence
of quasi-periodic solutions and nonlinear stability, Bull. Amer. Math.
Soc. (N.S.) 20(1) 35-40, 1989.

\bibitem{Loj}\L ojasiewicz, S., Sur la valeur et la limite d'une
distribution en un point, Studia Math. \textbf{16} (1957), 1-36.

\bibitem{LuGi17}Luperi Baglini, L., Giordano, P., The category of
Colombeau algebras. Monatshefte f�r Mathematik. 2017.

\bibitem{baglinipaolo}Luperi Baglini, L., Giordano\emph{, }\textit{\emph{P.,
}}A Grothendieck topos of generalized functions II:\textit{\emph{
ODE}}. See \url{https:/https://www.mat.univie.ac.at/giordap7/ToposII.pdf/www.mat.univie.ac.at/giordap7/ToposII.pdf}

\bibitem{PU1}Mannheim, P.D., Davidson, A., \textit{\emph{Dirac quantization
of the Pais-Uhlenbeck fourth order oscillator}}, Phys. Rev. A 71,
042110, 2005.

\bibitem{Mar68}Marsden, J.E., Generalized Hamiltonian mechanics,
Arch. Rat. Mech. Anal., 28(4), pp. 323- 361, 1968.

\bibitem{Mar68b}Marsden, J.E., Hamiltonian one parameter groups.
A mathematical exposition of infinite dimensional Hamiltonian systems
with applications in classical and quantum mechanics, Arch. Rat. Mech.
Anal., 28(5), pp. 362- 396, 1968.

\bibitem{Mar69}Marsden, J.E., Non-smooth geodesic flows and classical
mechanics, Canad. Math. Bull., 12, pp. 209-212, 1969.

\bibitem{MaHoAh12}Mazaheri, H., Hosseinzadeh, A., Ahmadian, M.T.
, Nonlinear oscillation analysis of a pendulum wrapping on a cylinder.
Scientia Iranica Transactions B: Mechanical Engineering, 19 (2), pp.
335\textendash 340, 2012.

\bibitem{MoSa16}Mordukhovich, B.S., Sarabi, M.E., Generalized differentiation
of piecewise linear functions in second-order variational analysis,
Nonlinear Analysis, Theory, Methods and Applications, Volume 132,
1 February 2016, Pages 240-273.

\bibitem{MTAG}Mukhammadiev, A., Tiwari, D., Apaaboah, G., Giordano,
P., Supremum, Infimum and hyperlimits of Colombeau generalized numbers.
Article in preparation, 2020. See \url{http: //www.mat.univie.ac.at/~giordap7/Hyperlim.pdf}.

\bibitem{ObVe08}Oberguggenberger, M., Vernaeve, H., Internal sets
and internal functions in Colombeau theory, \emph{J.~Math.~Anal.~Appl.}~341
(2008) 649\textendash 659.

\bibitem{MO01}Oberguggenberger, M., Generalized functions in nonlinear
models - A survey, Nonlinear Analysis, Theory, Methods and Applications,
Volume 47, Issue 8, August 2001, Pages 5029-5040.

\bibitem{PU}Pais, A., Uhlenbeck, G.E., \textit{\emph{On Field Theories
with Non-Localized Action}}, Phys. Rev. 79, 145\textendash 165, 1950.

\bibitem{Par79}Parker, P.E., Distributional geometry, Journal of
Mathematical Physics 20, 1423, 1979.

\bibitem{CD:MR29:3316b}Pontryagin, L.S., Boltyanskii, V.G., Gamkrelidze,
R.V., Mishchenko, E.F., \textit{Selected works. Vol. 4}. The mathematical
theory of optimal processes. Translated from the Russian by K. N.
Trirogoff, Translation edited by L. W. Neustadt, Reprint of the 1962
English translation, Gordon \& Breach, New York, 1986.

\bibitem{RaRKTu07}Rapoport, A., Rom-Kedar, V., Turaev, D., Approximating
Multi-Dimensional Hamiltonian Flows by Billiards. Commun. Math. Phys.
272, 567\textendash 600 (2007).

\bibitem{Rob73}Robinson, A., Function theory on some nonarchimedean
fields, Amer. Math. Monthly \textbf{80} (6) 87\textendash 109; Part
II: Papers in the Foundations of Mathematics (1973).

\bibitem{Sag68}Sage, A.P., \emph{Optimum Systems Control,} Englewood
Cliffs, N.J., Prentice-Hall, Inc., 1968.

\bibitem{Shv05}Shvartsman, I.A., Finite-dimensional approximations
in the derivation of necessary optimality conditions in nonsmooth
constrained optimal control, Nonlinear Analysis, Theory, Methods and
Applications, Volume 63, Issue 5-7, 30 November 2005, Pages e1665-e1672.

\bibitem{Syc11}Sychev, M.A., Another Theorem of Classical Solvability
\textquoteleft In Small\textquoteright{} for One-Dimensional Variational
Problems, Arch. Rational Mech. Anal. 202 (2011) 269\textendash 294.

\bibitem{Sto09}Stojanovi\'{c}, M., Extension of Colombeau algebra
to derivatives of arbitrary order D\textgreek{a}, \textgreek{a} \ensuremath{\in}
R+ \ensuremath{\bigcup} \{0\}. Application to ODEs and PDEs with entire
and fractional derivatives. Nonlinear Analysis, Theory, Methods and
Applications, Volume 71, Issue 11, 1 December 2009, Pages 5458-5475.

\bibitem{Tan93}Tanaka, K., A Prescribed Energy Problem for a Singular
Hamiltonian System with a Weak Force, Journal of Functional Analysis
113(2), pp 351-390, 1993.

\bibitem{Tro96}Troutman, J.L., (1996) The Lemmas of Lagrange and
du Bois-Reymond. In: \emph{Variational Calculus and Optimal Control}.
Undergraduate Texts in Mathematics. Springer, New York, NY.

\bibitem{Tuc93}Tuckey, C., Nonstandard methods in the calculus of
variations, Pitman Research Notes in Mathematics Series 297. Longman
Scientific \& Technical, Harlow, 1993.

\bibitem{TuRK98}Turaev, D., Rom-Kedar, V. (1998) On smooth Hamiltonian
flows limited to ergodic billiards. In: Benkadda S., Zaslavsky G.M.
(eds) Chaos, Kinetics and Nonlinear Dynamics in Fluids and Plasmas.
Lecture Notes in Physics, vol 511. Springer, Berlin, Heidelberg.

\bibitem{Vic12}Vickers, J.A., Distributional geometry in general
relativity, Journal of Geometry and Physics 62 (2012) 692\textendash 705.

\bibitem{vonN}von Neumann, J., Method in the Physical Sciences. Collected
Works Vol. VI. Theory of Games, Astrophysics, Hydro-dynamics and Meteorology,
A.H. Taub (ed.), Pergamon Press, Oxford, 1961.

\bibitem{NDSolve}See \url{https://reference.wolfram.com/language/ref/NDSolve.html},
Wolfram Research, Inc., Mathematica, Champaign, IL.

\bibitem{Hdelta}See \url{https://mathworld.wolfram.com/HeavisideStepFunction.html}
and \url{https://mathworld.wolfram.com/DeltaFunction.html}.
\end{thebibliography}
\end{document}